\documentclass{amsart}
\usepackage{amssymb}
\usepackage{braket}
\usepackage{mathrsfs}
\usepackage[all]{xy}
\usepackage{graphicx}
\usepackage{fontspec}
\usepackage{stmaryrd}
\usepackage{zxjatype}
\usepackage[ipaex]{zxjafont}

\usepackage{color}
\usepackage[colorinlistoftodos]{todonotes}

\usepackage{tikz}
\usepackage{tikz-cd}

\newtheorem{thm}{Theorem}[section]
\newtheorem{cor}[thm]{Corollary}
\newtheorem{prop}[thm]{Proposition}
\theoremstyle{definition}
\newtheorem{dfn}[thm]{Definition}

\newtheorem{claim}[thm]{Claim}
\newtheorem{lem}[thm]{Lemma}
\newtheorem{fact}[thm]{Fact}
\newtheorem{caution}[thm]{Caution}

\newtheorem{cond}[thm]{Condition}

\newtheorem{nota}[thm]{Notation}

\theoremstyle{remark}
\newtheorem{rem}[thm]{Remark}
\newtheorem*{introthmA}{Theorem A}
\newtheorem*{introthmB}{Theorem B}

\newcommand{\bbk}{\Bbbk}
\newcommand\blank{\,\operatorname{-}\,}

\newcommand\perdg{\operatorname{per}_{\mathit{dg}}}
\newcommand\per{\operatorname{per}}

\newcommand\up{^{\Cdot}}
\newcommand\upf{^{\, \Cdot}}
\newcommand\upp{^{\prime\Cdot}}
\newcommand\uppf{^{\,\prime\Cdot}}
\newcommand\uppp{^{\prime\prime\Cdot}}

\newcommand\Cdot{\raisebox{0.2ex}{\scalebox{0.5}{$\bullet$}}}

\newcommand\ds{\oplus}

\newcommand\dbr[1]{\llbracket#1\rrbracket}

\newcommand{\A}{\scA}
\newcommand{\B}{\scB}

\newcommand{\E}{\bbE}
\newcommand{\F}{\bbF}
\newcommand{\Z}{\bbZ}

\newcommand{\bbE}{\mathbb{E}}
\newcommand{\bbF}{\mathbb{F}}

\newcommand{\bbZ}{\mathbb{Z}}

\newcommand{\Hom}{\mathrm{Hom}}

\newcommand{\Mod}{\mathrm{Mod}}

\newcommand{\ob}{\mathrm{ob}}
\newcommand{\Ab}{\mathit{Ab}}
\newcommand{\Sets}{\mathit{Set}}

\newcommand{\op}{^\mathrm{op}}
\newcommand{\ush}{^\sharp}
\newcommand{\ssh}{_\sharp}

\newcommand{\id}{\mathrm{id}}
\newcommand{\tw}{\mathrm{tw}}

\newcommand{\Ker}{\mathrm{Ker}}
\DeclareMathOperator{\Cone}{\mathrm{Cone}}
\DeclareMathOperator{\CoCone}{\mathrm{CoCone}}

\newcommand{\co}{\colon}
\newcommand{\ci}{\circ}

\newcommand{\uas}{^{\ast}}
\newcommand{\sas}{_{\ast}}

\newcommand{\Ima}{\mathrm{Im}}

\newcommand{\sfr}{\mathfrak{s}}

\newcommand{\scA}{\mathscr{A}}
\newcommand{\scB}{\mathscr{B}}

\newcommand{\scE}{\mathscr{E}}
\newcommand{\scF}{\mathscr{F}}

\newcommand{\calA}{\mathcal{A}}
\newcommand{\calB}{\mathcal{B}}

\newcommand{\calH}{\mathcal{H}}

\newcommand{\calN}{\mathcal{N}}

\newcommand{\calS}{\mathcal{S}}

\newcommand{\Cbf}{\mathrm{tw}}
\newcommand{\Cbff}{\mathbf{C}}

\newcommand{\Kbf}{\mathrm{htw}}
\newcommand{\Kbff}{\mathbf{K}}

\newcommand{\Xbf}{\mathbf{X}}
\newcommand{\Ybf}{\mathbf{Y}}

\newcommand{\ovl}{\overline}

\newcommand{\bsm}{\begin{smallmatrix}}
\newcommand{\esm}{\end{smallmatrix}}

\newcommand{\abf}{\mathbf{a}}
\newcommand{\cbf}{\mathbf{c}}
\newcommand{\sbf}{\mathbf{s}}

\newcommand{\al}{\alpha}
\newcommand{\be}{\beta}
\newcommand{\gam}{\gamma}

\newcommand{\om}{\omega}
\newcommand{\thh}{\theta}
\newcommand{\del}{\delta}

\newcommand{\sig}{\sigma}

\newcommand{\dtri}[2]{\langle #1,#2\rangle}

\newcommand{\Ar}[3]{\ar[from=#1,to=#2,#3]}
\newcommand{\scT}{\mathscr{T}}

\numberwithin{equation}{section}

\begin{document}

\title{Higher exact dg-categories}

\author[Mochizuki]{Nao Mochizuki}
\address{Graduate School of Mathematics, Nagoya University, Furocho, Chikusaku, Nagoya 464-8602, Japan}
\email{mochizuki.nao.n8@s.mail.nagoya-u.ac.jp} 
    
\author{Hiroyuki Nakaoka}
\address{Department of Mathematical Sciences, University of the Ryukyus 1 Senbaru, Nishihara, Okinawa 903-0213 Japan}

\begin{abstract}
We introduce the notion of an $n$-exact dg-category.
This notion provides a higher analogue of Chen's exact dg-category, in the sense that the case where $n$ equals~1 recovers exact dg-categories.

We prove that, under a suitable vanishing condition on the cohomologies of $\Hom$-complexes of an $n$-exact dg-category $\A$,
its homotopy category admits a natural $n$-exangulated structure.
Thus $n$-exact dg-categories provide dg-enhancements of
$n$-exangulated categories.
At the same time, our framework can be regarded as a dg-categorical generalization of $n$-exangulated categories applicable even without the vanishing condition.

In the latter part of the article, we show that an $n$-cluster tilting subcategory of an exact dg-category naturally carries the structure of an $n$-exact dg-category.
This result indicates that $n$-exact dg-structures provide an intrinsic dg-categorical axiomatization of $n$-cluster tilting subcategories, highlighting the advantages of studying dg-generalizations of $n$-exangulated categories.
\end{abstract}

\maketitle

\tableofcontents

\section{Introduction}

The notion of an $n$-exangulated category was introduced in \cite{HLN1}
as a common generalization of $(n+2)$-angulated categories, $n$-exact categories, and $n$-abelian categories.
Here, $(n+2)$-angulated categories and $n$-exact categories were
introduced in \cite{GKO} and \cite{J}, respectively, in order to formalize $n$-cluster tilting subcategories inside triangulated and abelian categories, which play a central role in higher dimensional Auslander--Reiten theory~\cite{I}.
When $n=1$, the notions of $1$-exangulated category, $3$-angulated category, $1$-exact category, and $1$-abelian category coincide, respectively, with extriangulated categories~\cite{NP1}, exact categories, and abelian categories.

For the case $n=1$, namely extriangulated categories, it is remarkable that, besides unifying exact and triangulated categories, they admit various natural operations such as taking extension-closed subcategories, relative theories, and localization~\cite{NOS}.
Moreover, extriangulated categories possess natural enhancement notions:
Barwick's exact $\infty$-categories~\cite{Ba} and Chen's exact dg-categories~\cite{C1} provide $\infty$-categorical and dg enhancements, respectively.
(The relation with exact dg-categories is established in~\cite{C1}, while the relation with exact $\infty$-categories can be found in \cite{NP2} or in the embedding into stable $\infty$-categories obtained in~\cite{Kle}.)
The existence of such natural enhancement notions indicates that the definition of extriangulated categories captures a mathematically natural structure.

The framework of $n$-exangulated categories not only likewise includes $(n+2)$-angulated categories and $n$-exact categories (see also Remark~\ref{rem:comparison_with_n-exact}),
but also it is known that several operations analogous to those in the case $n=1$ are available, such as taking relative theories~\cite{HLN1}, and $n$-extension-closed subcategories~\cite{Kla}.
Furthermore, it was shown in \cite{HLN2} that an $n$-cluster tilting subcategory of an extriangulated category naturally carries an $n$-exangulated structure under some conditions, which is an $n$-exangulated categorical unification of the results by \cite{GKO} and \cite{J}.

In this article, we propose a dg-enhancement notion for
$n$-exangulated categories.
More precisely, we consider $n$-exact sequences in a connective additive dg-category $\A$, and use them to formulate the notion of an \emph{$n$-exact dg-category}.
When $n=1$, our definition recovers Chen's exact dg-category, while when $\A$ is an ordinary category, it coincides with Jasso's notion of an $n$-exact category.

Our main result (Theorem~\ref{thm:H0_of_n-ex-dg}) shows that if $\A$ satisfies the vanishing condition
$H^i(\A)=0$ for all $-n<i<0$ for the cohomologies of $\Hom$-complexes, then the additive category $H^0(\A)$ admits a natural $n$-exangulated structure.
\begin{introthmA}($=$ Theorem~\ref{thm:H0_of_n-ex-dg})
Let $(\A,\calS)$ be an $n$-exact dg-category, with skeletally small $H^0(\A)$. Suppose that $\A$ satisfies $H^i\A(A,B)=0$ for all $A,B\in\ob(\A)$ and for all $-n<i<0$. Then, $H^0(\A)$ has a structure of an $n$-exangulated category induced by $\calS$.
\end{introthmA}
In other words, this theorem asserts that $n$-exact dg-categories satisfying an appropriate cohomological vanishing condition provide dg-enhancements of $n$-exangulated categories.
From this viewpoint, the result can be regarded as supporting the naturalness of the definition of $n$-exangulated categories.
In the other direction, our notion of an $n$-exact dg-category can be interpreted as a dg-categorical generalization of $n$-exangulated categories
to situations where the above vanishing condition is not imposed. Thus, by working with dg-categories rather than ordinary additive categories, we propose $n$-exact dg-categories as a framework capable of treating such more general situations.

At the end of the article, we establish another main result (Theorem~\ref{thm:n-CT_is_n-exact}), showing that an $n$-cluster tilting subcategory of an exact dg-category naturally carries the structure of an $n$-exact dg-category.
This theorem can be viewed as a dg-categorical analogue of the preceding result for ordinary $n$-exangulated categories. We remark that in previous works, in order to equip an $n$-cluster tilting subcategory $\mathcal{T}$ of an extriangulated category $\mathcal{C}$ with an $n$-exangulated structure, additional assumptions were required:
in~\cite{GKO}, when $\mathcal{C}$ is triangulated, $\mathcal{T}$ was assumed to be closed under $n$-shifts, while in~\cite{HLN2} a corresponding technical assumption was imposed,
namely that $\mathcal{T}$ is nicely embedded in the extriangulated category $\mathcal{C}$.
\begin{introthmB}($=$ Theorem~\ref{thm:n-CT_is_n-exact})
Let $\scE$ be an exact dg-category, and let $\scT\subset\scE$ be an $n$-cluster tilting subcategory.
Let $\calS$ be the class of $n$-exact sequences in $\scT$ whose left and right end differentials are, respectively, an inflation and a deflation in $\scE$.
Then $(\scT,\calS)$ becomes an $n$-exact dg-category. 
\end{introthmB}
This result shows that our notion of an $n$-exact dg-category provides an intrinsic and dg-categorical axiomatization of the concept of an $n$-cluster tilting subcategory. 
For $H^0(\scT)$ to become an $n$-exangulated category, the cohomological vanishing condition required in Theorem~\ref{thm:H0_of_n-ex-dg} is equivalent,
to the condition that the subcategory $\scT$ is closed under the $n$-shift, or equivalently that it is an $n\mathbb{Z}$-cluster tilting subcategory.

The above theorem suggests that even in the absence of such conditions, a natural structure still exists at the dg-level. In general situations of this kind, the homotopy category $H^0(\scT)$ need not admit a structure describable in terms of ordinary $n$-exangulated categories. Hence the structure captured by our approach appears to be accessible only through the dg-categorical framework developed in this article.

\section{Preliminaries}

Throughout this article, let $\bbk$ be a commutative ring, and let $\A$ denote a dg-category over $\bbk$.
Also, let $n$ be a positive integer. We are mainly interested in the case of $n\ge2$.

\begin{dfn}\label{dfn:connective}
$\A$ is said to be \emph{connective} if positive cohomologies of the $\Hom$-complexes vanish; that is, if $H^i\A(A,B)=0$ holds for any $A,B\in\ob(\A)$ and all $i>0$. If $\A$ satisfies $\A(A,B)^i=0$ for all $A,B\in\ob(\A)$ and $i>0$, then $\A$ is said to be \emph{strictly connective}.
\end{dfn}

\begin{dfn}\label{dfn:additive}
$\A$ is said to be \emph{additive} if $H^0(\A)$ is an additive category. If $Z^0(\A)$ is an additive category, then $\A$ is said to be \emph{strictly additive}.
\end{dfn}

In our main theorem (Theorem~\ref{thm:H0_of_n-ex-dg}), we will show that $H^0(\A)$ of any $n$-exact dg-category $\A$ can be endowed with a natural $n$-exangulated structure under some condition (Condition~\ref{condition:H^i=0}).
We note that if $\A$ is connective and additive, then there exists a strictly connective and strictly additive dg-category that is quasi-equivalent to $\A$ (see \cite[Section~1]{G}).
We also remark that Bennett--Tennenhaus and Shah have shown in \cite{B-TS} that $n$-exangulated structures can be transported along equivalences of categories.
Thus, in view of \cite[Theorem 3.9]{B-TS}, we may replace the dg-category $\A$ with a quasi-equivalent dg-category for the purpose of proving our main theorem.
For this reason, we henceforth assume that $\A$ is strictly connective and strictly additive.

Under this assumption, we can make use of the following concrete description of the category of one-sided twisted complexes. 
The definition of a one-sided twisted complex in a dg-category goes back to \cite{BK}. Concerning the sign conventions, although we work with bounded ones, we follow those used in the unbounded case in \cite{G}.

In what follows, we restrict the definitions in \cite[Section 2]{G} to the bounded setting, with notation adjusted to suit our purposes. The corresponding definition will be cited explicitly in each case.
\begin{dfn}\label{dfn:complexes}
(\cite[Definition~2.1]{G}) Let $\A$ be as above.
\begin{enumerate}
\item A \emph{one-sided twisted complex} $X\up=((X^i)_{i\in\bbZ},(d_X^{i,j})_{i,j\in\bbZ})$ in $\A$ consists of the following data.
\begin{enumerate}
\item A sequence of objects $(X^i)_{i\in\bbZ}$ in $\A$. In this article we always assume it is bounded, meaning that $X^i=0$ except for finitely many $i\in\bbZ$.
\item A family of morphisms $(d_X^{i,j})_{i,j\in\bbZ}$ with $d_X^{i,j}\in\A(X^i,X^j)^{i-j+1}$ that satisfies $d_X^{i,j}=0$ for all $i\ge j$ and
\begin{equation}\label{eq_d}
d_{\A}(d_X^{i,j})+(-1)^j\sum_{k\in\bbZ}d_X^{k,j}\ci d_X^{i,k}=0
\end{equation}
for all $i,j\in\bbZ$. Especially, each $d_X^{i,i+1}$ is a closed morphism of degree $0$ in $\A$.
\end{enumerate}
\item A morphism of degree $m$ between two one-sided twisted complexes $X\up$ and $Y\up$ in $\A$ is a family of morphisms $f\up=(f^{i,j})_{i,j\in\bbZ}$
with $f^{i,j}\in\A(X^i,Y^j)^{j-i+m}$ for all $i,j\in\bbZ$.
\item The composition of two morphisms $f\up\colon X\up\to Y\up$ and $g\up\colon Y\up\to Z\up$ is defined by
\[
(g\up\ci f\up)^{i,j}=\sum_{k\in\bbZ}g^{k,j}\ci f^{i,k}.
\]
\item For any morphism $f\up\co X\up\to Y\up$ of degree $m$, its differential $df\up$ is given by
\[
(df)^{i,j}=(-1)^jd_{\A}(f^{i,j})+\sum_{k\in\bbZ}\Big(d_Y^{k,j}\circ f^{i,k}-(-1)^mf^{k,j}\circ d_X^{i,k}\Big)
\]
for all $i,j\in\bbZ$.

\end{enumerate}

These data form a dg-category, which we denote by $\Cbf(\A)$ in this article. The \emph{homotopy category} $\Kbf(\A)$ is defined to be $\Kbf(\A)=H^0(\Cbf(\A))$. Namely, the category $\Kbf(\A)$ has the same objects as $\Cbf(\A)$, and its morphisms are defined by
$\Kbf(\A)(X\up,Y\up)=H^0(\Cbf(\A)(X\up,Y\up))$.
\end{dfn}

\begin{rem}\label{rem:comparison_with_H3t}
For a morphism $f\up\colon X\up\to Y\up$ of degree $0$ in $\Cbf(\A)$, it is closed if and only if
\[
d_{\A}(f^{i,j})=(-1)^{j+1}\sum_{i\le k\le j}\left(d_Y^{k,j}\circ f^{i,k}-f^{k,j}\circ d_X^{i,k}\right)
\]
holds for all $i,j\in\mathbb{Z}$.

Especially when $X^i=Y^i=0$ holds except for $i\in\{0,1,2\}$, it is closed if and only if the equations
\begin{itemize}
\item $d_{\A}(f^{i,i})=0$ for $i=0,1,2$,
\item $d_{\A}(f^{i,i+1})=(-1)^i(d_Y^{i,i+1}\circ f^{i,i}-f^{i+1,i+1}\circ d_X^{i,i+1})$ for $i=0,1$,
\item $d_{\A}(f^{0,2})=-d_Y^{0,2}\circ f^{0,0}-d_Y^{1,2}\circ f^{0,1}+f^{1,2}\circ d_X^{0,1}+f^{2,2}\circ d_X^{0,2}$
\end{itemize}
hold.
\end{rem}

The dg-category $\Cbf(\A)$ is equipped with the shifts and cones defined as follows. This endows $\Kbf(\A)$ with the structure of a triangulated category.
\begin{dfn}\label{dfn:shifts}
(\cite[p.22]{G}) Let $\A$ be as above. For any $X\up\in\Cbf(\A)$ and $m\in\bbZ$, its shift $X\up[m]=X\upp\in\Cbf(\A)$ is defined by 
$X^{\prime i}=X^{i+m}$ and $d_{X'}^{i,j}=(-1)^m d_X^{i+m,j+m}$
for $i,j\in\bbZ$.
We note that $X\up[m+m']=(X\up [m])[m']$ hold for all $m,m'\in\bbZ$.

Also, in this article we let
\begin{itemize}
\item $s_{-,X}\up\in\Cbf(\A)(X\up[1],X\up)$ denote the closed morphism of degree $1$
\item $s_{+,X}\up\in\Cbf(\A)(X\up,X\up[1])$ denote the closed morphism of degree $-1$
\end{itemize}
which correspond to identities.
\end{dfn}

\begin{dfn}\label{dfn:cocone}\label{dfn:cone}(\cite[Lemma~2.13]{G})
Let $\A$ be as before. For any $f\up\in Z^0(\Cbf(\A))(X\up,Y\up)$, its  \emph{cone} $C_f\up=\Cone f\up\in \Cbf(\A)$ is defined by the following.
\begin{itemize}
\item $C_f^i=X^{i+1}\oplus Y^i$ for $i\in\mathbb{Z}$,
\item $d_{C_f}^{i,j}=\left[\begin{array}{cc}
-d_X^{i+1,j+1}&0\\
f^{i+1,j}& d_Y^{i,j}
\end{array}\right]\colon
X^{i+1}\oplus Y^i\to X^{j+1}\oplus Y^j$ for $i,j\in\mathbb{Z}$.
\end{itemize}
It is equipped with the following morphisms
\[
\begin{tikzcd}[column sep = 20]
Y\up
\arrow[r, shift left=0.6ex, "v_f\up"]
&
C_f\up
\arrow[l, shift left=0.6ex, "q_f\up"]
\arrow[r, shift left=0.6ex, "p_f\up"]
&
X\up [1]
\arrow[l, shift left=0.6ex,  "u_f\up"]
\end{tikzcd}
\]
of degree $0$ in $\Cbf(\A)$ defined by
\[
p_f^{i,j}=\begin{cases}
[1\ 0]&\text{if}\ i=j\\
0&\text{otherwise}
\end{cases}\ \ ,\quad
u_f^{i,j}=\begin{cases}
\scriptstyle{\left[\begin{array}{c}1\\0\end{array}\right]}&\text{if}\ i=j\\
0&\text{otherwise}
\end{cases},
\]
\[
q_f^{i,j}=\begin{cases}
[0\ 1]&\text{if}\ i=j\\
0&\text{otherwise}
\end{cases}\ \ ,\quad
v_f^{i,j}=\begin{cases}
\scriptstyle{\left[\begin{array}{c}0\\1\end{array}\right]}&\text{if}\ i=j\\
0&\text{otherwise}
\end{cases},
\]
that satisfy
\[
q_f\up\circ v_f\up=\id_{Y\up}, \ \ 
p_f\up\circ u_f\up=\id_{X\up [1]}, \ \ 
u_f\up\circ p_f\up+v_f\up\circ q_f\up=\id_{C_f\up}
\]
and
\[
dp_f\up=0, \ \ 
dv_f\up=0, \ \
dq_f\up=-f\up\ci s_{-,X}\up\ci p_f\up, \ \ 
du_f\up=v_f\up\ci f\up\ci s_{-,X}\up
\]
in $\Cbf(\A)$.
Similarly, cocone $\CoCone f\up\in\Cbf(\A)$ of $f\up\in Z^0(\Cbf(\A))(X\up,Y\up)$ exists. They satisfy $\CoCone f\up\cong(\Cone f\up)[-1]$ in $\Kbf(\A)$. 
\end{dfn}

\begin{dfn}\label{dfn:n+2}
Let $\A$ be as before.
For any $a,b\in\Z$, we define as follows.
\begin{enumerate}
\item $\Cbf_{\ge a}(\A)=\Cbf_{[a,\infty)}(\A)$ denotes the full dg-subcategory of $\Cbf(\A)$ consisting of $X\up$ satisfying $X^i=0$ for all $i<a$. 
\item $\Cbf_{\le b}(\A)=\Cbf_{(-\infty,b]}(\A)$ denotes the full dg-subcategory of $\Cbf(\A)$ consisting of $X\up$ satisfying $X^i=0$ for all $i>b$. 
\item $\Cbf_{[a,b]}(\A)=\Cbf_{\ge a}(\A)\cap\Cbf_{\le b}(\A)$. That is, $\Cbf_{[a,b]}(\A)$ is the full dg-subcategory of $\Cbf(\A)$ consisting of $X\up$ satisfying $X^i=0$ for all $i<a$ and $i>b$. 
\end{enumerate}
We denote the corresponding homotopy categories by $\Kbf_{\ge a}(\A)$, $\Kbf_{\le b}(\A)$ and $\Kbf_{[a,b]}(\A)$, respectively. 
These are full subcategories of $\Kbf(\A)$.

Especially, we write $\Cbf^{n+2}(\A)=\Cbf_{[0,n+1]}(\A)$ and call it the \emph{dg-category of $(n+2)$-term one-sided twisted complexes}, which will play a central role in this article. Its homotopy category will be denoted by $\Kbf^{n+2}(\A)=\Kbf_{[0,n+1]}(\A)$.

Also, $\Cbf_{[0,0]}(\A)$ can be identified with $\A$. 
Accordingly, for each integer $i\in\bbZ$, object $A$ and morphism $a$ in $\A$, we write $A[i]$ and $a[i]$ for their shifts in $\Cbf(\A)$ and also in $\Kbf(\A)$.
In particular, $A[i]$ is the one-sided twisted complex concentrated in degree $i$ with value $A$.
\end{dfn}

\normalcolor

The truncations of one-sided twisted complexes are described in Definition~2.7 of \cite{G} and the subsequent remarks.
\begin{dfn}\label{dfn:co-t-str}
Let $k\in\bbZ$ be any integer. We define as follows.
\begin{enumerate}
\item For any $X\up\in\ob(\Cbf(\A))$, define $\sigma_{\ge k}X\up\in\ob(\Cbf_{\ge k}(\A))$ and $\sigma_{\le k}X\up\in\ob(\Cbf_{\le k}(\A))$ by
\[
(\sigma_{\ge k}X)^i=
\begin{cases}
X^i & \text{if}\ i\ge k\\
0 & \text{otherwise}
\end{cases}\quad,\qquad 
(\sigma_{\le k}X)^i
=\begin{cases}
X^i & \text{if}\ i\le k\\
0 & \text{otherwise}
\end{cases}
\]
and
\[
d_{(\sigma_{\ge k}X)}^{i,j}=
\begin{cases}
d_X^{i,j} & \text{if}\ \min(i,j)\ge k\\
0 & \text{otherwise}
\end{cases}\quad,\qquad 
d_{(\sigma_{\le k}X)}^{i,j}=
\begin{cases}
d_X^{i,j} & \text{if}\ \max(i,j)\le k\\
0 & \text{otherwise}
\end{cases}\ .
\]

\item For any morphism $f\up\co X\up\to Y\up$ in $\Cbf(\A)$, define morphisms 
$\sigma_{\ge k}f\up\co \sig_{\ge k}X\up\to\sig_{\ge k}Y\up$ and $\sigma_{\le k}f\up\co \sig_{\le k}X\up\to\sig_{\le k}Y\up$ by
$$(\sigma_{\ge k}f)^{i,j}=
\begin{cases}
f^{i,j} & \text{if}\ \min\{i,j\}\ge k\\
0 & \text{otherwise}\ 
\end{cases}\quad,\qquad 
(\sigma_{\le k}f)^{i,j}
=\begin{cases}
f^{i,j} & \text{if}\ \max\{i,j\}\le k\\
0 & \text{otherwise}
\end{cases}\ .
$$
\end{enumerate}
For $k,k'\in\Z$ satisfying $k\le k'$, we also use the notation $\sig_{[k,k']}X\up=\sig_{\ge k}(\sig_{\le k'}X\up)=\sig_{\le k'}(\sig_{\ge k}X\up)$ for objects and $\sig_{[k,k']}f\up=\sig_{\ge k}(\sig_{\le k'}f\up)=\sig_{\le k'}(\sig_{\ge k}f\up)$ for morphisms.

These assignments do \emph{not} form dg-functors, since compositions are not preserved. However, they still form functors
\begin{equation}\label{truncations}
\sig_{\ge k}\co Z^0(\Cbf(\A))\to Z^0(\Cbf_{\ge k}(\A))
\quad\text{and}\quad
\sig_{\le k}\co Z^0(\Cbf(\A))\to Z^0(\Cbf_{\le k}(\A)).
\end{equation}
We will use the same symbols for functors obtained by restricting them. Similarly for $\sig_{[k,k']}$. Especially, we have functors 
\[
\sig_{\ge k}\co Z^0(\Cbf^{n+2}(\A))\to Z^0(\Cbf_{[k,n+1]}(\A))
\quad\text{and}\quad
\sig_{\le k}\co Z^0(\Cbf^{n+2}(\A))\to Z^0(\Cbf_{[0,k]}(\A)).
\]
\end{dfn}

\begin{rem}\label{rem:non-functorial1}
The truncations $(\ref{truncations})$ do not induce functors on $\Kbf(\A)$, since they do not respect homotopy relations.
\end{rem}

The following construction, which represents $X\up$ via a distinguished triangle, can be found in \cite[Proposition~2.14]{G}.
Since it will be used repeatedly in what follows, we introduce here the following notation. When $X\up$ or $k$ is clear from the context, we may omit the corresponding subscript in $\al_{X,k}\up$ to write $\al_X\up$, $\al_k\up$ or simply $\al\up$. Similarly for $\be_{X,k}\up$ and $\gam_{X,k}\up$.
\begin{prop}\label{prop:alphabetagamma}
Let $X\up\in\Cbf(\A)$ be any object. For any $k\in\bbZ$, we define $\al\up=\al_{X,k}\up\in Z^0(\Cbf(\A))((\sig_{\le k-1}X\up)[-1],\sig_{\ge k}X\up)$ by
\[
\al_{X,k}^{i,j}=\begin{cases}
d_X^{i-1,j}&\text{if}\ i\le k\ \text{and}\ j\ge k\\
0&\text{otherwise}
\end{cases}.
\]
The following holds.
\begin{enumerate}
\item $\Cone \al\up=X\up$, up to the identification via the canonical isomorphisms $X^i\oplus 0\cong X^i\cong 0\oplus X^i$.
\item Define $\be_{X,k}\up\in Z^0(\Cbf(\A))(\sig_{\ge k}X\up,X\up)$ and $\gam_{X,k}\in Z^0(\Cbf(\A))(X\up,\sig_{\le k-1}X\up)$ to be 
$\be_{X,k}\up=v_{\al}\up$ and $\gam_{X,k}\up=p_{\al}\up$ in Definition~\ref{dfn:cone}. Then, we obtain a distinguished triangle
\[
(\sig_{\le k-1}X\up)[-1]\xrightarrow{\al_{X,k}\up} \sig_{\ge k}X\up\xrightarrow{\be_{X,k}\up}X\up\xrightarrow{\gam_{X,k}\up}\sig_{\le k-1}X\up \]
in $\Kbf(\A)$.
\end{enumerate}
\end{prop}
\begin{proof}
See Lemma~2.13 and Proposition~2.14 in \cite{G}.
\end{proof}

\normalcolor

We note that for any morphism $f\up\co X\up\to Y\up$ of degree $0$ in $\Cbf(\A)$, the squares on the right and in the middle of the following diagram are commutative for any $k\in\bbZ$.
\begin{equation}\label{triangle_sigma}
\begin{tikzcd}
(\sig_{\le k-1}X\up)[-1] & \sig_{\ge k}X\up & X\up & \sig_{\le k-1}X\up \\
(\sig_{\le k-1}Y\up)[-1] & \sig_{\ge k}Y\up & Y\up & \sig_{\le k-1}Y\up 
\Ar{1-1}{1-2}{"\al_{X,k}\up"}
\Ar{1-2}{1-3}{"\be_{X,k}\up"}
\Ar{1-3}{1-4}{"\gam_{X,k}\up"}
\Ar{1-1}{2-1}{"(\sig_{\le k-1}f\up){[}-1{]}"'}
\Ar{1-2}{2-2}{"\sig_{\ge k}f\up"}
\Ar{1-3}{2-3}{"f\up"}
\Ar{1-4}{2-4}{"\sig_{\le k-1}f\up"}
\Ar{2-1}{2-2}{"\al_{Y,k}\up"'}
\Ar{2-2}{2-3}{"\be_{Y,k}\up"'}
\Ar{2-3}{2-4}{"\gam_{Y,k}\up"'}
\end{tikzcd}
\end{equation}
It is shown in \cite[Proposition~2.14]{G} that if $f\up$ is closed then the left-hand square is commutative in $\Kbf(\A)$, and $(\ref{triangle_sigma})$ gives a morphism of distinguished triangles in $\Kbf(\A)$. 
Since this will be used in subsequent sections, we record it here as Proposition~\ref{prop:give_a_morph}, making the correspondence explicit.
\begin{prop}\label{prop:give_a_morph}
Let $X\up,Y\up\in\Cbf(\A)$ be any pair of objects, and let $f\up\co X\up\to Y\up$ be any morphism of degree $0$ in $\Cbf(\A)$. For any $k\in\bbZ$, the following are equivalent.
\begin{enumerate}
\item $f\up$ is a closed morphism in $\Cbf(\A)$.
\item The morphism $\varphi\up\co (\sig_{\le k-1}X\up)[-1]\to\sig_{\ge k}Y\up$ of degree $-1$ in $\Cbf(\A)$ defined by
\begin{equation}\label{f_to_varphi}
\varphi^{i,j}=\begin{cases}
f^{i-1,j}&\text{if}\ i\le k\le j \\
0&\text{otherwise}
\end{cases}
\end{equation}
satisfies $d\varphi\up=\al_{Y,k}\up\ci ((\sig_{\le k-1}f\up) [-1])-(\sig_{\ge k}f\up)\ci\al_{X,k}\up$ in $\Cbf(\A)$.
\end{enumerate}
Moreover, whenever these equivalent conditions are satisfied, the diagram 
\[
\begin{tikzcd}
(\sig_{\le k-1}X\up)[-1] & \sig_{\ge k}X\up  \\
(\sig_{\le k-1}Y\up)[-1] & \sig_{\ge k}Y\up
\Ar{1-1}{1-2}{"\al_{X,k}\up"}
\Ar{1-1}{2-1}{"(\sig_{\le k-1}f\up){[}-1{]}"'}
\Ar{1-2}{2-2}{"\sig_{\ge k}f\up"}
\Ar{2-1}{2-2}{"\al_{Y,k}\up"'}
\end{tikzcd}
\]
becomes commutative in $\Kbf(\A)$, and $(\ref{triangle_sigma})$ gives a morphism of distinguished triangles in $\Kbf(\A)$. 
In particular, if $f\up$ is closed and if $\sig_{\le k-1}f\up$ and $\sig_{\ge k}f\up$ are isomorphisms in $\Kbf(\A)$, then $f\up$ is an isomorphism in $\Kbf(\A)$.
\end{prop}
\begin{proof}
The equivalence of {\rm (1)} and {\rm (2)} follows straightforwardly from the definition of the differential in $\Cbf(\A)$.
The second assertion is from \cite[Proposition~2.14]{G}.
The last assertion is immediate from the fact that $(\ref{triangle_sigma})$ is a morphism of distinguished triangles. 
\end{proof}

\begin{cor}\label{cor:give_a_morph}
Let $X\up,Y\up\in\Cbf(\A)$ be any pair of objects, and let $k\in\bbZ$ be any integer. Suppose that we are given morphisms $g\up\in Z^0(\Cbf(\A))(\sig_{\le k-1}X\up,\sig_{\le k-1}Y\up)$ and $e\up\in Z^0(\Cbf(\A))(\sig_{\ge k}X\up,\sig_{\ge k}Y\up)$.
Then, to give either of the following is equivalent.
\begin{enumerate}
\item A morphism $f\up\in Z^0(\Cbf(\A))(X\up,Y\up)$ that satisfies $\sig_{\le k-1}f\up=g\up$ and $\sig_{\ge k}f\up=e\up$ in $\Cbf(\A)$.
\item A morphism $\varphi\up\in\Cbf(\A)(X\up,Y\up)^{-1}$ that satisfies 
$d\varphi\up=\al_{Y,k}\up\ci (g\up [-1])-e\up\ci\al_{X,k}\up$
in $\Cbf(\A)$.
\end{enumerate}
Indeed, such $f\up$ and $\varphi\up$ correspond bijectively via the same correspondence as $(\ref{f_to_varphi})$. 
Moreover, if $g\up$ and $e\up$ are isomorphisms in $\Kbf(\A)$, then $f\up$ in {\rm (1)} also becomes an isomorphism in $\Kbf(\A)$.
\end{cor}
\begin{proof}
Let $g\up$ and $e\up$ be as in the statement.
Then, giving $f\up\in\Cbf(\A)(X\up,Y\up)^0$ satisfying $\sig_{\le k-1}f\up=g\up$, $\sig_{\ge k}f\up=e\up$ in $\Cbf(\A)$ is equivalent to giving $\varphi\up\in\Cbf(\A)(\sig_{\le k-1}X\up)[-1],\sig_{\ge k}Y\up)^{-1}$.
Indeed, given $\varphi\up$, we obtain $f\up$ by setting
\[
f^{i,j}=\begin{cases}
g^{i,j} & \text{if}\ i,j\le k-1\\
e^{i,j} & \text{if}\ i,j\ge k\\
\varphi^{i+1,j} & \text{if}\ i\le k-1\ \text{and}\ j\ge k\\
0 & \text{otherwise}
\end{cases}.
\]
Conversely, the inverse assignment is given in the same way as $(\ref{f_to_varphi})$, and the two constructions yield a bijection.
Proposition~\ref{prop:give_a_morph} shows that $f\up$ is closed if and only if $\varphi\up$ satisfies $d\varphi\up=\al_{Y,k}\up\ci ((\sig_{\le k-1}f\up) [-1])-(\sig_{\ge k}f\up)\ci\al_{X,k}\up$ in $\Cbf(\A)$. 
The last part is also from Proposition~\ref{prop:give_a_morph}.
\end{proof}

\normalcolor

\subsection{Functor $\overline{(\blank)}$}

Let $\A$ be as before.
Note that, since $\A$ is assumed to be strictly connective, 
any morphism of degree $0$ in $\A$ is closed and we have $\A(A,B)^0=Z^0\A(A,B)$ for all $A,B\in\ob(\A)$.
This allows us to define as follows.
\begin{dfn}\label{dfn:bar_functor1}
A dg-functor $\ovl{(\blank)}\co\A\to H^0(\A)$ is defined by the following. Here, $H^0(\A)$ is regarded as a dg-category whose morphisms are concentrated in degree $0$, equipped with trivial differentials.
\begin{enumerate}
\item The correspondence for objects is the identity.
\item Any morphism $f\in\A(A,B)^i$ of degree $i$ is sent to
$\ovl{f}$, which denotes the residue class $\ovl{f}\in H^0(\A)(A,B)$ of $f$ modulo $d_{\A}(\A(A,B)^{-1})$ if $i=0$, while we put $\ovl{f}=0$ if $i\ne 0$.
\end{enumerate}
\end{dfn}

Since $H^0(\A)$ regarded as a dg-category
satisfies the relevant assumptions,
we may apply the definitions and notation in the previous section. In particular, we have dg-category $\Cbf(H^0(\A))$ and categories $Z^0(\Cbf(H^0(\A)))$, $\Kbf(H^0(\A))$. We note that $\Cbf(H^0(\A))$ and $\Kbf(H^0(\A))$ can be naturally identified with the dg-category of bounded complexes $\Cbff(H^0(\A))$ in $H^0(\A)$ and the homotopy category $\Kbff(H^0(\A))$. Namely, we identify as
\[
\Cbf(H^0(\A))=\Cbff(H^0(\A))\quad\text{and}\quad \Kbf(H^0(\A))=\Kbff(H^0(\A)).
\]
The notation introduced in Definition~\ref{dfn:n+2} will also be used for $\Cbff(H^0(\A))$ and $\Kbff(H^0(\A))$, such as $\Cbff^{n+2}(H^0(\A))$ and $\Kbff^{n+2}(H^0(\A))$.

For objects and morphisms in $\Cbff(H^0(\A))$, we use the following, usual indexing conventions for complexes. Similarly for $\Kbff(H^0(\A))$ and other related categories.
\begin{itemize}
\item For an object $\Xbf\up\in \Cbff(H^0(\A))$, we denote its differential $d_{\Xbf}^{i,i+1}$ simply by $d_{\Xbf}^i \co \Xbf^i \to \Xbf^{i+1}$, since the other terms are $0$.
\item For a morphism $\mathbf{f}\upf\co\Xbf\up\to \Ybf\up$ of degree $m$ in $\Cbff(H^0(\A))$, 
we denote its $i$-th component $\mathbf{f}^{\, i,i+m}$ simply by $\mathbf{f}^{\, i}\co \Xbf^i \to \Ybf^{i+m}$, by the same reason.
\end{itemize}

One-sided twisted complexes in $\A$ and complexes in $H^0(\A)$ can be related by the following.
\begin{dfn}\label{dfn:bar_functor2}
We define $\ovl{(\blank)}\co\Cbf(\A)\to\Cbff(H^0(\A))$ to be the dg-functor induced by $\ovl{(\blank)}\co\A\to H^0(\A)$, and denote it by the same symbol $\ovl{(\blank)}$. More precisely, the correspondences are as follows, and we use the following notation. 
\begin{enumerate}
\item Any object $X\up\in\Cbf(\A)$ is sent to $\ovl{X}\up=(\ovl{X}^i,\ovl{d}_X^i)_{i\in\bbZ}\in \Cbff(H^0(\A))$, where we put $\ovl{X}^i=\ovl{X^i}=X^i$ and the differential is given by $\ovl{d}_X^i=\ovl{d_X^{i,i+1}}$ for each $i\in\bbZ$.
\item Any morphism $f\up\co X\up\to Y\up$ of degree $m$ in $\Cbf(\A)$ is sent to the morphism $\ovl{f}\up=(\ovl{f}^i)_{i\in\bbZ}\co \ovl{X}\up\to\ovl{Y}\up$ of degree $m$ in $\Cbff(H^0(\A))$, where we put $\ovl{f}^i=\ovl{f^{i,i+m}}$ for each $i\in\bbZ$.
\end{enumerate}
We use the same symbol $\ovl{(\blank)}$ for dg-functors obtained by restricting to the dg-subcategories introduced in Definition~\ref{dfn:n+2}. Especially, we have a dg-functor $\ovl{(\blank)}\co \Cbf^{n+2}(\A)\to\Cbff^{n+2}(H^0(\A))$. We note that these functors $\overline{(\blank)}$ commute strictly with shifts and truncations in Definition~\ref{dfn:co-t-str}.
\end{dfn}

\begin{dfn}\label{dfn:bar_functor3}
We define $\ovl{(\blank)}\co\Kbf(\A)\to\Kbff(H^0(\A))$ to be the triangle functor induced by the dg-functor $\ovl{(\blank)}\co\Cbf(\A)\to\Cbff(H^0(\A))$, and denote it by the same symbol $\ovl{(\blank)}$.

We use the same symbol $\ovl{(\blank)}$ for additive functors obtained for the homotopy categories of the dg-subcategories of $\Cbf(\A)$ introduced in Definition~\ref{dfn:n+2}. Especially, we have an additive functor $\ovl{(\blank)}\co \Kbf^{n+2}(\A)\to\Kbff^{n+2}(H^0(\A))$.
\end{dfn}

\begin{rem}\label{rem:non-functorial2}
In the above definition, though the map
\[
\Kbf^{n+2}(\A)(X\up,Y\up)\to \Kbff^{n+2}(H^0(\A))(X\up,Y\up)\ ;\ f\up\mapsto \ovl{f}\up
\]
is well-defined, the correspondence taking each of its terms
\[
\Kbf^{n+2}(\A)(X\up,Y\up)\to H^0\A(X^i,Y^i)\ ;\ f\up\mapsto \ovl{f}^i
\]
is \emph{not} a well-defined map either for $i=0,n+1$.
In particular, they do not give functors $\Kbf^{n+2}(\A)\to H^0(\A)$.
\end{rem}

The following is a Corollary of Proposition~\ref{prop:give_a_morph}.
\begin{cor}\label{cor:devide_f}
For any $f\up\in Z^0(\Cbf^{n+2}(\A))(X\up,Y\up)$, the following holds.
\begin{enumerate}
\item We have a morphism
\[
\begin{tikzcd}
X^0{[}-1{]}&\sig_{\ge 1}X\up & X\up & X^0 \\
Y^0{[}-1{]}&\sig_{\ge 1}Y\up & Y\up & Y^0 
\Ar{1-1}{1-2}{"\al_X\up"}
\Ar{1-2}{1-3}{"\be_X\up"}
\Ar{1-3}{1-4}{"\gam_X\up"}
\Ar{1-1}{2-1}{"(\sig_{\le 0}f\up){[}-1{]}"'}
\Ar{1-2}{2-2}{"\sig_{\ge 1}f\up"}
\Ar{1-3}{2-3}{"f\up"}
\Ar{1-4}{2-4}{"\sig_{\le 0}f\up"}
\Ar{2-1}{2-2}{"\al_Y\up"'}
\Ar{2-2}{2-3}{"\be_Y\up"'}
\Ar{2-3}{2-4}{"\gam_Y\up"'}
\end{tikzcd}
\]
of distinguished triangles in $\Kbf(\A)$.
\item If $\ovl{f}^0$ is an isomorphism in $H^0(\A)$, then there is an isomorphism $\CoCone(\sig_{\ge 1}f\up)\xrightarrow{\cong}\CoCone f\up$ in $\Kbf(\A)$. 
\end{enumerate}
\end{cor}
\begin{proof}
{\rm (1)} is from Proposition~\ref{prop:give_a_morph} applied to $k=1$.
{\rm (2)} is shown in a straightforward argument in a triangulated category, since $\ovl{f}^0$ is an isomorphism in $H^0(\A)$ if and only if $\sig_{\le 0}f\up\co X^0\to Y^0$ is an isomorphism in $\Kbf(\A)$.
\end{proof}

The following corollary will be used later in Section~\ref{section:n-exact-dg}.
\begin{cor}\label{cor:inflation_replace}
Let $X\up\in\Cbf_{\ge0}(\A)$ be any object. Suppose that a morphism $a\in Z^0\A(X^0,X^1)$ satisfies $\ovl{d}_X^0=\ovl{a}$ in $H^0(\A)$.
Then, there exist $X\upp\in\Cbf_{\ge0}(\A)$ and an isomorphism $f\up\co X\up\xrightarrow{\cong}X\upp$ in $\Cbf_{\ge0}(\A)$ that satisfy the following.
\begin{itemize}
\item $\sig_{\ge 1}X\up=\sig_{\ge 1}X\upp$ in $\Cbf(\A)$, and $X^0=X^{\prime 0}$, $d_{X'}^{0,1}=a$ in $\A$,
\item $\sig_{\ge 1}f\up=\id_{\sig_{\ge 1}X\up}$ in $\Cbf(\A)$, and $f^{0,0}=\id_{X^0}$ in $\A$.
\end{itemize}
\end{cor}
\begin{proof}
Let $\al\up_{X,1}\co X^0[-1]\to\sig_{\ge1}X\up$ be the morphism in Proposition~\ref{prop:alphabetagamma}.
By assumption, there exists $\phi\in\A(X^0,X^1)^{-1}$ such that $d\phi=d_X^{0,1}-a$ in $\A$.
Define $\varphi\up\in\Cbf(\A)(X^0[-1],\sig_{\ge1}X\up)^{-1}$ by
\[
\varphi^{i,j}=\begin{cases}
\phi & \text{if}\ (i,j)=(1,1)\\
0 & \text{otherwise}
\end{cases}
\]
and put $x\up=\al\up_{X,1}+d\varphi\up$. Then, up to the identification via the canonical isomorphisms $X^i\oplus 0\cong X^i\cong 0\oplus X^i$, $\Cone x\up$ agrees with the object $X\upp\in\Cbf_{\ge0}(\A)$ given by
\[
X^{\prime i}=X^i
\quad \text{and} \quad
d_{X'}^{i,j}=\begin{cases}
x^{i,j} & \text{if}\ i=0,\\
d_X^{i,j} & \text{otherwise}
\end{cases}
\]
for all $i,j\in\bbZ$.
In particular $\sig_{\ge 1}X\upp=\sig_{\ge 1}X\up$ holds. Then we have $\al\up_{X',1}=x\up$, and thus by Corollary~\ref{cor:give_a_morph}, we obtain a morphism $f\up\in Z^0(\Cbf(\A))(X\up,X\upp)$ given by
\[
f^{i,j}=\begin{cases}
\id & \text{if}\ i=j=0\ \text{or}\ \min(i,j)\ge1\\
\phi & \text{if}\ (i,j)=(0,1)\\
0 & \text{otherwise}
\end{cases}
\]
which corresponds to $\varphi\up$.
Similarly, we have a morphism $g\up\in Z^0(\Cbf(\A))(X\upp,X\up)$ given by
\[
g^{i,j}=\begin{cases}
\id & \text{if}\ i=j=0\ \text{or}\ \min(i,j)\ge1\\
-\phi & \text{if}\ (i,j)=(0,1)\\
0 & \text{otherwise}
\end{cases}\ \ .
\]
They satisfy $g\up\ci f\up=\id_{X\up}$ and $f\up\ci g\up=\id_{X\upp}$ strictly in $\Cbf(\A)$, hence $f\up$ is an isomorphism in $\Cbf(\A)$.
\end{proof}

\normalcolor

The following lemmas will be used later.
\begin{lem}\label{lem:4A}
Let $f\up\in Z^0(\Cbf(\A))(X\up,Y\up)$, $a\in Z^0(\A)(X^0,Y^0)$ and $c\in Z^0(\A)(X^{n+1},Y^{n+1})$ be any triplet of morphisms. If $f\up$ satisfies $\ovl{f}^0=\ovl{a}$ and $\ovl{f}^{n+1}=\ovl{c}$ in $H^0(\A)$, then there exists $f\upp\in Z^0(\Cbf(\A))(X\up,Y\up)$ that satisfies $f\up=f\upp$ in $\Kbf(\A)$ and $f^{0,0}=a$, $f^{n+1,n+1}=c$ in $\A$. 
\end{lem}
\begin{proof}
By assumption, there exist $\phi_0\in\A(X^0,Y^0)^{-1}$ and $\phi_{n+1}\in\A(X^{n+1},Y^{n+1})^{-1}$ such that $f^{0,0}-a=d_{\A}(\phi_0)$ and $f^{n+1,n+1}-c=d_{\A}(\phi_{n+1})$. If we define  $\xi\up\in\Cbf^{n+2}(\A)(X\up,Y\up)^{-1}$ by
\[
\xi^{i,j}=\begin{cases}
\phi_0&\text{if}\ (i,j)=(0,0)\\
(-1)^{n+1}\phi_{n+1}&\text{if}\ (i,j)=(n+1,n+1)\\
0&\text{otherwise}
\end{cases},
\]
then $f\upp=f\up-d\xi\up$ satisfies the required properties. 
\end{proof}

\begin{lem}\label{lem:comparison_H}
For any $X\up\in\Cbf_{\ge 0}(\A)$ and $B\in\ob(\A)$, the homomorphism of abelian groups 
\[
\ovl{(\blank)}\co\Kbf(\A)(X\up,B)\to \Kbff(H^0(\A))(\ovl{X}\up,B)
\]
induced by the functor $\ovl{(\blank)}\co\Kbf(\A)\to\Kbff(H^0(\A))$, is an isomorphism.
\end{lem}
\begin{proof}
By the definition of morphisms in $\Cbf(\A)$, we see that the maps
\[
v\co Z^0(\Cbf(\A))(X\up,B)\to Z^0(\A)(X^0,B)\ ;\ f\up\mapsto f^{0,0},
\]
\[
w\co\Cbf(\A)(X\up,B)^{-1}\to\A(X^0,B)^{-1}\oplus\A(X^1,B)^0\ ;\ \varphi\up\mapsto (\varphi^{0,0},\varphi^{1,0})
\]
are isomorphisms of abelian groups. 
Also, by the definition of $\Kbff(H^0(\A))$, we have an exact sequence
\[
H^0(\A)(X^1,B)\xrightarrow{-\ci \ovl{d}_X^0}H^0(\A)(X^0,B)\xrightarrow{u}\Kbff(H^0(\A))(\ovl{X}\up,B)\to 0,
\]
where $u$ is the map which sends each $\abf\in H^0(\A)(X^0,B)$ to the complex morphism $u(\abf)\up\co \ovl{X}\up\to B$ given by
\[
u(\abf)^i=\begin{cases}
\abf& \text{if}\ i=0,\\
0&\text{otherwise}.
\end{cases}
\]
If we let $u'\co Z^0(\A)(X^0,B)\to\Kbff(H^0(\A))(\ovl{X}\up,B)$ denote the composition of 
\[
Z^0(\A)(X^0,B)\xrightarrow{\ovl{(\blank)}}H^0(\A)(X\up,B)\xrightarrow{u}\Kbff(H^0(\A))(\ovl{X}\up,B),
\]
then, the diagram
\[
\begin{tikzcd}
\Cbf(\A)(X\up,B)^{-1} & Z^0(\Cbf(\A))(X\up,B) & \Kbf(\A)(X\up,B) & 0\\
\A(X^0,B)^{-1}\oplus\A(X^1,B)^0 & Z^0(\A)(X^0,B) & \Kbff(H^0(\A))(\ovl{X}\up,B) & 0
\Ar{1-1}{1-2}{"d"}
\Ar{1-2}{1-3}{"\mathit{quot}"}
\Ar{1-3}{1-4}{}
\Ar{1-1}{2-1}{"w"',"\cong"}
\Ar{1-2}{2-2}{"v"',"\cong"}
\Ar{1-3}{2-3}{"\ovl{(\blank)}"}
\Ar{2-1}{2-2}{" {[}d_{\A}\ \,\blank\!\ci d_X^{0,0}{]}"'}
\Ar{2-2}{2-3}{"u'"'}
\Ar{2-3}{2-4}{}
\end{tikzcd}
\]
becomes commutative, where $\mathit{quot}$ denotes the canonical quotient map.
Since both rows are exact, the rightmost vertical map
$\ovl{(\blank)}\co\Kbf(\A)(X\up,B)\to \Kbff(H^0(\A))(\ovl{X}\up,B)$
is also an isomorphism.
\end{proof}

\subsection{Category $\Kbf^{1,n,1}(\A)$}

To resolve the non-functoriality observed in Remarks~\ref{rem:non-functorial1} and~\ref{rem:non-functorial2}, we introduce the following.
\begin{dfn}\label{dfn:K1n1}
We define categories $\Kbf^{1,n,1}(\A),\Kbf^{n,1}(\A),\Kbf^{1,n}(\A)$ by the following.
\begin{enumerate}
\item The category $\Kbf^{1,n,1}(\A)$ is
the quotient of $Z^0(\Cbf^{n+2}(\A))$ by the ideal $\calN$ defined by 
\[
\calN(X\up,Y\up)=\left\{d\varphi\up \in Z^0(\Cbf^{n+2}(\A))(X\up,Y\up) \;\middle|\; \begin{array}{l}
\varphi\up\in\Cbf^{n+2}(\A)(X\up,Y\up)^{-1},\\
\varphi^{1,0}=0\ \text{and}\ \varphi^{n+1,n}=0\ \text{in}\ \A
\end{array}\right\}
\]
for all $X\up, Y\up\in \Cbf^{n+2}(\A)$.  Namely, we define $\Kbf^{1,n,1}(\A)$ by
\begin{itemize}
\item $\ob(\Kbf^{1,n,1}(\A))=\ob(\Cbf^{n+2}(\A))$,
\item $\Kbf^{1,n,1}(\A)(X\up, Y\up)=Z^0(\Cbf^{n+2}(\A))(X\up,Y\up)/\calN(X\up,Y\up)$ for all objects $X\up,Y\up$.
\end{itemize}
\item The category $\Kbf^{n,1}(\A)$ is
the quotient of $Z^0(\Cbf_{[1,n+1]}(\A))$ by the ideal $\calN'$ defined by 
\[
\calN'(X\up,Y\up)=\left\{d\varphi\up \in Z^0(\Cbf_{[1,n+1]}(\A))(X\up,Y\up) \;\middle|\; \begin{array}{l}
\varphi\up\in\Cbf_{[1,n+1]}(\A)(X\up,Y\up)^{-1},\\
\varphi^{n+1,n}=0\ \text{in}\ \A
\end{array}\right\}
\]
for all $X\up, Y\up\in \Cbf_{[1,n+1]}(\A)$.
\item The category $\Kbf^{1,n}(\A)$ is
the quotient of $Z^0(\Cbf_{[0,n]}(\A))$ by the ideal $\calN^{\prime\prime}$ defined by 
\[
\calN^{\prime\prime}(X\up,Y\up)=\left\{d\varphi\up \in Z^0(\Cbf_{[0,n]}(\A))(X\up,Y\up) \;\middle|\; \begin{array}{l}
\varphi\up\in\Cbf_{[0,n]}(\A)(X\up,Y\up)^{-1},\\ \varphi^{1,0}=0\ \text{in}\ \A
\end{array}\right\}
\]
for all $X\up, Y\up\in \Cbf_{[0,n]}(\A)$.
\end{enumerate}
By definition, we have functors $\Cbf^{n+2}(\A)\to\Kbf^{1,n,1}(\A)\to\Kbf^{n+2}(\A)$, $\Cbf_{[1,n+1]}(\A)\to\Kbf^{n,1}(\A)\to\Kbf_{[1,n+1]}(\A)$ and $\Cbf_{[0,n]}(\A)\to\Kbf^{1,n}(\A)\to\Kbf_{[0,n]}(\A)$ canonically induced by the universality of the ideal quotients. 
\end{dfn}

\begin{dfn}\label{def:K1n1_truncations}
We define as follows.
\begin{enumerate}
\item We define  $\sig_{\ge 1}\co\Kbf^{1,n,1}(\A)\to\Kbf^{n,1}(\A)$ to be the additive functor induced from $\sig_{\ge 1}\co Z^0(\Cbf^{n+2}(\A))\to Z^0(\Cbf_{[1,n+1]})(\A)$ by the universality of the ideal quotient. Also, we define $\sig_{\le n}\co\Kbf^{1,n,1}(\A)\to\Kbf^{1,n}(\A)$ to be the additive functor induced from  $\sig_{\le n}\co Z^0(\Cbf^{n+2}(\A))\to Z^0(\Cbf_{[0,n]}(\A))$.
\item Similarly, we define $\sig_{\ge 1}\co\Kbf^{1,n}(\A)\to\Kbf_{[1,n]}(\A)$ and $\sig_{\le n}\co\Kbf^{n,1}(\A)\to\Kbf_{[1,n]}(\A)$ to be the additive functors induced from  $\sig_{\ge 1}\co Z^0(\Cbf_{[0,n]}(\A))\to Z^0(\Cbf_{[1,n]}(\A))$ and $\sig_{\le n}\co Z^0(\Cbf_{[1,n+1]})(\A)\to Z^0(\Cbf_{[1,n]})(\A)$, respectively.
\end{enumerate}
We note that
\[
\begin{tikzcd}[column sep=2em, row sep=0.8em]
 & \Kbf^{1,n,1}(\A) & \\
\Kbf^{1,n}(\A) & & \Kbf^{n,1}(\A) \\
 & \Kbf_{[1,n]}(\A) & 
\Ar{1-2}{2-1}{"\sig_{\le n}"'}
\Ar{1-2}{2-3}{"\sig_{\ge 1}"}
\Ar{2-1}{3-2}{"\sig_{\ge 1}"'}
\Ar{2-3}{3-2}{"\sig_{\le n}"}
\end{tikzcd}
\]
is commutative.
\end{dfn}

\begin{dfn}\label{def:K1n1_st}
We define as follows.
\begin{enumerate}
\item We define functors $s,t\co\Kbf^{1,n,1}(\A)\to H^0(\A)$ by the following.
\begin{itemize}
\item $s(X\up)=X^0$ and $t(X\up)=X^{n+1}$ for any object $X\up$ in $\Kbf^{1,n,1}(\A)$,
\item $s(f\up)=\ovl{f}^0$ and $t(f\up)=\ovl{f}^{n+1}$ for any morphism  $f\up$ in $\Kbf^{1,n,1}(\A)$.
\end{itemize}
\item Similarly, functors $s\co\Kbf^{1,n}(\A)\to H^0(\A)$ and $t\co\Kbf^{n,1}(\A)\to H^0(\A)$, which we denote by the same symbols as in {\rm (1)}, are defined by the following.
\begin{itemize}
\item $s(X\up)=X^0$ for any object $X\up$ in $\Kbf^{1,n}(\A)$, and 
$s(f\up)=\ovl{f}^0$ for any morphism  $f\up$ in $\Kbf^{1,n}(\A)$.
\item $t(X\up)=X^{n+1}$ for any object $X\up$ in $\Kbf^{n,1}(\A)$, and 
$t(f\up)=\ovl{f}^{n+1}$ for any morphism  $f\up$ in $\Kbf^{n,1}(\A)$.
\end{itemize}
\end{enumerate}
By definition, the diagrams
\[
\begin{tikzcd}
\Kbf^{1,n}(\A) & \Kbf^{1,n,1}(\A)\\
 & H^0(\A)
\Ar{1-2}{1-1}{"\sig_{\le n}"'}
\Ar{1-1}{2-2}{"s"'}
\Ar{1-2}{2-2}{"s"}
\end{tikzcd}
\quad\text{and}\quad
\begin{tikzcd}
\Kbf^{1,n,1}(\A) & \Kbf^{n,1}(\A)\\
H^0(\A) & 
\Ar{1-1}{1-2}{"\sig_{\ge 1}"}
\Ar{1-1}{2-1}{"t"'}
\Ar{1-2}{2-1}{"t"}
\end{tikzcd}
\]
are commutative.
\end{dfn}

\begin{prop}\label{prop:9-A}
Let $f\up\in Z^0(\Cbf^{n+2}(\A))(X\up,Y\up)$ be any morphism such that $\ovl{f}^0=\id$ and $\ovl{f}^{n+1}=\id$ in $H^0(\A)$. Suppose that $g\up\in Z^0(\Cbf^{n+2}(\A))(Y\up,X\up)$ gives the inverse of $f\up$ in $\Kbf^{1,n,1}(\A)$. Then, we have $\ovl{g}^0=\id$ and $\ovl{g}^{n+1}=\id$ in $H^0(\A)$.
\end{prop}
\begin{proof}
This is immediate from the functoriality of $s,t\co\Kbf^{1,n,1}(\A)\to H^0(\A)$.
\end{proof}

When $n=1$, the category $\Kbf^{1,n,1}(\A)$ is equivalent to the category $\calH_{3t}(\A)$ introduced by Chen in \cite{C1}, as follows.
\begin{prop}\label{prop:K111=H3t}
There is an equivalence of categories $F\co\Kbf^{1,1,1}(\A)\xrightarrow{\cong}\calH_{3t}(\A)$. Here, $\calH_{3t}(\A)$ denotes the \emph{homotopy category of $3$-term $h$-complexes} in 
\cite[Definition~3.14]{C1}.
\end{prop}
\begin{proof}
For details of the definition of category $\calH_{3t}(\A)$ and the notation, we refer to \cite{C1}.
Here, we describe only the explicit correspondence of objects and morphisms. It is straightforward to verify from the definitions of $\Kbf^{1,1,1}(\A)$ and $\calH_{3t}(\A)$ that this correspondence yields an equivalence of categories.

For any object $X\up\in\Kbf^{1,1,1}(\A)$, we define $F(X\up)$ to be
\begin{equation}\label{diag_3t}
\begin{tikzcd}
X^0 \arrow[r,"{\scriptstyle d_X^{0,1}}"] \arrow[rr, bend right=30, dashed, "{\scriptstyle d_X^{0,2}}"'] & X^1 \arrow[r,"{\scriptstyle d_X^{1,2}}"] & X^2
\end{tikzcd}
\end{equation}
which is indeed a $3$-term $h$-complex by $(\ref{eq_d})$, and hence an object in $\calH_{3t}(\A)$. This gives a correspondence between $\ob(\Kbf^{1,1,1}(\A))$ and $\ob(\calH_{3t}(\A))$.

For any morphism $f\up\in Z^0(\Cbf^3(\A))(X\up,Y\up)$, we associate a $6$-tuple $(r_0,r_1,r_2,s_1,s_2,t)$ by
$r_i=f^{i,i}\ (i=0,1,2)$, $s_j=(-1)^{j-1}f^{j-1,j}$ $(j=1,2)$, and $t=f^{0,2}$. By Remark~\ref{rem:comparison_with_H3t}, we have
\begin{eqnarray*}
&d_{\A}(r_i)=0\quad(i=0,1,2),&\\
&d_{\A}(s_1)=d_Y^{0,1}\ci r_0-r_1\ci d_X^{0,1},\quad
d_{\A}(s_2)=d_Y^{1,2}\ci r_1-r_2\ci d_X^{1,2},&\\
&d_{\A}(t)=-d_Y^{0,2}\ci r_0-d_Y^{1,2}\ci s_1-s_2\ci d_X^{0,1}+r_2\ci d_X^{0,2},&
\end{eqnarray*}
which is nothing but the equalities required in \cite[Definition~3.14]{C1}.
We define $F(f\up)\in\calH_{3t}(\A)(F(X\up),F(Y\up))$ to be the equivalence class with respect to the \emph{homotopy equivalence} given in \cite[Definition~3.14]{C1}.

In fact, the homotopy equivalence relation in \cite[Definition~3.14]{C1} agrees with the relation introduced in Definition~\ref{dfn:K1n1} {\rm (1)} via the above correspondence. More precisely, let $f\upp\in Z^0(\Cbf^3(\A))(X\up,Y\up)$ be another morphism, and associate $(r'_0,r'_1,r'_2,s'_1,s'_2,t')$ by 
$r'_i=f^{\prime i,i}\ (i=0,1,2)$, $s'_j=(-1)^{j-1}f^{\prime j-1,j}$ $(j=1,2)$, $t'=f^{\prime 0,2}$. By definition, $f\up=f\upp$ holds in $\Kbf^{1,1,1}(\A)$ if and only if there exists $\varphi\up\in\Cbf^3(\A)(X\up,Y\up)^{-1}$ such that $\varphi^{1,0}=0,\varphi^{2,1}=0$ in $\A$ and $f\up-f\upp=d\varphi\up$ in $\Cbf^3(\A)$.
Existence of such $\varphi\up$ is equivalent to the existence of a $6$-tuple 
\[
(r_0^{\prime\prime},r_1^{\prime\prime},r_2^{\prime\prime},s_1^{\prime\prime},s_2^{\prime\prime},t^{\prime\prime})=(\varphi^{0,0},-\varphi^{1,1},\varphi^{2,2},-\varphi^{0,1},-\varphi^{1,2},\varphi^{2,2})
\]
of morphisms in $\A$ that satisfies
\begin{eqnarray*}
&r_i-r'_i=d_{\A}(r^{\prime\prime}_i)\quad (i=0,1,2),&\\
&s_j-s'_j=d_{\A}(s^{\prime\prime}_j)+d_Y^{j-1,j}\ci r^{\prime\prime}_{j-1}-r^{\prime\prime}_j\ci d_X^{j-1,j}\quad (j=1,2),&\\
&t-t'=d_{\A}(t^{\prime\prime})+d_Y^{0,2}\ci r^{\prime\prime}_0-d_Y^{1,2}\ci s^{\prime\prime}_1+r^{\prime\prime}_2\ci d_X^{0,2}-s^{\prime\prime}_2\ci d_X^{0,1},& 
\end{eqnarray*}
which is equivalent to the homotopy equivalence defined in \cite[Definition~3.14]{C1}.

Thus $F$ gives a well-defined bijection $\Kbf^{1,1,1}(\A)(X\up,Y\up)\xrightarrow{\cong}(\calH_{3t}(\A))(F(X\up),F(Y\up))$ for each $X\up,Y\up\in\Kbf^{1,1,1}(\A)$. It is straightforward to check that this correspondence preserves composition of morphisms and identity morphisms. Hence we obtain an equivalence of categories $F\co\Kbf^{1,1,1}(\A)\xrightarrow{\cong}\calH_{3t}(\A)$.
\end{proof}

\normalcolor

\section{$n$-exact sequences and related notions}

\subsection{$n$-exact sequences}

\begin{dfn}\label{dfn:n-(co)-kernel}
Let $X\up$ be any object in $\Cbf^{n+2}(\A)$.
\begin{enumerate}
\item $X\up$ is \emph{left $n$-exact} if
$\Kbf(\A)(A[-i],X\up)=0$ for all $A\in\ob(\A)$ and for all $i\le n$.
\item $X\up$ is \emph{right $n$-exact} if $\Kbf(\A)(X\up,A[i-n])=0$ for all $A\in\ob(\A)$ and for all $i\le n$.
\item $X\up$ is \emph{$n$-exact} if it is left $n$-exact and right $n$-exact.
\end{enumerate}
\end{dfn}

\begin{rem}\label{rem:isom_n-ex}
By definition, if $X\up,X\upp\in\Cbf^{n+2}(\A)$ are isomorphic in $\Kbf^{n+2}(\A)$, then $X\up$ is left $n$-exact (respectively, right $n$-exact or $n$-exact) if and only if so is $X\upp$.
\end{rem}

If $\calA$ is an ordinary additive category, we may regard $\calA$ as a dg-category with $\calA^i=0$ whenever $i\ne0$, in which case $H^0(\calA)$ agrees with $\calA$. Under this identification, the functor $\ovl{(\blank)}\co Z^0(\Cbf(\calA)) \to Z^0(\Cbf(H^0(\calA)))=Z^0(\Cbff(\calA))$ may be viewed as the identity functor, and we identify $Z^0(\Cbf(\calA))$ with
the category of bounded complexes in $\calA$.
The same applies to the functor $\ovl{(\blank)}\co \Kbf(\calA) \to \Kbf(H^0(\calA))=\Kbff(\calA)$, as well as to their restrictions to full subcategories.
With these identifications, we will denote $\ovl{f}^i$ simply by $f^i$ for any morphism and any $i\in\bbZ$ below.
The following holds.
\begin{lem}\label{lem:n-exact_ordinary}
Suppose that $\calA$ is an ordinary additive category.
Let $i\in\bbZ$ be any integer.
For any $X\up\in\Cbff(\calA)$ and any $A\in\ob(\calA)$, the following are equivalent.
\begin{enumerate}
\item $\Kbff(\calA)(A[-i],X\up)=0$.
\item For the complex $(\calA(A,X^k),\calA(A,d_{X}^k))_{k\in\bbZ}$, that is, 
\[
\cdots\to
\calA(A,X^{k-1})\xrightarrow{d_X^{k-1}\ci\blank}
\calA(A,X^k)\xrightarrow{d_X^k\ci\blank}
\calA(A,X^{k+1})\xrightarrow{d_X^{k+1}\ci\blank}\cdots
\]
in $\Ab$, its $i$-th cohomology group is isomorphic to $0$.
\end{enumerate}
\end{lem}
\begin{proof}
Since $\calA$ is an ordinary additive category, the additive map
\[
\Cbff(\calA)(A[-i],X\up)^0\to \calA(A,X^i)\ ;\ f\up\mapsto f^i
\]
is an isomorphism of abelian groups.
We see that $f\up$ is closed in $\Cbff(\calA)$ if and only if $d_X^i\ci f^i=0$ holds in $\calA$. Moreover, for a closed morphism $f\up$, we have $f\up=0$ in $\Kbff(\calA)$ if and only if there exists $h\in \calA(A,X^{i-1})$ such that $f^i=d_X^{i-1}\ci h$. This shows that {\rm (1)} and {\rm (2)} are equivalent.
\end{proof}

\begin{prop}\label{prop:n-exact_ordinary}
Let $\calA$ be an ordinary additive category.
For any $X\up\in\Cbff^{n+2}(\calA)$, the following are equivalent.
\begin{enumerate}
\item $X\up$ is left $n$-exact in the sense of Definition~\ref{dfn:n-(co)-kernel}.
\item $X\up$ is a left $n$-exact sequence in $\calA$, in the sense of \cite[Definition~2.2]{J}.
\end{enumerate}
Similarly for right $n$-exact (respectively, $n$-exact) sequences. 
\end{prop}
\begin{proof}
By Lemma~\ref{lem:n-exact_ordinary}, $\Kbff^{n+2}(\calA)(A[-i],X\up)=0$ holds for all $A\in\ob(\calA)$ and all $i\le n$ if and only if the sequence
\[
0\to \calA(\blank,X^0)\xrightarrow{d_X^0\ci\blank}
\calA(\blank,X^1)\xrightarrow{d_X^1\ci\blank}
\cdots
\xrightarrow{d_X^n\ci\blank}
\calA(\blank,X^{n+1})
\]
is exact, which means that $X\up$ is a left $n$-exact sequence in the sense of \cite[Definition~2.2]{J}. Dually for the right $n$-exactness.
\end{proof}

\normalcolor

\begin{rem}\label{rem:1-exact}
When $n=1$, through the equivalence of categories $F\co\Kbf^{1,1,1}(\A)\xrightarrow{\cong}\calH_{3t}(\A)$ in Proposition~\ref{prop:K111=H3t}, object $X\up\in\Kbf^{1,1,1}(\A)$ is a (respectively, left/right) $1$-exact sequence if and only if $F(X\up)$ is a \emph{homotopy (resp.\ left/right) short exact sequence} in \cite[Definition~3.18]{C1}.
\end{rem}

We now return to the general case of a strictly connective and strictly additive dg-category $\A$.
By an argument similar to that in the proof of Lemma~\ref{lem:4A}, we obtain the following.
\begin{lem}\label{lem:htpy_modif}
Let $h\up\in Z^0(\Cbf(\A))(X\up,Y\up)$ be any morphism. Assume that $\varphi\up\in\Cbf(\A)(X\up,Y\up)^{-1}$ satisfies $h\up=d\varphi\up$.
Then, there exist $h\upp\in Z^0(\Cbf(\A))(X\up,Y\up)$ and  $\varphi\upp\in\Cbf(\A)(X\up,Y\up)^{-1}$ that satisfy the following conditions.
\begin{itemize}
\item $h\upp=h\up$ holds in $\Kbf(\A)$.
\item $h^{\prime n+1,n+1}=h^{n+1,n+1}-(-1)^{n+1}d_{\A}(\varphi^{n+1,n+1}$).
\item $\varphi^{\prime i,j}=\varphi^{i,j}$ for all $(i,j)$ except for $(i,j)=(n+1,n+1)$. In addition, $\varphi^{\prime\, n+1,n+1}=0$.
\item $h\upp=d\varphi\upp$.
\end{itemize}
\end{lem}
\begin{proof}
Define $\xi\up\in\Cbf(\A)(X\up,Y\up)^{-1}$ by
\[
\xi^{i,j}=\begin{cases}
\varphi^{n+1,n+1}&\text{if}\ (i,j)=(n+1,n+1)\\
0&\text{otherwise}
\end{cases}.
\]
Then, $h\upp=h\up-d\xi\up$ and $\varphi\upp=\varphi\up-\xi\up$ satisfy the required properties.
\end{proof}

The following is an analog of \cite[Claim~2.19]{HLN1}.
\begin{prop}\label{prp:htpy_modif}
Let $X\up,Y\up\in\Cbf^{n+2}(\A)$ be any pair of objects. For a morphism $h\up\in Z^0(\Cbf(\A))(X\up,Y\up)$, assume that $\varphi\up\in\Cbf(\A)(X\up,Y\up)^{-1}$ satisfies $h\up=d\varphi\up$.
\begin{enumerate}
\item Assume that $X\up,Y\up$ satisfy $\Kbf(\A)(X^{n+1}[-n],Y\up)=0$ $($for instance, when $Y\up$ is left $n$-exact$)$. Then, the following holds.
\begin{enumerate}
\item If $\overline{h}^{n+1}=0$ in $H^0(\A)$, then there exists $\psi\up\in\Cbf(\A)(X\up,Y\up)^{-1}$ such that $\psi^{i,j}=\varphi^{i,j}$ for $j\le n-2$, $\psi^{n+1,n}=0$, and $h\up=d\psi\up$.
\item If $h^{n+1,n+1}=0$ holds strictly in $\A$, then the morphism $\psi\up$ in the above {\rm (a)} can be chosen to satisfy $\psi^{n+1,n+1}=0$ too.
\end{enumerate}
\item Assume that $X\up,Y\up$ satisfy $\Kbf(\A)(X\up,Y^0[-1])=0$ $($for instance, when $X\up$ is right $n$-exact$)$. Then the following holds.
\begin{enumerate}
\item If $\overline{h}^0=0$ in $H^0(\A)$, then there exists $\psi\up\in\Cbf(\A)(X\up,Y\up)^{-1}$ such that $\psi^{i,j}=\varphi^{i,j}$ for $i\ge 3$, $\psi^{1,0}=0$, and $h\up=d\psi\up$.
\item If $h^{0,0}=0$ holds strictly in $\A$, then the morphism $\psi\up$ in the above {\rm (a)} can be chosen to satisfy $\psi^{0,0}=0$ too.
\end{enumerate}
\item Assume $n\ge 2$. If both the assumptions in {\rm (1)} and {\rm (2)} are satisfied, then the following holds.
\begin{enumerate}
\item If $\overline{h}^0=0$ and $\overline{h}^{n+1}=0$ in $H^0(\A)$, then there exists $\psi\up\in\Cbf(\A)(X\up,Y\up)^{-1}$ such that  $\psi^{1,0}=0$, $\psi^{n+1,n}=0$, and $h\up=d\psi\up$.
\item If $h^{0,0}=0$ and $h^{n+1,n+1}=0$ hold strictly in $\A$, then the morphism $\psi\up$ in the above {\rm (a)} can be chosen to satisfy $\psi^{0,0}=0$ and $\psi^{n+1,n+1}=0$ too.
\end{enumerate}
\end{enumerate}
\end{prop}
\begin{proof}
{\rm (1)} By Lemma~\ref{lem:htpy_modif}, we see that {\rm (a)} follows from {\rm (b)}. Let us show {\rm (b)}. Let $\al\up=\al_{X,n+1}\up\co(\sig_{\le n}X\up)[-1]\to X^{n+1}[-(n+1)]$ be the morphism in Proposition~\ref{prop:alphabetagamma}.
We have morphisms
\[
\begin{tikzcd}[column sep = 20]
X^{n+1}[-(n+1)]
\arrow[r, shift left=0.6ex, "v_{\al}\up"]
&
X\up
\arrow[l, shift left=0.6ex, "q_{\al}\up"]
\arrow[r, shift left=0.6ex, "p_{\al}\up"]
&
\sig_{\le n}X\up
\arrow[l, shift left=0.6ex,  "u_{\al}\up"]
\end{tikzcd}
\]
of degree $0$ in $\Cbf(\A)$
that satisfy the equations described in Definition~\ref{dfn:cocone}. We also have a commutative square \[
\begin{tikzcd}[column sep=2.4em, row sep=2em]
\sig_{\le n}X\up
  \arrow[r, "\scriptstyle\al\up {[}1{]}"]
  \arrow[d, "\scriptstyle s_-\upp"'] &
X^{n+1}{[}-n{]}
  \arrow[d, "\scriptstyle s_-\up"] \\
\sig_{\le n}X\up{[}-1{]}
  \arrow[r, "\scriptstyle\al\up"'] &
X^{n+1}{[}-(n+1){]}
\end{tikzcd}
\]
in $\Cbf(\A)$, where we put
$s_-\up=s_{-,X^{n+1}[-(n+1)]}\up\co X^{n+1}[-n]\to X^{n+1}[-(n+1)]$ and
$s_-\upp=s_{-,(\sig_{\le n}X\up)[-1]}\up\co \sig_{\le n}X\up\to (\sig_{\le n}X\up)[-1]$ for the morphisms introduced in Definition~\ref{dfn:shifts}.

By $h^{n+1,n+1}=0$, we have $h\up\ci v_{\al}\up=0$. Thus, $\varphi\up\ci v_{\al}\up\ci s_-\up\in\Cbf(\A)(X^{n+1}[-n],Y\up)^0$ satisfies
$d(\varphi\up\ci v_{\al}\up\ci s_-\up)=(d\varphi\up)\ci v_{\al}\up\ci s_-\up=h\up\ci v_{\al}\up\ci s_-\up=0$, hence it is closed.
Since $\Kbf(\A)(X^{n+1}[-n],Y\up)=0$ holds by assumption, there exists some $\eta\up\in\Cbf(\A)(X^{n+1}[-n],Y\up)^{-1}$ such that $\varphi\up\ci v_{\al}\up\ci s_-\up=d\eta\up$.
Put $\varphi\upp=\varphi\up\ci u_{\al}\up\ci p_{\al}\up$, $\zeta\up=\eta\up\ci\al\up[1]\ci p_{\al}\up$ and $\psi\up=\varphi\upp+\zeta\up$. Then $\psi\up$ satisfies
\begin{eqnarray*}
d\psi\up &=& (d\varphi\up)\ci u_{\al}\up\ci p_{\al}\up-\varphi\up\ci (du_{\al}\up)\ci p_{\al}\up+(d\eta\up)\ci\al\up[1]\ci p_{\al}\up\\
&=& h\up\ci(\id_{X\up}-v_{\al}\up\ci q_{\al}\up)-\varphi\up\ci(v_{\al}\up\ci s_-\upp\ci\al\up)\ci p_{\al}\up+(\varphi\up\ci v_{\al}\ci s_-\up)\up\ci\al\up[1]\ci p_{\al}\up\\
&=& h\up.
\end{eqnarray*}

It remains to check the other conditions.
Since $\eta^{i,j}=0$ for $j\le n-2$ and $(\al\up[1]\ci p_{\al}\up)^{n,j}=0$ for all $j$, we have $\zeta^{i,j}=0$ for $j\le n-2$ and $\zeta^{n+1,n}=0$, $\zeta^{n+1,n+1}=0$.
Also, since
\[
(u_{\al}\up\ci p_{\al}\up)^{i,j}=\begin{cases}
\id_{X^j}&\text{if}\ i=j\le n\\
0&\text{oteherwise}
\end{cases},
\]
we have
\[
\varphi^{\prime i,j}=\begin{cases}
\varphi^{i,j}&\text{if}\ i\le n\\
0&\text{oteherwise}
\end{cases}. 
\]
Since $\varphi\up$ is of degree $-1$, especially we have $\varphi^{\prime i,j}=0$ for $j\le n-2$.
Thus their sum $\psi\up$ satisfies $\psi^{i,j}=\varphi^{i,j}$ for $j\le n-2$ and $\psi^{n+1,n}=0$, $\psi^{n+1,n+1}=0$.

{\rm (2)} can be shown in a dual manner. {\rm (3)} follows immediately from {\rm (1)} and {\rm (2)}.
\end{proof}

\begin{cor}\label{cor:htpy_modif}
Assume $n\ge 2$.
Let $X\up,Y\up\in\Cbf^{n+2}(\A)$ be any pair of objects. Assume that $X\up$ is right $n$-exact and $Y\up$ is left $n$-exact. Suppose that $f\up,g\up\in Z^0(\Cbf^{n+2}(\A))(X\up,Y\up)$ satisfy $\ovl{f}^0=\ovl{g}^0$ and $\ovl{f}^{n+1}=\ovl{g}^{n+1}$. Then $f\up=g\up$ holds in $\Kbf^{n+2}(\A)(X\up,Y\up)$ if and only if $f\up=g\up$ holds in $\Kbf^{1,n,1}(\A)(X\up,Y\up)$.
\end{cor}
\begin{proof}
This follows immediately from Proposition~\ref{prp:htpy_modif} {\rm (3)(a)} applied to $h\up=g\up-f\up$.
\end{proof}

\begin{cor}\label{cor:htpy_modif2}
Let $X\up,Y\up\in\Cbf^{n+2}(\A)$ be any pair of objects with $X^0=Y^0$ and $X^{n+1}=Y^{n+1}$, and let $f\up \in Z^0(\Cbf^{n+2}(\A))(X\up,Y\up)$ be a morphism satisfying $\ovl{f}^0=\id$ and $\ovl{f}^{n+1}=\id$ in $H^0(\A)$. 
Assume that $f\up$ is an isomorphism in $\Kbf^{n+2}(\A)$. If $X\up$ or $Y\up$ is $n$-exact, then $f\up$ is an isomorphism in $\Kbf^{1,n,1}(\A)$. In particular, there exists a morphism $u\up\in Z^0(\Cbf^{n+2}(\A))(Y\up,X\up)$ that gives the inverse of $f\up$ in $\Kbf^{1,n,1}(\A)$, and such $u\up$ should satisfy $\ovl{u}^0=\id$ and $\ovl{u}^{n+1}=\id$ in $H^0(\A)$.
\end{cor}
\begin{proof}
Since $f\up$ an isomorphism in $\Kbf^{n+2}(\A)$, both $X\up$ and $Y\up$ are $n$-exact in either case. The final statement of the corollary follows from Proposition~\ref{prop:9-A}, so it
suffices to show that $f\up$ is an isomorphism in $\Kbf^{1,n,1}(\A)$.

By Lemma~\ref{lem:4A}, there exists $f\upp\in Z^0(\Cbf(\A))(X\up,Y\up)$ that satisfies $f\up=f\upp$ in $\Kbf(\A)$ and $f^{\prime 0,0}=\id$, $f^{\prime n+1,n+1}=\id$ in $\A$. Thus, replacing $f\up$ with $f\upp$ if necessary, we may assume that $f\up$ satisfies $f^{0,0}=\id$ and $f^{n+1,n+1}=\id$ from the beginning.

Suppose that $g\up \in Z^0(\Cbf(\A))(Y\up,X\up)$ gives the inverse of $f\up$ in $\Kbf^{n+2}(\A)$.
First, we show that we can modify $g\up$
so that it moreover satisfies $\ovl{g}^{n+1}=\id$. Put $X^{n+1}=Y^{n+1}=C$.
Since $g\up\ci f\up=\id$ holds in $\Kbf^{n+2}(\A)$, there should exist $\varphi\up\in \Cbf(\A)(X\up,X\up)^{-1}$ such that
$d\varphi\up=\id_{X\up}-g\up\ci f\up$. Since
$f^{n+1,n+1}=\id_C$, it follows
$\id_C-g^{n+1,n+1}=(d\varphi)^{n+1,n+1}$ in $\A$. Define $\psi\up\in\Cbf(\A)(Y\up,X\up)^{-1}$ by
\[
\psi^{i,j}=\begin{cases}
\varphi^{n+1,n}&\text{if}\ (i,j)=(n+1,n)\\
\varphi^{n+1,n+1}&\text{if}\ (i,j)=(n+1,n+1)\\
0&\text{otherwise}
\end{cases}
\]
and put $g\upp=g\up+d\psi\up\in Z^0(\Cbf^{n+2}(\A))(Y\up,X\up)$.
Then, since $(d\psi)^{0,0}=0$ and $(d\psi)^{n+1,n+1}=(d\varphi)^{n+1,n+1}$, we see that $g\upp$ satisfies $g^{\prime0,0}=g^{0,0}$ and $g^{\prime n+1,n+1}=\id_C$ in $\A$. Since $g\up=g\upp$ holds in $\Kbf^{n+2}(\A)$, it still gives the inverse of $f\up$ in $\Kbf^{n+2}(\A)$.

In a dual manner, by using $f\up\ci g\up=\id$ in $\Kbf^{n+2}(\A)$, we can modify $g\upp$ to obtain $g\uppp\in Z^0(\Cbf^{n+2}(\A))(Y\up,X\up)$ that gives the inverse of $f\up$ in $\Kbf^{n+2}(\A)$ and moreover satisfies $g^{\prime\prime 0,0}=\id$ and $g^{\prime\prime n+1,n+1}=g^{\prime n+1,n+1}=\id$ in $\A$.
Thus, replacing $g\up$ with $g\uppp$ if necessary, we may assume that $g\up$ satisfies $g^{0,0}=\id$ and $g^{n+1,n+1}=\id$ in $\A$ from the beginning.

When $n\ge 2$, by Corollary~\ref{cor:htpy_modif}, the equalities $g\up\ci f\up=\id$ and $f\up\ci g\up=\id$ in $\Kbf^{n+2}(\A)$ imply $g\up\ci f\up=\id$ and $f\up\ci g\up=\id$ in $\Kbf^{1,n,1}(\A)$. This means that $f\up$ is an isomorphism in $\Kbf^{1,n,1}(\A)$.

It remains to consider the case $n=1$. By the above argument, a morphism $g\up\in Z^0(\Cbf(\A))(Y\up,X\up)$ which gives the inverse of $f\up$ in $\Kbf^3(\A)$ can be taken to satisfy $g^{0,0}=\id$ and $g^{2,2}=\id$. Let $\eta\up\in\Cbf(\A)(X\up,X\up)^{-1}$ be a morphism such that $\id_{X\up}-g\up\ci f\up=d\eta\up$.
By Proposition~\ref{prp:htpy_modif} {\rm (2)}, we may assume that $\eta\up$ satisfies $\eta^{0,0}=0$ and $\eta^{1,0}=0$. Define $\xi\up\in\Cbf(\A)(Y\up,X\up)^{-1}$ by
\[ 
\xi^{i,j}=\begin{cases}
\eta^{2,1} & \text{if}\ (i,j)=(2,1)\\
0 & \text{otherwise}
\end{cases}
\]
and put $h\up=g\up+d\xi\up\in Z^0(\Cbf(\A))(Y\up,X\up)$, $\zeta\up=\eta\up-\xi\up\ci f\up\in\Cbf(\A)(X\up,X\up)^{-1}$.
Then they satisfy
\[
\zeta^{1,0}=\eta^{1,0}-\xi^{1,0}\ci f^{1,1}=\eta^{1,0}=0,
\]
\[
\zeta^{2,1}=\eta^{2,1}-\xi^{2,1}\ci f^{2,2}=\eta^{2,1}-\eta^{2,1}=0
\]
and
\[
\id_{X\up}-h\up\ci f\up=\id_{X\up}-\big(g\up\ci f\up+(d\xi\up)\ci f\up\big)
=d\eta\up-d(\xi\up\ci f\up)=d\zeta\up.
\]
This means that $h\up\ci f\up=\id_{X\up}$ holds in $\Kbf^{1,1,1}(\A)$. Thus $f\up$ is a split monomorphism in $\Kbf^{1,1,1}(\A)$. A dual argument shows that $f\up$ is also a split epimorphism in $\Kbf^{1,1,1}(\A)$, hence an isomorphism.
\end{proof}

\begin{cor}\label{cor:htpy_modif3}
Let $X\up,Y\up\in\Cbf^{n+2}(\A)$ be any pair of objects with $X^0=Y^0=A$ and $X^{n+1}=Y^{n+1}=C$, and let $f\up \in \Kbf(\A)(X\up,Y\up)$ be any morphism. Assume that $f\up$ is an isomorphism in $\Kbf^{n+2}(\A)$, and that $X\up$ or $Y\up$ is $($hence both are$)$ $n$-exact. 
If $f\up$ makes the diagrams
\begin{equation}\label{comm_two}
\begin{tikzcd}[row sep = 22, column sep = 14]
X\up &   & Y\up \\
     & A &   
\Ar{1-1}{1-3}{"f\up"}
\Ar{1-1}{2-2}{"\gam_{X,1}\up"'}
\Ar{1-3}{2-2}{"\gam_{Y,1}\up"}
\end{tikzcd}
\quad\text{and}\quad
\begin{tikzcd}[row sep = 22, column sep = -4]
     & C[-(n+1)] &      \\
X\up &   & Y\up 
\Ar{1-2}{2-1}{"\be_{X,n+1}\up"'}
\Ar{1-2}{2-3}{"\be_{Y,n+1}\up"}
\Ar{2-1}{2-3}{"f\up"'}
\end{tikzcd}
\end{equation}
commutative in $\Kbf(\A)$, then $f\up$ is an isomorphism in $\Kbf^{1,n,1}(\A)$. Here, the morphisms in $(\ref{comm_two})$ other than $f\up$ are those introduced in Proposition~\ref{prop:alphabetagamma}.
\end{cor}
\begin{proof}
By the commutativity of $(\ref{comm_two})$, there exists $f\upp\in Z^0(\Cbf(\A))(X\up,Y\up)$ such that $f\up=f\upp$ in $\Kbf(\A)$ and  $f^{\prime 0,0}=\id_A$, $f^{\prime n+1,n+1}=\id_C$ in $\A$. Thus, we may apply Corollary~\ref{cor:htpy_modif2} to $f\upp$.
\end{proof}

\normalcolor

\begin{lem}\label{lem:extend_morph}
Let $X\up\in\Cbf_{\le n+1}(\A)$ and $Y\up\in\Cbf^{n+2}(\A)$ be any pair of objects. Assume that $Y\up$ is left $n$-exact.
Then, the following holds.
\begin{enumerate}
\item For any $g\up\in Z^0(\Cbf(\A))(\sig_{\ge n}X\up,Y\up)$, there exists $f\up\in Z^0(\Cbf(\A))(X\up,Y\up)$ such that $f\up\ci \be_{X,n}\up=g\up$ in $\Kbf(\A)$ and $\ovl{f}^{n+1}=\ovl{g}^{n+1}$ in $H^0(\A)$ hold. Here, the morphism $\be_{X,n}\up$ is that in Proposition~\ref{prop:alphabetagamma}.
\item Let $f\up \in Z^0(\Cbf(\A))(X\up,Y\up)$ be any morphism. If there exist $b_n\in\A(X^{n+1},Y^n)^0$ and $b_{n+1}\in\A(X^{n+1},Y^{n+1})^{-1}$ satisfying $f^{n+1,n+1}=d_Y^{n,n+1}\ci b_n+d_{\A}(b_{n+1})$, then $f\up$ is equal to $0$ in $\Kbf(\A)$.
\end{enumerate}
\end{lem}

\begin{proof}
{\rm (1)} By Proposition~\ref{prop:alphabetagamma}, there is a distinguished triangle
\[ (\sig_{\le n-1}X\up)[-1]\xrightarrow{\al\up_{X,n}} \sig_{\ge n}X\up\xrightarrow{\be\up_{X,n}}X\up\xrightarrow{\gam\up_{X,n}}\sig_{\le n-1}X\up \]
in $\Kbf(\A)$.
Since $Y\up$ is left $n$-exact, we have $\Kbf(\A)((\sig_{\le n-1}X\up)[-1],Y\up)=0$, hence there exists $f\upp\in Z^0(\Cbf(\A))(X\up,Y\up)$ such that $f\upp\ci \be_{X,n}\up=g\up$ in $\Kbf(\A)$. In particular, there exists some $h\in\A(X^{n+1},Y^n)^0$ such that $\ovl{g}^{n+1}=\ovl{f'}^{n+1}+\ovl{d}_Y^n\ci\ovl{h}$ in $H^0(\A)$. Define $\xi\up\in\Cbf(\A)(X\up,Y\up)^{-1}$ by
\[
\xi^{i,j}=\begin{cases}
h&\text{if}\ (i,j)=(n+1,n)\\
0&\text{otherwise}
\end{cases},
\]
and put $f\up=f\upp+d\xi\up\in Z^0(\Cbf(\A))(X\up,Y\up)$. Then, $f\up$ satisfies the required properties.

{\rm (2)} The assumed existence of $b_n$ and $b_{n+1}$ implies that the morphism $f\up\ci\be_{X,n+1}\up\in Z^0(\Cbf(\A))(\sig_{\ge n+1}X\up,Y\up)$ becomes equal to $0$ in $\Kbf(\A)$.
Since there is a distinguished triangle
\[ (\sig_{\le n}X\up)[-1]\xrightarrow{\al\up_{X,n+1}} \sig_{\ge n+1}X\up\xrightarrow{\be\up_{X,n+1}}X\up\xrightarrow{\gam\up_{X,n+1}}\sig_{\le n}X\up \]
by Proposition~\ref{prop:alphabetagamma}, the morphism $f\up$ should factor through $\sig_{\le n}X\up$ in $\Kbf(\A)$. This forces $f\up$ to be $0$ in $\Kbf(\A)$, since we have $\Kbf(\A)(\sig_{\le n}X\up,Y\up)=0$ by the left $n$-exactness of $Y\up$.
\end{proof}

\begin{prop}\label{prp:uniqueness_n-kernel}
Let $X\up,Y\up\in\Cbf^{n+2}(\A)$ be any pair of left $n$-exact sequences. For any $g\up\in Z^0(\Cbf_{[n,n+1]}(\A))(\sig_{\ge n}X\up,\sig_{\ge n}Y\up)$,
there exists $f\up\in Z^0(\Cbf^{n+2}(\A))(X\up,Y\up)$ which makes
\[
\begin{tikzcd}[column sep=2.4em, row sep=2em]
\sig_{\ge n}X\up
  \arrow[r, "\scriptstyle\be_{X,n}\up"]
  \arrow[d, "\scriptstyle g\up"'] &
X\up
  \arrow[d, "\scriptstyle f\up"] \\
\sig_{\ge n}Y\up
  \arrow[r, "\scriptstyle\be_{Y,n}\up"'] &
Y\up
\end{tikzcd}
\]
commutative in $\Kbf(\A)$ and satisfies $\ovl{f}^{n+1}=\ovl{g}^{n+1}$ in $H^0(\A)$. Moreover, if $g\up$ gives an isomorphism in $\Kbf(\A)$, then such $f\up$ becomes an isomorphism in $\Kbf(\A)$.
\end{prop}
\begin{proof}
The first part is immediate from Lemma~\ref{lem:extend_morph} {\rm (1)}.
For the latter part, suppose that $g\up$ is an isomorphism in $\Kbf(\A)$. Then, there is $g\upp\in Z^0(\Cbf_{[n,n+1]}(\A))(\sig_{\ge n}Y\up,\sig_{\ge n}X\up)$ that gives the inverse of $g\up$ in $\Kbf(\A)$. By the same argument, there also exists $f\upp\in Z^0(\Cbf^{n+2}(\A))(Y\up,X\up)$ such that $f\upp\ci\be_{Y,n}\up=\be_{X,n}\up\ci g\upp$ holds in $\Kbf(\A)$. In $\Kbf(\A)$, since $g\up\ci g\upp=\id$ holds, we have $f\up\ci f\upp\ci\be_{Y,n}\up=\be_{Y,n}\up$. Hence there exist $b_n\in\A(Y^{n+1},Y^n)^0$ and $b_{n+1}\in\A(Y^{n+1},Y^{n+1})^{-1}$ satisfying $f^{n+1,n+1}\ci f^{\prime n+1,n+1}-\id_{Y^{n+1}}=d_Y^{n,n+1}\ci b_n+d_{\A}(b_{n+1})$. By Lemma~\ref{lem:extend_morph} {\rm (2)} we obtain $f\up\ci f\upp=\id$ in $\Kbf(\A)$. Similarly for $f\upp\ci f\up=\id$, thus $f\up$ is an isomorphism in $\Kbf(\A)$.
\end{proof}

\normalcolor

\begin{cor}\label{cor:uniqueness_n-kernel}
Let $X\up,X\upp\in\Cbf^{n+2}(\A)$ be left $n$-exact sequences. If $d_X^{n,n+1}=d_{X'}^{n,n+1}$, then $X\up\cong X\upp$ in $\Kbf^{n+2}(\A)$. More precisely, there are $f\up\in Z^0(\A)(X\up,X\upp)$ and $f\upp\in Z^0(\A)(X\upp,X\up)$ which are mutually inverse in $\Kbf^{n+2}(\A)$, and moreover satisfies $\ovl{f}^{n+1}=\ovl{f'}^{n+1}=\id$ in $H^0(\A)$.
\end{cor}
\begin{proof}
This follows from Proposition~\ref{prp:uniqueness_n-kernel} and its proof, applied to $\id\co \sig_{\ge n}X\up\to \sig_{\ge n}X\upp$.
\end{proof}

\begin{rem}\label{rem:n-kernel}
Corollary~\ref{cor:uniqueness_n-kernel} means that a left $n$-exact sequence $X\up$ is determined by $d_X^{n,n+1}$, uniquely up to isomorphism in $\Kbf^{n+2}(\A)$. From this reason, we may call $X\up$ the \emph{$n$-kernel} of $d_X^{n,n+1}$. Similarly, if $X\up$ is right $n$-exact, we call $X\up$ the \emph{$n$-cokernel} of $d_X^{0,1}$.
\end{rem}

\begin{lem}\label{lem:vanishing}
Let $X\up\in\Kbf^{n+2}(\A)$. Then the following conditions are equivalent:
\begin{enumerate}
    \item $X\up$ is a left $n$-exact sequence in $\A$. 
    \item The morphism $X^0[-1]\xrightarrow{\al\up_{X,1}}\sig_{\ge 1}X\up$ in $\Kbf(\A)$ obtained in Proposition~\ref{prop:alphabetagamma} induces an isomorphism
    $$\al\up_{X,1}\ci\blank\co\Kbf(\A)(A[-i], X^0[-1])\xrightarrow{\cong} \Kbf(\A)(A[-i], \sig_{\ge 1}X\up)$$
  for any $i\le n$ and any $A\in\ob(\A)$.
\end{enumerate}
\end{lem}
\begin{proof}
Consider the distinguished triangle
$$X^0[-1]\xrightarrow{\al\up_{X,1}} \sig_{\ge 1}X\up\to X\up\to X^0$$
and apply the cohomological functor $\Kbf(\A)(A[-i],\blank)$ to this triangle, for arbitrary $A\in\A$ and $i\le n$. Then the statement follows from the resulting long exact sequence.
\end{proof}

\normalcolor

\subsection{Split sequences}
The following characterization of zero objects in $\Kbf^{n+2}(\A)$ is obtained from Lemma~\ref{lem:extend_morph}. 
\begin{prop}\label{prp:characterization_split_seq}
For any $X\up\in\Cbf^{n+2}(\A)$, the following are equivalent.
\begin{enumerate}
  \item $X\up$ is a zero object in $\Kbf^{n+2}(\A)$.
  \item $X\up$ is left $n$-exact and $\ovl{d}_X^n$ is a split epimorphism in $H^0(\A)$.
  \item $X\up$ is right $n$-exact and $\ovl{d}_X^0$ is a split monomorphism in $H^0(\A)$.
\end{enumerate}
\end{prop}
\begin{proof}
Since $(1)\Leftrightarrow(3)$ can be shown in a dual manner, we only show $(1)\Leftrightarrow(2)$.

Suppose that $X\up$ is left $n$-exact and $\ovl{d}_X^n$ is a split epimorphism in $H^0(\A)$. By applying Lemma~\ref{lem:extend_morph} {\rm (2)} to $\id_{X\up}$, we obtain $\id_{X\up}=0$ in $\Kbf(\A)$, which means $X\up\cong 0$ in $\Kbf^{n+2}(\A)$. 

Conversely, suppose that $X\up\cong 0$ holds in $\Kbf(\A)$. The left $n$-exactness of $X\up$ is obvious by Remark~\ref{rem:isom_n-ex}. Since $\id_{X\up}=0$ holds as morphisms in $\Kbf(\A)$, there exists $\varphi\up\in\Cbf(\A)(X\up,X\up)^{-1}$ such that $d\varphi\up=\id_{X\up}$. In particular, we have $\id_{X^{n+1}}=(-1)^{n+1}d_{\A}(\varphi^{n+1,n+1})+d_X^{n,n+1}\ci\varphi^{n+1,n}$, hence $\id_{X^{n+1}}=\ovl{d}_X^n\ci \ovl{\varphi}^{n+1}$ holds in $H^0(\A)$. This shows that $\ovl{d}_X^n$ is a split epimorphism in $H^0(\A)$.
\end{proof} 

\begin{dfn}\label{dfn:split_seq}
We call $X\up\in \Cbf^{n+2}(\A)$ a \emph{split sequence} if the equivalent condition in Proposition~\ref{prp:characterization_split_seq} is satisfied. 
\end{dfn}

By definition, any split sequence is $n$-exact.
For each $A,C\in\ob(\A)$, the following gives an explicit description of a specified split sequence that starts at $A$ and ends with $C$. 
\begin{dfn}\label{dfn:splitN}
Let $A,C\in\ob(\A)$ be any pair of objects. There exists a split sequence $N\up={}_AN\up_C\in\Cbf^{n+2}(\A)$ satisfying $N^0=A$ and $N^{n+1}=C$ given by the following.
\begin{enumerate}
\renewcommand{\labelenumi}{(\roman{enumi})}
\item If $n\ge 2$, then 
$N^i=\begin{cases}
A&\text{if}\ i=0,1\\
C&\text{if}\ i=n,n+1\\ 
0&\text{otherwise}
\end{cases}$\ \ , \ \ 
$d_N^{i,j}=\begin{cases}
\id_A&\text{if}\ (i,j)=(0,1)\\
\id_C&\text{if}\ (i,j)=(n,n+1)\\ 
0&\text{otherwise}
\end{cases}$\ \ .

\item If $n=1$, then $N^i=\begin{cases}
A&\text{if}\ i=0\\
A\oplus C&\text{if}\ i=1\\ 
C&\text{if}\ i=2\\ 
0&\text{otherwise}
\end{cases}$\ \ , \ \ 
$d_N^{i,j}=\begin{cases}
\begin{bmatrix}1\\0\end{bmatrix}
&\text{if}\ (i,j)=(0,1)\\
[0\ 1]&\text{if}\ (i,j)=(1,2)\\ 
0&\text{otherwise}
\end{cases}$\ \ .
\end{enumerate}
Indeed, such $N\up$ obviously satisfies conditions {\rm (2)} and {\rm (3)} in Proposition~\ref{prp:characterization_split_seq}.
\end{dfn}

Since $N\up$ is a split sequence, as in Proposition~\ref{prp:characterization_split_seq}, morphism ${}_AN_{C}\up\to X\up$ is unique in $\Kbf^{n+2}(\A)$. We may take its specified representative by the following lemma.
\begin{lem}\label{lem:NtoX}
Let $A,C\in\ob(\A)$ be any pair of objects, and let ${}_AN\up_C$ be the split sequence in Definition~\ref{dfn:splitN}. For any object $X\up\in\Cbf^{n+2}(\A)$, the following holds.
\begin{enumerate}
\item If $X^0=A$, then there is a morphism $f\up\in Z^0(\Cbf(\A))({}_AN_C\up,X\up)$ such that $f^{0,0}=\id_A$ and $f^{n+1,n+1}=0$ in $\A$.
\item If $X\up$ is a split sequence that satisfies $X^0=A$ and $X^{n+1}=C$, then there is a morphism $g\up\in Z^0(\Cbf(\A))({}_AN_C\up,X\up)$ such that $g^{0,0}=\id_A$ and $g^{n+1,n+1}=\id_C$ in $\A$. 
\end{enumerate}
\end{lem}
\begin{proof}
{\rm (1)} It is straightforward to show that $f\up\co {}_AN_C\up\to X\up$ defined by
\[
f^{i,j}=\begin{cases}
\id_A&\text{if}\ (i,j)=(0,0)\\
d_X^{0,j}&\text{if}\ i=1\\
0&\text{otherwise}
\end{cases}
\]
if $n\ge2$, and by
\[
f^{i,j}=\begin{cases}
\id_A&\text{if}\ (i,j)=(0,0)\\
[d_X^{0,j}\ 0]&\text{if}\ i=1\\
0&\text{otherwise}
\end{cases}
\]
if $n=1$, gives such a morphism.

{\rm (2)} Since $X\up$ is a split sequence, $\ovl{d}_X^n$ is a split epimorphism in $H^0(\A)$. Thus there exist $\phi_n\in Z^0(\A)(C,X^n)$ and $\phi_{n+1}\in\A(C,X^{n+1})^{-1}$ such that $\id_C-d_X^{n,n+1}\ci\phi_n=d_{\A}(\phi_{n+1})$ in $\A$.
Define $h\up\co {}_AN_C\up\to X\up$ by
\[
h^{i,j}=\begin{cases}
\phi_n&\text{if}\ (i,j)=(n,n)\\
(-1)^{n+1}\phi_{n+1}&\text{if}\ (i,j)=(n,n+1)\\
\id_C&\text{if}\ (i,j)=(n+1,n+1)\\
0&\text{otherwise}
\end{cases}
\]
if $n\ge2$, and by
\[
h^{i,j}=\begin{cases}
[0\ \phi_1]&\text{if}\ (i,j)=(1,1)\\
[0\ \phi_2]&\text{if}\ (i,j)=(n,2)\\
\id_C&\text{if}\ (i,j)=(2,2)\\
0&\text{otherwise}
\end{cases}
\]
if $n=1$.
We may check that it gives a morphism $h\up\in Z^0(\Cbf(\A))({}_AN_C\up,X\up)$. If we put $g\up=f\up+h\up$, where $f\up$ is a morphism obtained in {\rm (1)}, then $g\up$ satisfies the required properties.
\end{proof}

\begin{prop}\label{prop:NtoX}
Let $X\up\in\Cbf^{n+2}(\A)$ be any object, with $X^0=A$ and $X^{n+1}=C$.
The following are equivalent.
\begin{enumerate}
\item $X\up$ is a split sequence.
\item $X\up$ is isomorphic to ${}_AN_C\up$ in $\Kbf^{1,n,1}(\A)$, where ${}_AN_C\up$ is the split sequence given in Definition~\ref{dfn:splitN}.
\end{enumerate}
Moreover, if either of these equivalent conditions {\rm (1),(2)} is satisfied, then a morphism $f\up\in Z^0(\Cbf(\A))({}_AN_C\up,X\up)$ that gives an isomorphism in $\Kbf^{1,n,1}(\A)$ can be taken to satisfy $f^{0,0}=\id_A$ and $f^{n+1,n+1}=\id_C$ in $\A$.
\end{prop}
\begin{proof}
$(2)\Rightarrow(1)$ is obvious. Let us show $(1)\Rightarrow(2)$.
Suppose that $X\up$ is a split sequence. By Lemma~\ref{lem:NtoX} {\rm (2)}, there exists $f\up\in Z^0(\Cbf(\A))({}_AN_C\up,X\up)$ such that $f^{0,0}=\id_A$ and $f^{n+1,n+1}=\id_C$ in $\A$. Since both $X\up$ and ${}_AN\up_C$ are zero objects in $\Kbf^{n+2}(\A)$, it follows that $f\up$ is an isomorphism in $\Kbf^{n+2}(\A)$.
Then, Corollary~\ref{cor:htpy_modif2} shows that $f\up$ is an isomorphism in $\Kbf^{1,n,1}(\A)$.

The last assertion is obvious from the above argument.
\end{proof}

\begin{cor}\label{cor:NtoX}
Let $A,C\in\ob(\A)$ be any pair of objects.
All split sequences $X\up\in\Cbf^{n+2}(\A)$ with $X^0=A$ and $X^{n+1}=C$ are mutually isomorphic in $\Kbf^{1,n,1}(\A)$.
\end{cor}
\begin{proof}
This is immediate from Proposition~\ref{prop:NtoX}. Indeed, any split sequence $X\up\in\Cbf^{n+2}(\A)$ with $X^0=A$ and $X^{n+1}=C$ is isomorphic to ${}_AN\up_C$ in $\Kbf^{1,n,1}(\A)$.
\end{proof}

\normalcolor

\subsection{$n$-pullback morphisms}

\begin{dfn}\label{dfn:pull_of_admissible}
Let $f\up\in Z^0(\Cbf^{n+2}(\A))(X\up,Y\up)$ be any morphism.
\begin{enumerate}
\item $f\up$ is called an \emph{$n$-pullback morphism} if 
$\Kbf(\A)(A[-i],\CoCone f\up)=0$ holds for all $A\in\A$ whenever $i\le n+1$, and satisfies $\ovl{f}^0=\id_{X^0}$ in $H^0(\A)$.
\item $f\up$ is called an \emph{$n$-pushout morphism} if $\Kbf(\A)(\Cone f\up,A[i-n-1])=0$ holds for all $A\in\A$ whenever $i\le n+1$, and satisfies $\ovl{f}^{n+1}=\id_{X^{n+1}}$ in $H^0(\A)$. 
\end{enumerate}
\end{dfn}

\begin{lem}\label{lem:equivalent_n-pbm1}
Let $X\up,Y\up\in\Cbf^{n+2}(\A)$ be any pair of objects.
Suppose that a morphism $f\up\in Z^0(\Cbf^{n+2}(\A))(X\up,Y\up)$ satisfies $\ovl{f}^0=\id$ in $H^0(\A)$. Then, the following are equivalent.
\begin{enumerate}
\item $f\up$ is an $n$-pullback morphism.
\item $\Kbf(\A)(W\up,\CoCone f\up)=0$ for all $W\up\in\Cbf_{\le n+1}(\A)$.
\item $\CoCone((\sig_{\ge 1}f\up)[1])$ is left $n$-exact.
\end{enumerate}
\end{lem}
\begin{proof}
$(1)\Leftrightarrow(2)$ is shown by the usual argument in a triangulated category, using Proposition~\ref{prop:alphabetagamma}.
$(1)\Leftrightarrow(3)$ follows from Corollary~\ref{cor:devide_f} {\rm (2)}.
\end{proof}

\begin{prop}\label{prop:npb_of_split}
If $f\up\in Z^0(\Cbf^{n+2}(\A))(X\up,Y\up)$ is an $n$-pullback morphism and if $Y\up$ is a split sequence, then $X\up$ is also a split sequence.
\end{prop}
\begin{proof}
Since $Y\up[-1]\to \CoCone f\up\to X\up\xrightarrow{f\up}Y\up$ is a distinguished triangle in $\Kbf(\A)$, the sequence
\[
\Kbf(\A)(\blank,\CoCone f\up)\to 
\Kbf(\A)(\blank,X\up)\to
\Kbf(\A)(\blank,Y\up) 
\]
is exact. Since $Y\up$ is a split sequence, we have $\Kbf(\A)(X\up,Y\up)=0$. Also, since $f\up$ is an $n$-pullback morphism, we have $\Kbf(\A)(X\up,\CoCone f\up)=0$ by Lemma~\ref{lem:equivalent_n-pbm1}. Thus, the above exactness implies $\Kbf(\A)(X\up,X\up)=0$,  and hence $X\up\cong 0$ holds in $\Kbf(\A)$.
\end{proof}

\normalcolor

\normalcolor

\begin{prop}\label{prp:univ_npb}
Assume that $Y\up\in\Cbf^{n+2}(\A)$ is a left $n$-exact sequence
and $f\up\co X\up\to Y\up$ is an $n$-pullback morphism.
Let $W\up\in\Cbf^{n+2}(\A)$ be any object. Suppose that for a morphism $g\up\in Z^0(\Cbf^{n+2}(\A))(W\up,Y\up)$, there exist $b\in Z^0(\A)(W^{n+1},X^{n+1})$ and $\phi_n\in\A(W^{n+1},Y^n)^0$ satisfying $\ovl{g}^{n+1}=\ovl{f}^{n+1}\circ\ovl{b}+\ovl{d}_Y^n\circ\ovl{\phi_n}$ in $H^0(\A)$.
Then, there exists $e\up\in Z^0(\Cbf^{n+2}(\A))(W\up,X\up)$ that satisfies $g\up=f\up\ci e\up$ in $\Kbf^{n+2}(\A)$. Moreover, such $e\up$ is unique in $\Kbf^{n+2}(\A)$.
Furthermore, $e\up$ can be chosen satisfy $\ovl{e}^0=\ovl{g}^0$ and $\ovl{e}^{n+1}=\ovl{b}$ in $H^0(\A)$ too.
\end{prop}
\begin{proof}
By a similar way as in the proof of Lemma~\ref{lem:4A}, we may assume $f^{0,0}=\id$ from the beginning.
Let $\CoCone f\up\to X\up\xrightarrow{f\up}Y\up\xrightarrow{f\upp}(\CoCone f\up)[1]$ be the distinguished triangle in $\Kbf(\A)$ associated to $f\up$.
Uniqueness of $e\up$ follows from the exactness of
\[
0\to\Kbf(\A)(W\up,X\up)\xrightarrow{f\up\ci\blank}\Kbf(\A)(W\up,Y\up).
\]
As in Proposition~\ref{prop:alphabetagamma}, we also have a distinguished triangle
\[
(\sig_{\le n}W\up)[-1]\xrightarrow{\al_{W,n+1}}W^{n+1}[-(n+1)]\xrightarrow{\be_{W,n+1}\up}W\up\xrightarrow{\gam_{W,n+1}\up}\sig_{\le n}W\up
\]
in $\Kbf(\A)$.
By the assumption, there exists $\phi_{n+1}\in\A(W^{n+1},Y^{n+1})^{-1}$ such that \begin{equation}\label{eq_gfp}
g^{n+1,n+1}=f^{n+1,n+1}\ci b+d_Y^{n,n+1}\ci\phi_n+d_{\A}(\phi_{n+1})
\end{equation}
in $\A$.
Define $g\upp\in Z^0(\Cbf(\A))(W^{n+1}[-(n+1)],X\up)$ and $\varphi\up\in\Cbf(\A)(W^{n+1}[-(n+1)],Y\up)^{-1}$ by
\[
g^{\prime i,j}=\begin{cases}
b & \text{if}\ (i,j)=(n+1,n+1)\\
0 & \text{otherwise}
\end{cases}
\quad,\quad
\varphi^{i,j}=\begin{cases}
\phi_n & \text{if}\ (i,j)=(n+1,n)\\
\phi_{n+1} & \text{if}\ (i,j)=(n+1,n+1)\\
0 & \text{otherwise}
\end{cases}\ .
\]
Then by $(\ref{eq_gfp})$, we have $g\up\ci\be_{W,n+1}\up-f\up\ci g\upp=d\varphi\up$ in $\Cbf(\A)$. Thus $g\up\ci\be_{W,n+1}\up=f\up\ci g\upp$ holds in $\Kbf(\A)$, hence we obtain a morphism of distinguished triangles
$$
\begin{tikzcd}
(\sig_{\le n}W\up)[-1] & W^{n+1}[-(n+1))] & W\up & \sig_{\le n}W\up \\ 
\CoCone f\up & X\up & Y\up & (\CoCone f\up)[1]
\Ar{1-1}{1-2}{"\al\up_{W,n+1}"}
\Ar{1-2}{1-3}{"\be\up_{W,n+1}"}
\Ar{1-3}{1-4}{"\gam\up_{W,n+1}"}
\Ar{2-1}{2-2}{}
\Ar{2-2}{2-3}{"f\up"'}
\Ar{2-3}{2-4}{"f\upp"'}
\Ar{1-1}{2-1}{}
\Ar{1-2}{2-2}{"g\upp"}
\Ar{1-3}{2-3}{"g\up"}
\Ar{1-4}{2-4}{}
\end{tikzcd}
$$
in $\Kbf(\A)$.
Since $f\up$ is an $n$-pullback morphism, we have $\Kbf(\A)(\sig_{\le n}W\up,(\CoCone f\up)[1])\cong\Kbf(\A)((\sig_{\le n}W\up)[-1],\CoCone f\up)=0$. This implies $f\upp\ci g\up=0$, hence there exists $e\up\in Z^0(\Cbf^{n+2}(\A))(W\up,X\up)$ that satisfies $g\up=f\up\ci e\up$ in $\Kbf(\A)$.
Then, since $f\up\ci(e\up\ci\be\up_{W,n+1}-g\upp)=g\up\ci\be\up_{W,n+1}-g\up\ci\be\up_{W,n+1}=0$, morphism $e\up\ci\be\up_{W,n+1}-g\upp$ should factor through $\CoCone f\up$ in $\Kbf(\A)$. By $\Kbf(\A)(W^{n+1}[-(n+1)],\CoCone f\up)=0$, this implies $e\up\ci\be\up_{W,n+1}=g\upp$ in $\Kbf(\A)$.

Let us show that we may modify $e\up$ to moreover satisfy $\ovl{e}^0=\ovl{g}^0$ and $\ovl{e}^{n+1}=\ovl{b}$ in $H^0(\A)$.
By $e\up\ci\be\up_{W,n+1}=g\upp$, there exist $\eta_n\in\A(W^{n+1},X^n)^0$ and $\eta_{n+1}\in\A(W^{n+1},X^{n+1})^{-1}$ such that
\[
b-e^{n+1,n+1}=d_X^{n,n+1}\ci\eta_n+d_{\A}(\eta_{n+1}).
\]
Also, by $g\up=f\up\ci e\up$ in $\Kbf^{n+2}(\A)$, there exists $\psi\up\in\Cbf^{n+2}(\A)(W\up,Y\up)^{-1}$ that satisfies 
$d\psi\up=g\up-f\up\ci e\up$.
In particular we have $(d\psi)^{0,0}=g^{0,0}-e^{0,0}$.
If we define $\mu\up\in\Cbf(\A)(W\up,X\up)^{-1}$ by
\[
\mu^{i,j}=\begin{cases}
\psi^{i,j}&\text{if}\ j=0\\
\eta_n&\text{if}\ (i,j)=(n+1,n)\\
(-1)^{n+1}\eta_{n+1}&\text{if}\ (i,j)=(n+1,n+1)\\
0&\text{otherwise}
\end{cases}
\]
then it satisfies $(d\varphi\up)^{0,0}=(d\mu\up)^{0,0}$.
Then, $e\upp=e\up+d\mu\up$ satisfies $g\up=f\up\ci e\upp$ in $\Kbf^{n+2}(\A)$ and
$\ovl{e'}^0=\ovl{g}^0$, $\ovl{e'}^{n+1}=\ovl{b}$ in $H^0(\A)$ as intended.
\end{proof}

\begin{cor}\label{cor:univ_npb}
Let $Y\up\in\Cbf^{n+2}(\A)$ be a left $n$-exact sequence, and let $f_1\up\co X_1\up\to Y\up$, $f_2\up\co X_2\up\to Y\up$ be $n$-pullback morphisms.
Suppose that there is $b\in Z^0(\A)(X_1^{n+1},X_2^{n+1})$ satisfying $\ovl{f_2}^{n+1}\ci\ovl{b}=\ovl{f_1}^{n+1}$ in $H^0(\A)$. Then, the following holds.
\begin{enumerate}
\item There exists $e\up\in Z^0(\Cbf^{n+2}(\A))(X_1\up,X_2\up)$ satisfying $f_2\up\ci e\up=f_1\up$ in $\Kbf^{n+2}(\A)$, which is unique in $\Kbf^{n+2}(\A)$. Moreover, $e\up$ can be chosen to satisfy $\ovl{e}^0=\id$ and $\ovl{e}^{n+1}=\ovl{b}$ in $H^0(\A)$ too.
\item If $\ovl{b}$ is an isomorphism in $H^0(\A)$, then $e\up$ obtained in {\rm (1)} is an isomorphism in $\Kbf^{n+2}(\A)$.
\item If $X\up$ is right $n$-exact and $\ovl{b}=\id$ $($hence $\ovl{f_1}^{n+1}=\ovl{f_2}^{n+1}$$)$ in $H^0(\A)$, then $e\up$ in {\rm (1)} is an isomorphism in $\Kbf^{1,n,1}(\A)$.
\end{enumerate}
\end{cor}
\begin{proof}
{\rm (1)} This is immediate from Proposition~\ref{prp:univ_npb}. 

{\rm (2)}
If $\ovl{b}$ is an isomorphism, then there exists $b'\in Z^0(\A)(X_2^{n+1},X_1^{n+1})$ that gives the inverse $\ovl{b'}$ of $\ovl{b}$ in $H^0(\A)$. By {\rm (1)} applied to $b'$, there exists $e\upp\in Z^0(\Cbf^{n+2}(\A))(X_2\up,X_1\up)$ such that $f_1\up\ci e\upp=f_2\up$ in $\Kbf^{n+2}(\A)$.
Then, since $f\up\ci(e\upp\ci e\up)=f\up$ holds in $\Kbf^{n+2}(\A)$,
we obtain $e\upp\ci e\up=\id$ in $\Kbf^{n+2}(\A)$ by the uniqueness. Similarly, we have $e\up\ci e\upp=\id$ in $\Kbf^{n+2}(\A)$, hence $e\up$ is an isomorphism in $\Kbf^{n+2}(\A)$.

{\rm (3)} By {\rm (1)} and {\rm (2)}, $e\up$ is an isomorphism in $\Kbf^{n+2}(\A)$ satisfying $\ovl{e}^0=\id$ and $\ovl{e}^{n+1}=\id$ in $H^0(\A)$. By assumption and Remark~\ref{rem:isom_n-ex}, $X\up$ and $Y\up$ are $n$-exact, hence Corollary~\ref{cor:htpy_modif2} shows that $e\up$ is an isomorphism in $\Kbf^{1,n,1}(\A)$.
\end{proof}

\begin{cor}\label{cor:iso-npbm}
Let $Y\up\in\Cbf^{n+2}(\A)$ be a left $n$-exact sequence, and let $f\up\co X\up\to Y\up$ be an $n$-pullback morphism. If $\ovl{f}^{n+1}$ is an isomorphism in $H^0(\A)$, then $f\up$ is an isomorphism in $\Kbf^{n+2}(\A)$.
\end{cor}
\begin{proof}
This follows from Corollary~\ref{cor:univ_npb} {\rm (2)} applied to $f_1\up=f\up$ and $f_2\up=\id$.
\end{proof}

\section{$n$-exact dg-categories}\label{section:n-exact-dg}

As before, $\A$ denotes a strictly connective and strictly additive dg-category.

\normalcolor

\subsection{Definition of $n$-exact dg-categories}
\begin{dfn}\label{dfn:n-exact-dg}
Let $\A$ be as above, and let $\calS\subset\ob(\Cbf^{n+2}(\A))$ be a class consisting of some $n$-exact sequences.  
We call a morphism in $Z^0(\A)$ an $\calS$-\emph{inflation} (respectively, an $\calS$-\emph{deflation}) or simply an \emph{inflation} (resp.\ \emph{deflation}) if it appears as $d_X^{0,1}$ (resp.\ $d_X^{n,n+1}$) for some $X\up\in\calS$.
The pair $(\A,\calS)$ is called an \emph{$n$-exact dg-category} if it satisfies the following conditions.
\begin{enumerate}
\item[($n$-Ex0)] 
$\calS$ is closed by isomorphisms in $\Kbf^{1,n,1}(\A)$.
Moreover, there exists \emph{some} split sequence $U\up\in \Cbf^{n+2}(\A)$ which belongs to $\calS$.  
\normalcolor
\item[($n$-Ex1)] Deflations and inflations are closed under compositions in $\A$.
\item[($n$-Ex2)] For any $X\up\in\calS$ and any $a\in Z^0(\A)(X^0,A)$, there exist $Y\up\in\calS$ and an $n$-pushout morphism $g\up\colon X\up\to Y\up$ such that $Y^0=A$ and $\ovl{g}^0=\ovl{a}$ in $H^0(\A)$.
\item[($n$-Ex2$\op$)]  Dually, for any $Y\up\in\calS$ and any $c\in Z^0(\A)(C,Y^{n+1})$, there exist $X\up\in\calS$ and an $n$-pullback morphism $f\up\colon X\up\to Y\up$ such that $X^{n+1}=C$ and $\ovl{f}^{n+1}=\ovl{c}$ in $H^0(\A)$.
\end{enumerate}
An $n$-exact sequence $X\up$ is called an \emph{admissible $n$-exact sequence} if it belongs to $\calS$. 
\end{dfn}
In the sequel, we will mainly regard $\calS$ as a full subcategory of $\Kbf^{1,n,1}(\A)$. By restricting the functors $s,t$ in Definition~\ref{def:K1n1_st} to the full subcategory $\calS\subset\Kbf^{1,n,1}(\A)$, we have functors $s,t\colon\calS\to H^0(\A)$, which we denote by the same symbols.

\begin{rem}\label{rem:strict_npo}
\begin{enumerate}
\item By Lemma~\ref{lem:4A}, we may choose $g\up$ in {\rm ($n$-Ex2)} to satisfy $g^{0,0}=a$ and $g^{n+1,n+1}=\id$ strictly in $\A$. 
Similarly, $f\up$ in {\rm ($n$-Ex2$\op$)} can be chosen to satisfy $f^{0,0}=\id$ and $f^{n+1,n+1}=c$ strictly in $\A$.
\item  By Corollary~\ref{cor:inflation_replace}, if $a,a'\in Z^0(\A)(A,B)$ satisfies $\ovl{a}=\ovl{a'}$ in $H^0(\A)$, then $a$ is an inflation if and only if $a'$ is an inflation. Similarly for deflations.
\normalcolor
\end{enumerate}
\end{rem}

\begin{prop}\label{prop:closed_equiv}
Let $\A$ be as above, and let $\calS\subset\ob(\Cbf^{n+2}(\A))$ be a class consisting of some $n$-exact sequences. 
Assume that $\calS$ satisfies {\rm ($n$-Ex2)} and {\rm ($n$-Ex2$\op$)}.
Then, the following are equivalent.
\begin{enumerate}
\item $\calS\subset\Kbf^{1,n,1}(\A)$ is closed by isomorphisms.
\item Let $X\up,Y\up\in\Cbf^{n+2}(\A)$ be any pair of  $n$-exact sequences and let $f\up\in Z^0(\Cbf^{n+2}(\A))(X\up,Y\up)$ be a morphism inducing an isomorphism in $\Kbf^{n+2}(\A)$ such that $\ovl{f}^0=\id$ and $\ovl{f}^{n+1}=\id$ in $H^0(\A)$. Then $X\up\in\calS$ if and only if $Y\up\in\calS$.
\end{enumerate}
\end{prop}
\begin{proof}
$(1)\Rightarrow (2)$ This is immediate from Corollary~\ref{cor:htpy_modif2}.

$(2)\Rightarrow (1)$ Suppose that $n$-exact sequences $X\up,Y\up\in\Cbf^{n+2}(\A)$ are isomorphic in $\Kbf^{1,n,1}(\A)$. Assuming {\rm (2)}, it suffices to show that $X\up\in\calS$ implies $Y\up\in\calS$. 
Take $f\up\in Z^0(\Cbf^{n+2}(\A))(X\up,Y\up)$ that gives an isomorphism in $\Kbf^{1,n,1}(\A)$. In particular $\ovl{f}^0=s(f\up)$ and $\ovl{f}^{n+1}=t(f\up)$ are isomorphisms in $H^0(\A)$.

By {\rm ($n$-Ex2)}, there exist $X\upp\in\calS$ and an $n$-pushout morphism $x\up\in Z^0(\Cbf^{n+2}(\A))(X\up,X\upp)$ such that $\ovl{x}^0=\ovl{f}^0$. Then, $x\up$ is an isomorphism in $\Kbf^{n+2}(\A)$ by the dual of Corollary~\ref{cor:iso-npbm}. 
By the dual of Proposition~\ref{prp:univ_npb}, there exists $e\up\in Z^0(\Cbf^{n+2}(\A))(X\upp,Y\up)$ such that $e\up\ci x\up=f\up$ in $\Kbf^{n+2}(\A)$ and $\ovl{e}^0=\id$, $\ovl{e}^{n+1}=\ovl{f}^{n+1}$ in $H^0(\A)$.  Since $f\up$ is an isomorphism in $\Kbf^{n+2}(\A)$, so is $e\up$.
 
 Let $c\in Z^0(\A)(Y^{n+1},X^{n+1})$ be a morphism which gives the inverse $\ovl{c}$ of $\ovl{f}^{n+1}$ in $H^0(\A)$.
 By {\rm ($n$-Ex2$\op$)}, there exist $X\uppp\in\calS$ and an $n$-pullback morphism $x\upp\in Z^0(\Cbf^{n+2}(\A))(X\uppp,X\upp)$ such that $\ovl{x'}^{n+1}=\ovl{c}$ in $H^0(\A)$.
Then, $x\upp$ is an isomorphism in $\Kbf^{n+2}(\A)$ by Corollary~\ref{cor:iso-npbm}.
If we put $f\upp=e\up\ci x\upp\in Z^0(\Cbf^{n+2}(\A))(X\uppp,Y\up)$, then it gives an isomorphism in $\Kbf^{n+2}(\A)$ and satisfies $\ovl{f'}^0=\id$, $\ovl{f'}^{n+1}=\id$ in $H^0(\A)$. Since $X\uppp\in\calS$, it follows $Y\up\in\calS$ by assumption.
\end{proof}

\begin{prop}\label{prop:any_split_belongs_to_S}
Let $(\A,\calS)$ be an $n$-exact dg-category. Then, \emph{any} split sequence belongs to $\calS$.
\end{prop}
\begin{proof}
By {\rm ($n$-Ex0)}, there exists some split sequence $U\up$ which belongs to $\calS$.
Let $A,C\in\ob(\A)$ be any pair of objects. It suffices to show that any split sequence $X\up\in\Cbf^{n+2}(\A)$ with $X^0=A$ and $X^{n+1}=C$ belongs to $\calS$.

Applying {\rm ($n$-Ex2)} to $U\up$ and the morphism $0\in Z^0(\A)(U^0,A)$, we obtain $Y\up\in\calS$ with $Y^0=A$, $Y^{n+1}=U^{n+1}$ and an $n$-pushout morphism $g\up\co U\up\to Y\up$.
By the dual of Proposition~\ref{prop:npb_of_split}, such $Y\up$ is a split sequence. Then, by applying {\rm ($n$-Ex2$\op$)} to $Y\up$ and the morphism $0\in Z^0(\A)(C,U^{n+1})$, we obtain $X\upp\in\calS$ with $X^{\prime 0}=A$, $X^{\prime n+1}=C$ and an $n$-pullback morphism $f\up\co X\upp\to Y\up$. By Proposition~\ref{prop:npb_of_split}, such $X\upp$ is also a split sequence. In particular $X\upp$ is isomorphic to $X\up$ in $\Kbf^{1,n,1}(\A)$ by Corollary~\ref{cor:NtoX}.
Since $\calS\subset\Kbf^{1,n,1}(\A)$ is closed by isomorphisms, we obtain $X\up\in\calS$ as desired.
\end{proof}

\normalcolor

\begin{cor}\label{cor:deflation_examples}
Let $(\A,\calS)$ be an $n$-exact dg-category. For any $C,C'\in\ob(\A)$, the following holds.
\begin{enumerate}
\item Let $c\in Z^0(\A)(C',C)$ be any morphism. If $\ovl{c}$ is an isomorphism in $H^0(\A)$, then $c$ is a deflation.
\item The projection $[0\ 1]\co C'\oplus C\to C$ is a deflation. 
\end{enumerate}
\end{cor}
\begin{proof}
{\rm (1)} This follows from Proposition~\ref{prop:any_split_belongs_to_S}, since $X\up\in\Cbf^{n+2}(\A)$ given by
\[
X^i=\begin{cases}
C'&\text{if}\ i=n\\ 
C&\text{if}\ i=n+1\\
0&\text{otherwise}
\end{cases}
\quad,\quad
d_X^{i,j}=\begin{cases}
c&\text{if}\ (i,j)=(n,n+1)\\ 
0&\text{otherwise}
\end{cases}
\]
is a split sequence.

{\rm (2)} Similarly as in {\rm (1)} the morphism in question appears as a deflation of a split sequence. Indeed, for $n=1$ the split sequence $N\up={}_{C'}N\up_C$ in Definition~\ref{dfn:splitN} satisfies $d_N^{n,n+1}=[0\ 1]$. For $n\ge2$, we may check that the object $X\up\in\Cbf^{n+2}(\A)$ given by
\[
X^i=\begin{cases}
C'&\text{if}\ i=n-1\\ 
C'\oplus C&\text{if}\ i=n\\ 
C&\text{if}\ i=n+1\\
0&\text{otherwise}
\end{cases}
\quad,\quad
d_X^{i,j}=\begin{cases}
\begin{bmatrix}1\\0\end{bmatrix}&\text{if}\ (i,j)=(n-1,n)\\ 
[0\ 1]&\text{if}\ (i,j)=(n,n+1)\\ 
0&\text{otherwise}
\end{cases}
\]
is a split sequence.
\end{proof}

\begin{lem}\label{lem:pushout_inflation}
Let $(\A,\calS)$ be an $n$-exact dg-category. For any $X\up\in\Cbf^{n+2}(\A)$, the following are equivalent.
\begin{enumerate}
\item $X\up\in\calS$.
\item $X\up$ is left $n$-exact and $d_X^{n,n+1}$ is a deflation.
\item $X\up$ is right $n$-exact and $d_X^{0,1}$ is an inflation.
\end{enumerate}
\end{lem}
\begin{proof}
Since $(1)\Leftrightarrow(2)$ can be shown in a dual manner, we only show $(1)\Leftrightarrow(3)$. 

$(1)\Rightarrow(3)$ is obvious. Let us show the converse.
Suppose that $X\up$ is right $n$-exact and $d_X^{0,1}$ is an inflation. By the definition of inflations, there exists $Z\up\in\calS$ such that $d_Z^{0,1}=d_X^{0,1}$. 
By the dual of Corollary~\ref{cor:uniqueness_n-kernel}, there are $g\up\co X\up\to Z\up$ and $h\up\co Z\up\to X\up$ which are inverses to each other in $\Kbf^{n+2}(\A)$ and satisfy $\ovl{g}^0=\ovl{h}^0=\id$ in $H^0(\A)$. In particular $X\up$ is $n$-exact.

Put
$c=g^{n+1,n+1}$. By {\rm ($n$-Ex2$\op$)}, there exist $Y\up\in\calS$ and an $n$-pullback morphism $f\up\co Y\up\to Z\up$ such that $\ovl{f}^{n+1}=\ovl{c}$ in $H^0(\A)$.
By Proposition~\ref{prp:univ_npb}, there exists $e\up\co X\up\to Y\up$ such that $g\up=f\up\ci e\up$ in $\Kbf^{n+2}(\A)$, $\ovl{e}^0=\ovl{g}^0=\id$ and $\ovl{e}^{n+1}=\id$ in $H^0(\A)$.
We claim that $e\up$ is an isomorphism in $\Kbf^{n+2}(\A)$.
In fact, for $u\up=h\up\ci f\up$ we have $u\up\ci e\up=h\up\ci g\up=\id$ in $\Kbf(\A)$.
Besides, we have $f\up\ci(e\up\ci u\up-\id)=g\up\ci h\up\ci f\up-f\up=0$ in $\Kbf(\A)$. Thus $e\up\ci u\up-\id$ should factor through $\CoCone f\up$. Since we have $\Kbf(\A)(Y\up,\CoCone f\up)=0$ by Lemma~\ref{lem:equivalent_n-pbm1}, we obtain $e\up\ci u\up=\id$ in $\Kbf(\A)$. Thus $e\up$ is an isomorphism in $\Kbf^{n+2}(\A)$.

By Corollary~\ref{cor:htpy_modif2}, this shows that $e\up$ is an isomorphism in $\Kbf^{1,n,1}(\A)$, hence $X\up\in\calS$ follows by {\rm ($n$-Ex0)} since $Y\up\in\calS$.
\end{proof}

The following is an analog of \cite[Proposition~4.6]{J}.
\begin{prop}\label{prop:S_is_closed_by_oplus}
Let $(\A,\calS)$ be an $n$-exact dg-category. If $X\up,Y\up\in\calS$, then $X\up\oplus Y\up\in\calS$. Here, $Z\up=X\up\oplus Y\up\in\Cbf^{n+2}(\A)$ is the object defined by
\[
Z^i=X^i\oplus Y^i\quad\text{and}\quad d_Z^{i,j}=d_X^{i,j}\oplus d_Y^{i,j}=\begin{bmatrix}d_X^{i,j}&0\\0&d_Y^{i,j}\end{bmatrix}
\]
for all $i,j\in\bbZ$, using the direct sum in $\A$.
\end{prop}
\begin{proof}
The same proof as that of \cite[Proposition~4.6]{J} works.
Put $X^0=A$ for simplicity. Since the split sequence ${}_AN\up_0$ in Definition~\ref{dfn:splitN} is $n$-exact, $W\up={}_AN\up_0\oplus Y\up$ is also $n$-exact.

When $n\ge 2$, we have $d_W^{n,n+1}=d_Y^{n,n+1}$ by definition.
When $n=1$, we have $d_W^{1,2}=[0\ d_Y^{1,2}]\co A\oplus Y^1\to Y^2$, hence $d_W^{1,2}=d_Y^{1,2}\ci [0\ \id_{Y^1}]$.
By Corollary~\ref{cor:deflation_examples} and {\rm ($n$-Ex1)}, it follows that $d_W^{n,n+1}$ is a deflation also in this case.
Thus in either case we have $W\up\in\calS$ by Lemma~\ref{lem:pushout_inflation}, hence the morphism $d_W^{0,1}=\id_A\oplus d_Y^{0,1}$ is an inflation.
A similar argument shows that $d_X^{0,1}\oplus \id_{Y^1}$ is also an inflation. By {\rm ($n$-Ex1)}, it follows that $d_Z^{0,1}=d_X^{0,1}\oplus d_Y^{0,1}=(d_X^{0,1}\oplus \id_{Y^1})\ci(\id_A\oplus d_Y^{0,1})$ is an inflation. Since $Z\up$ is $n$-exact, Lemma~\ref{lem:pushout_inflation} shows $Z\up\in\calS$.
\end{proof}
\normalcolor

\begin{prop}\label{prop:pushout_inflation}
Let $(\A,\calS)$ be an $n$-exact dg-category, let $f\up\in Z^0(\Cbf^{n+2}(\A))(X\up,Y\up)$ be any morphism, and put $C\up:=\CoCone ((\sigma_{\geq 1}f\up)[1])$.
If $\ovl{f}^0=\id$ in $H^0(\A)$ and $Y\up\in\calS$ and, then the following are equivalent.
\begin{enumerate}
\item $f\up$ is an $n$-pullback morphism.
\item $C\up\in\calS$.
\item $X\up\in\calS$.
\end{enumerate}
\end{prop}
\begin{proof}
$(1)\Leftrightarrow(2)$ Note that $d_C^{n,n+1}=\left[f^{n+1}\ -d_{Y}^{n,n+1}\right]\colon X^{n+1}\ds Y^n\to Y^{n+1}$ can be decomposed as
\[
\begin{bmatrix}f^{n+1}& -d^{n,n+1}_{Y}\end{bmatrix}=
\begin{bmatrix}0& 1\end{bmatrix}\circ \begin{bmatrix}1& 0\\f^{n+1}&1\end{bmatrix}\circ\begin{bmatrix}1& 0\\0&d^{n,n+1}_{Y}\end{bmatrix}\circ\begin{bmatrix}1& 0\\0&-1\end{bmatrix}.
\]
On the right hand side, since $\begin{bmatrix}0& 1\end{bmatrix}\co X^{n+1}\oplus Y^{n+1}\to Y^{n+1}$ is a projection and $\begin{bmatrix}1& 0\\f^{n+1}&1\end{bmatrix}$ , $\begin{bmatrix}1& 0\\0&-1\end{bmatrix}$ are isomorphisms, these are deflations by Corollary~\ref{cor:deflation_examples}.
Also, as $\begin{bmatrix}1& 0\\0&d^{n,n+1}_{Y}\end{bmatrix}\co X^{n+1}\oplus Y^n\to X^{n+1}\oplus Y^{n+1}$ is given by the direct sum of $\id_{X^{n+1}}$ and $d_{Y}^{n,n+1}$, it is also a deflation since $\calS$ is closed under finite direct sum by Proposition~\ref{prop:S_is_closed_by_oplus}. By {\rm ($n$-Ex1)}, their composition $d_C^{n,n+1}$ is also a deflation. 
Thus, by Lemma~\ref{lem:pushout_inflation}, $C\up\in\calS$ holds if and only if $C\up$ is left $n$-exact. By Lemma~\ref{lem:equivalent_n-pbm1}, this is equivalent to that $f\up$ is an $n$-pullback morphism.

$(1)\Leftrightarrow(3)$ Put $c=f^{n+1,n+1}$. Suppose that $f\up$ is an $n$-pullback morphism.
By {\rm ($n$-Ex2$\op$)}, there exist $X\upp\in\calS$ and an $n$-pullback morphism $f\upp\co X\up\to Y\up$ such that $\ovl{f'}^{n+1}=\ovl{c}$. By Corollary~\ref{cor:univ_npb} {\rm (3)} we have $X\up\cong X\upp$ in $\Kbf^{1,n,1}(\A)$, hence $X\up\in\calS$ by {\rm ($n$-Ex0)}.

Conversely, suppose that $X\up$ belongs to $\calS$. 
Similarly as in Definition~\ref{dfn:cocone}, we have a morphism $p\up\in Z^0(\Cbf(\A))(\CoCone f\up,X\up)$ that induces a distinguished triangle
\begin{equation}\label{dt_cocone}
Y\up[-1]\to\CoCone f\up\xrightarrow{p\up}X\up\xrightarrow{f\up}Y\up
\end{equation}
in $\Kbf(\A)$. Let $A\in\ob(\A)$ be any object. Since $X\up$ and $Y\up$ are left $n$-exact, we obtain $\Kbf(\A)(A[-i],\CoCone f\up)\cong\Kbf(\A)(A[-i],X\up)=0$ for all $i\le n$. It remains to show $\Kbf(\A)(A[-(n+1)],\CoCone f\up)=0$.

Let $h\up\in Z^0(\Cbf(\A))(A[-(n+1)],\CoCone f\up)$ be any morphism.
 Put $g\up=p\up\ci h\up\in Z^0(\Cbf(\A))(A[-(n+1)],X\up)$, and $c=g^{n+1,n+1}\in Z^0(\A)(A,X^{n+1})$. By {\rm ($n$-Ex2$\op$)} and Remark~\ref{rem:strict_npo} {\rm (1)}, there exist $W\up\in\calS$ and an $n$-pullback morphism $e\up\co W\up\to X\up$ such that $W^{n+1}=A$ and $e^{n+1,n+1}=c$ in $\A$.
Note that we have $\sig_{\ge n+1}W\up=A[-(n+1)]$ and $e\up\ci \be_{n+1}\up=g\up$ for the morphism $\be_{n+1}\up\co\sig_{\ge n+1}W\up\to W\up$ in Proposition~\ref{prop:alphabetagamma}.

Since $(f\up\ci e\up)\ci\be_{n+1}\up=f\up\ci g\up=f\up\ci p\up\ci h\up=0$ is satisfied in $\Kbf(\A)$, we may apply Lemma~\ref{lem:extend_morph} {\rm (2)} to obtain $f\up\ci e\up=0$ in $\Kbf(\A)$. Thus there exists $\varphi\up\in\Cbf(\A)(W\up,Y\up)^{-1}$ such that $d\varphi\up=f\up\ci e\up$ in $\Cbf(\A)$. Then, $\ovl{\varphi}^1=\ovl{\varphi^{1,0}}$ satisfies $\ovl{f}^0\ci\ovl{e}^0=\ovl{\varphi}^1\ci\ovl{d}_W^0$. Since $\ovl{f}^0=\ovl{e}^0=\id$, this means that $\ovl{d}_W^0$ is a split monomorphism in $H^0(\A)$, hence $W\up$ is a zero object in $\Kbf^{n+2}(\A)$, as shown in Proposition~\ref{prp:characterization_split_seq}.
Thus $e\up=0$ holds in $\Kbf(\A)$, and $g\up=e\up\ci\be_{n+1}\up=0$ follows.
Since $Y\up$ is left $n$-exact, the distinguished triangle $(\ref{dt_cocone})$ induces a left exact sequence
\[ 0\to\Kbf(\A)(A[-(n+1)],\CoCone f\up)\xrightarrow{p\up\ci\blank}\Kbf(\A)(A[-(n+1)],X\up), \]
thus $p\up\ci h\up=g\up=0$ implies $h\up=0$ in $\Kbf(\A)$, as desired.
\end{proof}

\normalcolor

\begin{prop}\label{prop:pullpush_factorization}
Let $(\A,\calS)$ be an $n$-exact dg-category. For any $f\up\in Z^0(\Cbf^{n+2}(\A))(X\up,Y\up)$ with $X\up,Y\up\in\calS$, the following holds.
\begin{enumerate}
\item There exist $Z\up\in\calS$, an $n$-pushout morphism $g\up\in Z^0(\Cbf^{n+2}(\A))(X\up,Z\up)$ and an $n$-pullback morphism $h\up\in Z^0(\Cbf^{n+2}(\A))(Z\up,Y\up)$ that satisfy $f\up=h\up\ci g\up$ in $\calS$ and $\ovl{f}^0=\ovl{g}^0$, $\ovl{f}^{n+1}=\ovl{h}^{n+1}$ in $H^0(\A)$.
\item If $\ovl{f}^0$ is an isomorphism in $H^0(\A)$, then $g\up$ in {\rm (1)} is an isomorphism in $\Kbf^{n+2}(\A)$. 
\item If $\ovl{f}^{n+1}$ is an isomorphism in $H^0(\A)$, then $h\up$ in {\rm (1)} is an isomorphism in $\Kbf^{n+2}(\A)$. 
\item If $\ovl{f}^0$ and $\ovl{f}^{n+1}$ are isomorphisms in $H^0(\A)$, then $f\up$ is an isomorphism in $\Kbf^{n+2}(\A)$.
\end{enumerate}
\end{prop}
\begin{proof}
If {\rm (1)} is shown, then {\rm (3)} is immediate from Corollary~\ref{cor:iso-npbm}. {\rm (2)} can be shown dually. Also, {\rm (4)} follows from {\rm (2)} and {\rm (3)}.

It remains to show {\rm (1)}. When $n=1$, this is essentially shown in \cite{C1}. Indeed, by applying \cite[Lemma~4.12.]{C1} via the equivalence of categories $F\co\Kbf^{1,1,1}(\A)\xrightarrow{\cong}\calH_{3t}(\A)$ in Proposition~\ref{prop:K111=H3t}, we obtain $Z\up\in\calS$ and morphisms $g\up\in Z^0(\Cbf^{3}(\A))(X\up,Z\up)$, $h\up\in Z^0(\Cbf^{3}(\A))(Z\up,Y\up)$ that satisfy $f\up=h\up\ci g\up$ in $\Kbf^{1,1,1}(\A)$ and $\ovl{g}^0=\ovl{f}^0$, $\ovl{g}^{2}=\id$, $\ovl{h}^0=\id$, $\ovl{h}^{2}=\ovl{f}^{2}$ in $H^0(\A)$. Then $g\up$ is a $1$-pushout morphism by the dual of Proposition~\ref{prop:pushout_inflation}, and $h\up$ is a $1$-pullback morphism by Proposition~\ref{prop:pushout_inflation}.

Assume $n\ge2$. By {\rm ($n$-Ex2$\op$)}, there exist $Z\up\in\calS$ and an $n$-pullback morphism $h\up\in Z^0(\Cbf^{n+2}(\A))(Z\up,Y\up)$ such that $\ovl{f}^{n+1}=\ovl{h}^{n+1}$ in $H^0(\A)$.
By Proposition~\ref{prp:univ_npb}, we obtain a morphism $g\up\in Z^0(\Cbf^{n+2}(\A))(X\up,Z\up)$ that satisfies $f\up=h\up\ci g\up$ in $\Kbf^{n+2}(\A)$ and $\ovl{g}^0=\ovl{f}^0$, $\ovl{g}^{n+1}=\id_{X^{n+1}}$ in $H^0(\A)$. Then, by Corollary~\ref{cor:htpy_modif}, it satisfies $f\up=h\up\ci g\up$ in $\calS$. Moreover, $g\up$ is an $n$-pushout morphism by the dual of Proposition~\ref{prop:pushout_inflation}.
\end{proof}

\begin{cor}\label{cor:S_groupoid}
Let $(\A,\calS)$ be an $n$-exact dg-category. Let $f\up\in Z^0(\Cbf^{n+2}(\A))(X\up,Y\up)$ be any morphism with $X\up,Y\up\in\calS$.
If $\ovl{f}^0=\id$ and $\ovl{f}^{n+1}=\id$ in $H^0(\A)$, then $f\up$ is an isomorphism in $\Kbf^{1,n,1}(\A)$.
Moreover, the inverse $g\up$ of $f\up$ in $\Kbf^{1,n,1}(\A)$ also satisfies $\ovl{g}^0=\id$ and $\ovl{g}^{n+1}=\id$ in $H^0(\A)$.
\end{cor}
\begin{proof}
Proposition~\ref{prop:pullpush_factorization} {\rm (4)} shows that $f\up$ is an isomorphism in $\Kbf^{n+2}(\A)$. Thus Corollary~\ref{cor:htpy_modif2} can be applied.
\end{proof}

\subsection{Relation to Chen's exact dg-category}

When $n=1$, Definition~\ref{dfn:n-exact-dg} is equivalent to that of \emph{exact dg-category} in \cite{C1}.
\begin{prop}\label{prop:exact_dg_=_1-exact_dg}
Let $\A$ be a strictly connective and strictly additive dg-category as before, and let $\calS\subset\Kbf^{1,n,1}(\A)$ be a full subcategory. Define $\calS'\subset\calH_{3t}(\A)$ by $\calS'=F(\calS)$, using the equivalence $F$ of categories obtained in Proposition~\ref{prop:K111=H3t}. Then, the following are equivalent.
\begin{enumerate}
\item $(\A,\calS)$ is a $1$-exact dg-category in the sense of Definition~\ref{dfn:n-exact-dg}.
\item $(\A,\calS')$ is an exact dg-category in the sense of \cite[Definition~4.1]{C1}.
\end{enumerate}
\end{prop}
\begin{proof}

As in Remark~\ref{rem:1-exact}, $1$-exact sequences in $\Kbf^{1,1,1}(\A)$ and homotopy exact sequences in $\calH_{3t}(\A)$ correspond through $F$.

$(1)\Rightarrow(2)$ Suppose that $(\A,\calS)$ is a $1$-exact dg-category. It suffices to check that conditions {\rm Ex0}, {\rm Ex1}, {\rm Ex2} and {\rm Ex2$\op$} in \cite[Definition~4.1]{C1} are satisfied by $(\A,\calS')$.

{\rm Ex0} is obvious by Corollary~\ref{cor:deflation_examples} {\rm (1)}. {\rm Ex1} follows from {\rm ($1$-Ex1)}. {\rm Ex2} is by {\rm ($1$-Ex2$\op$)} and Remark~\ref{rem:strict_npo} {\rm (1)}. Dually for {\rm Ex2$\op$}.

$(2)\Rightarrow(1)$ Suppose that $(\A,\calS')$ is an exact dg-category in the sense of \cite[Definition~4.1]{C1}.
By \cite[Remark~4.4]{C1} and \cite[Proposition~4.9 {\rm (a)}]{C1}, the subcategory $\calS'\subset\calH_{3t}(\A)$ is closed under isomorphisms.
Thus, so is $\calS\subset\Kbf^{1,1,1}(\A)$. Moreover, the split sequence
${}_0N_0$ in Definition~\ref{dfn:splitN} belongs to $\calS$ by {\rm Ex0}.
Hence $(\A,\calS)$ satisfies {\rm ($1$-Ex0)}.
{\rm ($1$-Ex1)} follows from Ex1 in \cite[Definition~4.1]{C1} and \cite[Proposition~4.9 {\rm (d)}]{C1}.
{\rm ($1$-Ex2)} follows from the argument in the beginning of Section~4.2 of \cite{C1}. Dually for {\rm ($1$-Ex2$\op$)}.
\end{proof}

\normalcolor

The following notion gives a special case of an $n$-exact dg-category, 
corresponding to an $n$-analogue of the stable dg-category introduced in \cite[Definition~6.1]{C1}.
\normalcolor
\begin{dfn}\label{dfn:n-stable-dg}
Let $\A$ be a strictly connective and strictly additive dg-category, as before. $\A$ is said to be an \emph{$n$-stable dg-category} if it satisfies  the following conditions.
\begin{enumerate}
\item[(St1)] Any closed morphism in $\A$ of degree $0$ has an $n$-kernel and an $n$-cokernel (see Remark~\ref{rem:n-kernel}).
\normalcolor
\item[(St2)] For any $X\up\in \Cbf^{n+2}(\A)$, it is left $n$-exact if and only if it is right $n$-exact.
\end{enumerate}
\end{dfn}

\begin{prop}\label{prop:characterization_n-stable}
Let $\A$ be a strictly connective and strictly additive dg-category, as before. If $\A$ is an $n$-stable dg-category, then it has a natural structure of an $n$-exact dg-category. Indeed, if we define $\calS\subset\Kbf^{1,n,1}(\A)$ to be the class of all $n$-exact sequences, then $(\A,\calS)$ becomes an $n$-exact dg-category. Conversely, if $(\A,\calS)$ is an $n$-exact dg-category in which any closed morphism in $\A$ of degree $0$ is both a deflation and an inflation, then it is an $n$-stable dg-category.
\end{prop}
\begin{proof}
First, suppose that $\A$ is an $n$-stable dg-category, and let $\calS\subset\Kbf^{1,n,1}(\A)$ be the class of all $n$-exact sequences. We  show that $(\A,\calS)$ is an $n$-exact dg-category.

For $n=1$, this follows from Proposition~\ref{prop:exact_dg_=_1-exact_dg} and  \cite[Proposition~6.4]{C1}. 
\normalcolor
Therefore, we assume $n\ge2$.

Obviously, $(\A,\calS)$ satisfies {\rm ($n$-Ex0)} and {\rm ($n$-Ex1)}. Let us show {\rm ($n$-Ex2)}. Let $X\up\in\Cbf^{n+2}(\A)$ be any $n$-exact sequence, and let $a\in Z^0(\A)(X^0,A)$ be any morphism. By {\rm (St1)}, the morphism 
$
\begin{bmatrix}d_X^{0,1}\\ a\end{bmatrix}\co X^0\to X^1\oplus A
$
has an $n$-cokernel. Namely, there exists a right $n$-exact sequence $W\up\in\Cbf^{n+2}(\A)$ such that $d_W^{0,1}=\begin{bmatrix}d_X^{0,1}\\ a\end{bmatrix}$.

By Proposition~\ref{prop:alphabetagamma}, there are sequences of morphisms
\[
(\sig_{\le 1}W\up)[-1]\xrightarrow{\al_{W,2}\up} \sig_{\ge 2}W\up\xrightarrow{\be_{W,2}\up}W\up\xrightarrow{\gam_{W,2}\up}\sig_{\le 1}W\up \]
and
\[
(\sig_{\le 1}X\up)[-1]\xrightarrow{\al_{X,2}\up} \sig_{\ge 2}X\up\xrightarrow{\be_{X,2}\up}X\up\xrightarrow{\gam_{X,2}\up}\sig_{\le 1}X\up, \]
which give distinguished triangles in $\Kbf(\A)$.
Define 
$
e\up\in Z^0(\Cbf(\A))(\sig_{\le 1}W\up, \sig_{\le 1}X\up)
$
by
\[
e^{i,j}=\begin{cases}
\id_{X^0} & \text{if}\ (i,j)=(0,0) \\
[1\ 0] & \text{if}\ (i,j)=(1,1)\\
0 & \text{otherwise}
\end{cases}\ .
\]
By the right $n$-exactness of $W\up$, we have $\Kbf(\A)(W\up,(\sig_{\ge 2}X\up)[1])=0$. In particular $\al\up_{X,2}[1]\ci e\up\ci \gam\up_{X,2}=0$ holds in $\Kbf(\A)$. Thus, there exists $e\upp\in Z^0(\Cbf(\A))(\sig_{\ge 2}W\up,\sig_{\ge 2}X\up)$ such that
\[
\begin{tikzcd}[column sep=2.4em, row sep=2em]
\sig_{\le 1}W\up
  \arrow[r, "\scriptstyle\al_{W,2}\up {[}1{]}"]
  \arrow[d, "\scriptstyle e\up"'] &
(\sig_{\ge2}W\up)[1]
  \arrow[d, "\scriptstyle e\upp {[}1{]}"] \\
\sig_{\le 1}X\up
  \arrow[r, "\scriptstyle\al_{X,2}\up {[}1{]}"'] &
(\sig_{\ge2}X\up)[1]
\end{tikzcd}
\]
becomes commutative in $\Kbf(\A)$. By Corollary~\ref{cor:give_a_morph}, we obtain $f\up\in Z^0(\Cbf(\A))(W\up,X\up)$ such that $\sig_{\le 1}f\up=e\up$ and $\sig_{\ge 2}f\up=e\upp$ in $\tw(\A)$. Put $C\up=\Cone f\up$. 
We use the notation in Definition~\ref{dfn:cone}.
Since $W\up\xrightarrow{f\up}X\up\xrightarrow{v_f\up}C\up\xrightarrow{p_f\up} W\up[1]$ is a distinguished triangle,
\[
\Kbf(\A)(W\up[1],B[i-n])\to\Kbf(\A)(C\up,B[i-n])\to\Kbf(\A)(X,B[i-n])
\]
is exact for any $B\in\ob(\A)$ and $i\in\bbZ$. Moreover, we have $\Cone v_f\up\cong W\up[1]$ in $\Kbf(\A)$. Since $W\up$ and $X\up$ are right $n$-exact, for any $B\in\ob(\A)$ it follows that 
\begin{equation}\label{vf:n-po}
\Kbf(\A)(\Cone v_f\up,B[i-n-1])=0
\end{equation}
for all $i\le n+1$, and that
\begin{equation}\label{c:nexact}
\Kbf(\A)(C\up,B[i-n])=0
\end{equation}
for all $i\le n$.
\normalcolor

Define $M\up\in\Cbf^{n+2}(\A)$ by
\[
M^i=\begin{cases}
X^0&\text{if}\ i=0\\
X^1\oplus X^0&\text{if}\ i=1\\
X^1&\text{if}\ i=2\\
0&\text{otherwise}
\end{cases}\ ,\ \ \ 
d_M^{i,j}=\begin{cases}
\begin{bmatrix}d_X^{0,1}\\-1\end{bmatrix}&\text{if}\ (i,j)=(0,1)\\
[-1\ -d_X^{0,1}] &\text{if}\ (i,j)=(1,2)\\
0&\text{otherwise}
\end{cases}\ .
\]
Then, it is a split sequence. Indeed, $\psi\up\in \Cbf(\A)(M\up,M\up)^{-1}$ given by
\[
\psi^{i,j}=\begin{cases}
[0\ -1] & \text{if}\ (i,j)=(1,0)\\
\begin{bmatrix}-1\\0\end{bmatrix} & \text{if}\ (i,j)=(2,1)\\
0 & \text{otherwise}
\end{cases}
\]
satisfies $d\psi\up=\id_{M\up}$ in $\Cbf(\A)$.

Define $Y\up\in\Cbf^{n+2}(\A)$ by
\[
Y^i=\begin{cases}
A&\text{if}\ i=0\\
W^2&\text{if}\ i=1\\
C^i&\text{if}\ i\ge2\\
0&\text{otherwise}
\end{cases}\ ,\ \ \ 
d_Y^{i,j}=\begin{cases}
[1\ 0]\ci d_C^{i,j}\ci \begin{bmatrix}0\\1\\0\end{bmatrix}&\text{if}\ (i,j)=(1,2) \\
d_C^{i,j}\ci \begin{bmatrix}0\\1\\0\end{bmatrix}&\text{if}\ i=0\ \text{and}\ j\ge2\\
d_C^{i,j}\ci \begin{bmatrix}1\\0\end{bmatrix}&\text{if}\ i=1\\
d_C^{i,j}&\text{if}\ i\ge2\\
0&\text{otherwise}
\end{cases}\ .
\]
Then, there exists $m\up\in Z^0(\Cbf(\A))(M\up,Y\up)$ and an isomorphisms $\iota\up\co C\up\xrightarrow{\cong}\Cone m\up$ in $\Cbf(\A)$.
Explicitly, $m\up$ is a morphism satisfying
$m^{0,0}=-a$ and $m^{i,j}=0$ for all $i\ge3$, 
\normalcolor
and $\iota\up$ is the morphism induced by the isomorphisms in $\A$
\[
\iota^{0,0}=\begin{bmatrix}1&0&0\\0&0&1\\0&1&0\end{bmatrix}\co X^1\oplus A\oplus X^0\xrightarrow{\cong}X^1\oplus X^0\oplus A\ \]
\[
\iota^{1,1}=\begin{bmatrix}0&1\\1&0\end{bmatrix}\co W^2\oplus X^1\xrightarrow{\cong}X^1\oplus W^2,
\]
which are just switching the direct sum components.  
\normalcolor

Let
\[
\begin{tikzcd}[column sep = 30]
Y\up
\arrow[r, shift left=0.6ex, "v_m\up"]
&
\Cone m\up
\arrow[l, shift left=0.6ex, "q_m\up"]
\arrow[r, shift left=0.6ex, "p_m\up"]
&
M\up [1]
\arrow[l, shift left=0.6ex,  "u_m\up"]
\arrow[r, shift left=0.6ex, "s_{-,M}\up"]
&
M\up
\arrow[l, shift left=0.6ex,  "s_{+,M}\up"]
\end{tikzcd}
\]
be the morphisms introduced in Definitions~\ref{dfn:shifts} and~\ref{dfn:cone}.
Note that the left square in the diagram below is commutative in $\Kbf(\A)$, since $M\up$ is a split sequence.
\begin{equation}\label{mortri}
\begin{tikzcd}
M\up & Y\up & \Cone m\up & M\up[1] \\ 
0 & Y\up & Y\up & 0
\Ar{1-1}{1-2}{"m\up"}
\Ar{1-2}{1-3}{"v_m\up"}
\Ar{1-3}{1-4}{"p_m\up"}
\Ar{2-1}{2-2}{}
\Ar{2-2}{2-3}{""'}
\Ar{2-3}{2-4}{""'}
\Ar{1-1}{2-1}{}
\Ar{1-2}{2-2}{equal}
\Ar{1-3}{2-3}{"g\up"}
\Ar{1-4}{2-4}{}
\end{tikzcd}
\end{equation}
In fact, $\eta\up=m\up\ci\psi\up$ satisfies $d\eta\up=m\up$. 
Define $g\up\in \Cbf(\A)(\Cone m\up,Y\up)^{0}$ by
$g\up=\eta\up\ci s_{-,M}\up\ci p_m\up+q_m\up$.
Then, it satisfies $g\up\ci v_m\up=\id_{Y\up}$ and
\[
dg\up=d\eta\up\ci s_{-,M}\up\ci p_m\up+ dq_m\up=m\up\ci s_{-,M}\up\ci p_m\up-m\up\ci s_{-,M}\up\ci p_m\up\ =\ 0, 
\]
hence $g\up\in Z^0(\Cbf(\A))(\Cone m\up,Y\up)$. 
This makes the above $(\ref{mortri})$ a morphism of distinguished triangles in $\Kbf(\A)$. In particular, $g\up$ is an isomorphism in $\Kbf(\A)$.
Moreover, by its definition, we have
\begin{eqnarray*}
g^{0,0}&=&  m^{0,0}\ci\psi^{1,0}\ci p_m^{0,0}+q_m^{0,0} \\
&=&-a\ci \begin{bmatrix}0&-1\end{bmatrix}\ci \begin{bmatrix}1&0&0\\0&1&0\end{bmatrix}+\begin{bmatrix}0&0&1\end{bmatrix}
\ =\ \begin{bmatrix}0&a&1\end{bmatrix}\ \  \co X^1\oplus X^0\oplus A\to A
\end{eqnarray*}
and $g^{n+1,n+1}=q_m^{n+1,n+1}=\id\co X^{n+1}\to X^{n+1}$.

Then, the composite $h\up$ of 
\[
X\up\xrightarrow{v_f\up}C\up\xrightarrow{\iota\up}\Cone m\up\xrightarrow{g\up}Y\up
\]
satisfies $h^{0,0}=a$ and $h^{n+1,n+1}=\id$. As we have seen, for any $B\in\ob(\A)$, we have $(\ref{vf:n-po})$ for all $i\le n+1$, and $(\ref{c:nexact})$ for all $i\le n$. Since
$g\up\ci\iota\up$ is an isomorphism in $\Kbf(\A)$, this implies that, for any $B\in\ob(\A)$,
\[
\Kbf(\A)(\Cone h\up,B[i-n-1])=0
\]
holds for all $i\le n+1$, and
\[
\Kbf(\A)(Y\up,B[i-n])=0
\]
holds for all $i\le n$. Since $Y\up\in\Cbf^{n+2}(\A)$, this means that 
$h\up$ is an $n$-pushout morphism and $Y\up$ is right $n$-exact. By {\rm (St2)}, it follows that $Y\up$ is $n$-exact.
Thus {\rm ($n$-Ex2)} is shown.
{\rm ($n$-Ex2$\op$)} can be shown in a dual manner.


Conversely, let $(\A,\calS)$ be an $n$-exact dg-category in which any closed morphism in $\A$ of degree $0$ is both a deflation and an inflation. 
For any $a\in Z^0(\A)(A,B)$, by assumption it is an inflation, thus there exists $X\up\in\calS$ such that $a=d_X^{0,1}$. In particular $a$ has an $n$-cokernel. Similarly $a$ has an $n$-kernel since it is a deflation. This shows {\rm (St1)}. Let us show {\rm (St2)}. By duality, it is enough to show that any left $n$-exact sequence is $n$-exact. Let $X\up\in\Cbf^{n+2}(\A)$ be any left $n$-exact sequence. By assumption $d_X^{n,n+1}$ is a deflation, and hence there exists $X\upp\in\calS$ such that $d_{X'}^{n,n+1}=d_X^{n,n+1}$.
By Corollary~\ref{cor:uniqueness_n-kernel} we have $X\up\cong X\upp$ in $\Kbf^{n+2}(\A)$, hence $X\up$ is also $n$-exact by Remark~\ref{rem:isom_n-ex}.
\end{proof}

\normalcolor

\section{Induced biadditive functor $\E$ on $H^0(\A)$} \label{section:induced_biadditive}
Throughout this section, let $(\A,\calS)$ denote an $n$-exact dg-category for which $H^0(\A)$ is skeletally small.

\subsection{Functor $H^0(\A)\op\times H^0(\A)\to\Sets$}

In this subsection we will give a functor $\E\co H^0(\A)\op\times H^0(\A)\to\Sets$. This is a consequence of Proposition~\ref{prp:ts-fibration}, which states that the functors $s,t\co\calS\to H^0(\A)$ form a two-sided fibration.

In our main theorem (Theorem~\ref{thm:H0_of_n-ex-dg}), under the additional assumption that $H^i\A(A,B)=0$ for all $A,B\in\ob(\A)$ and $-n<i<0$ (Condition~\ref{condition:H^i=0}), we will show that any $n$-exact dg-category $H^0(\A)$ admits a structure of an $n$-exangulated category introduced in \cite{HLN1}. 

We remark that, when $n=1$ this additional assumption is requiring nothing, and the following is shown by Chen.
\begin{fact}(\cite[Theorem 4.26]{C1})
If $(\A,\calS')$ is an exact dg-category, then $H^0(\A)$ has a structure of an extriangulated category.
\end{fact}
As shown in \cite[Proposition 4.3]{HLN1}, an extriangulated category is nothing but a $1$-exangulated category. Thus, the case $n=1$ of Theorem~\ref{thm:H0_of_n-ex-dg} is already covered by the above result of Chen.
At the same time, some of the methods used in the following argument (for instance, those involving Proposition~\ref{prp:htpy_modif} {\rm (3)}) apply only when $n\ge2$, and certain modifications would be necessary in the argument when $n=1$. To avoid these technicalities, we restrict ourselves to the case $n\ge2$:
\begin{caution}\label{caution_n>1}
In the remainder of this section and Section~\ref{section:n-exangulated}, we assume $n\ge2$.
\end{caution}

\normalcolor

We remark that, for any $X\up,Y\up\in\calS$, a morphism $f\up\in Z^0(\Cbf^{n+2}(\A))(X\up,Y\up)$ is an $n$-pullback morphism if and only if $\ovl{f}^0=\id$ by Proposition~\ref{prop:pushout_inflation}. The following lemma shows that any such morphism is also a $t$-cartesian lift of $\ovl{f}^{n+1}$. For the definition of cartesian lifts, see \cite[Definition 2.2.1]{LR}.
\begin{lem}\label{lem:equiv_npb_lift}
Let $X\up,Y\up$ be in $\calS$ with $X^0=Y^0$.
If a morphism $f\up\in Z^0(\Cbf^{n+2}(\A))(X\up,Y\up)$ satisfies $\ovl{f}^0=\id$ in $H^0(\A)$, then 
$f\up$ is a cartesian lift of $\ovl{f}^{n+1}$ along the functor $t\co\calS\to H^0(\A)$.
\end{lem}
\begin{proof}
Take any $W\up\in\calS$, $g\up\in Z^0(\Cbf(\A))(W\up,Y\up)$ and $b\in Z^0(\A)(W^{n+1},X^{n+1})$ satisfying $\ovl{f}^{n+1}\ci\ovl{b}=\ovl{g}^{n+1}$.
By Proposition~\ref{prp:univ_npb}, there exists $e\up\in Z^0(\Cbf^{n+2}(\A))(W\up,X\up)$ that satisfies $g\up=f\up\ci e\up$ in $\Kbf^{n+2}(\A)$, $\ovl{e}^0=\ovl{g}^0$ and $\ovl{e}^{n+1}=\ovl{b}$ in $H^0(\A)$.
By Corollary~\ref{cor:htpy_modif}, such $e\up$ satisfies $g\up=f\up\ci e\up$ in $\calS$.

Let us show the uniqueness of $e\up$ in $\calS$.
Suppose that $e\upp\in Z^0(\Cbf^{n+2}(\A))(W\up,X\up)$ also satisfies $g\up=f\up\ci e\upp$ in $\calS$ and $\ovl{e}^{\prime n+1}=\ovl{b}$ in $H^0(\A)$. By $g\up=f\up\ci e\upp$ in $\calS$, in particular it satisfies $g\up=f\up\ci e\upp$ in $\Kbf^{n+2}(\A)$ and $\ovl{g}^0=\ovl{f}^0\ci\ovl{e}^{\prime 0}$ in $H^0(\A)$. By Proposition~\ref{prp:univ_npb}, it follows $e\up=e\upp$ in $\Kbf^{n+2}(\A)$. Since $W\up,X\up$ are $n$-exact, we obtain $e\up=e\upp$ in $\calS$ by Corollary~\ref{cor:htpy_modif}.
\end{proof}

\begin{prop}\label{prp:ts-fibration}
The pair of functors $s,t\co\calS\to H^0(\A)$ forms a two-sided fibration.
\end{prop}
\begin{proof}
By the definition in \cite[Definition~2.3.4]{LR}, the pair $s,t\co\calS\to H^0(\A)$ is a two-sided fibration if it satisfies the following conditions.
\begin{enumerate}\renewcommand{\labelenumi}{(\roman{enumi})}
\item For any $X\up\in\calS$ and any $a\in Z^0(\A)(X^0,A)$, there exists an opcartesian lift $g\up\in\calS(X\up,Y\up)$ of $\ovl{a}$ along the functor $s\co\calS\to H^0(\A)$ such that $t(g\up)=\id$ in $H^0(\A)$.
\item For any $Y\up\in\calS$ and any $c\in Z^0(\A)(C,Y^{n+1})$, there exists a cartesian lift $f\up\in\calS(X\up,Y\up)$  of $\ovl{c}$ along the functor $t\co\calS\to H^0(\A)$ such that $s(f\up)=\id$ in $H^0(\A)$.
\item Suppose that $f\up\in\calS(X\up,Y\up)$ is a $t$-cartesian lift of $t(f\up)$ and that $g\up\in\calS(Y\up,Z\up)$ is an $s$-opcartesian lift of $s(g\up)$. Let $f\upp\in\calS(X\up,Y\upp)$ be an $s$-opcartesian lift of $s(g\up\ci f\up)$ and let $g\upp\in\calS(Y\uppp,Z\up)$ be a $t$-cartesian lift of $t(g\up\ci f\up)$. Then, the morphism $h\up\in\calS(Y\upp,Y\uppp)$ induced by the universality of the (op)cartesian lifts, which satisfies $s(h\up)=\id$ and $t(h\up)=\id$, is an isomorphism in $\calS$.
\end{enumerate}
Let us confirm that they are satisfied. {\rm (ii)} follows from {\rm ($n$-Ex2$\op$)} by Lemma~\ref{lem:equiv_npb_lift}. {\rm (i)} can be shown in a dual manner. {\rm (iii)} follows from Corollary~\ref{cor:S_groupoid}.
\end{proof}

As a corollary, we obtain a functor $\E\co H^0(\A)\op\times H^0(\A)\to\Sets$.
Since we will later need to relate this with sequences in $H^0(\A)$ in Section~\ref{section:realization}, we describe the details of this functor in what follows. 
In the rest, if $X\up\in\Cbf^{n+2}(\A)$ satisfies $X^0=A$ and $X^{n+1}=C$, then we will also denote it by ${}_AX\up_C$ when we emphasize the end-terms.

\begin{dfn}\label{dfn:ASC}
Let $A,C\in\ob(\A)$ be any pair of objects. Define categories $\Kbf_{A,C}^{n+2}(\A)$, $\Kbf_{A,C}^{1,n,1}(\A)$ and ${}_A\calS_C$ by the following. 
\begin{enumerate}
\item $\Kbf_{A,C}^{n+2}(\A)$ is defined as a (not full) subcategory of $\Kbf^{n+2}(\A)$. An object $X\up\in\Kbf^{n+2}(\A)$ belongs to $\Kbf_{A,C}^{n+2}(\A)$ if it satisfies $X^0=A$ and $X^{n+1}=C$.
For any ${}_AX\up_C,{}_AY\up_C\in\Kbf^{n+2}_{A,C}(\A)$, a morphism $X\up\to Y\up$ in $\Kbf^{n+2}_{A,C}(\A)$ is a morphism in $\Kbf^{n+2}(\A)$ which can be represented by some $f\up\in Z^0(\Cbf^{n+2}(\A))(X\up,Y\up)$ that satisfies $\ovl{f}^0=\id_A$ and $\ovl{f}^{n+1}=\id_C$.
\item Similarly, $\Kbf_{A,C}^{1,n,1}(\A)$ is defined as a  subcategory of $\Kbf^{1,n,1}(\A)$. An object $X\up\in\Kbf^{1,n,1}(\A)$ belongs to $\Kbf_{A,C}^{1,n,1}(\A)$ if it satisfies $X^0=A$ and $X^{n+1}=C$.
For any ${}_AX\up_C,{}_AY\up_C\in\Kbf^{1,n,1}_{A,C}(\A)$, a morphism $f\up\co X\up\to Y\up$ in $\Kbf^{1,n,1}_{A,C}(\A)$ is a morphism in $\Kbf^{1,n,1}(\A)$ which satisfies  $s(f\up)=\id_A$ and $t(f\up)=\id_C$.
\item ${}_A\calS_C$ is defined to be the full subcategory of $\Kbf^{1,n,1}_{A,C}(\A)$, consisting of objects $X\up\in\Kbf^{1,n,1}(\A)$ satisfying $X\up\in\calS$. 
\end{enumerate}
\end{dfn}

By definition, ${}_A\calS_C$ is nothing but the category fibered over $A,C$ for the two-sided fibration $(s,t)\co\calS\to H^0(\A)\times H^0(\A)$. Note that ${}_A\calS_C$ is a groupoid for each $A,C\in\calS$, by Corollary~\ref{cor:S_groupoid}.

The following definition describes the object $\E(C,A)\in\Sets$ corresponding to $(C,A)\in \ob(H^0(\A)\op\times H^0(\A))$ through the functor $\E\co H^0(\A)\op\times H^0(\A)\to\Sets$. 

\begin{dfn}\label{dfn:equivalence_of_sequences}
Let $A,C\in\ob(\A)$ be any pair of objects.
Define $\E(C,A)$ to be $\ob({}_A\calS_C)/_{\cong}$, the set of isoclasses of objects in the category ${}_A\calS_C$. For any ${}_AX\up_C\in\ob({}_A\calS_C)$, the isoclass to which $X\up$ belongs will be denoted by $\dbr{X\up}\in\E(C,A)$.
Any element $\del=\dbr{X\up}\in\E(C,A)$ is called an $\E$-\emph{extension}, or simply an \emph{extension}.
\end{dfn}

Correspondence of the morphisms is described as follows.
\begin{dfn}\label{dfn:maps_ac}
Let $\ovl{c}\in H^0(\A)(C',C)$ and $\ovl{a}\in H^0(\A)(A,A')$ be arbitrary.
\begin{enumerate}
\item For any $\del=\dbr{X\up}\in\E(C,A)$, there exist $W\up\in{}_A\calS_{C'}$ and an $n$-pullback morphism $f\up\colon W\up\to X\up$ such that $\ovl{f}^{n+1}=\ovl{c}$.
We denote the isoclass $\dbr{W\up}\in E(C',A)$ of such $W\up$ by $\ovl{c}\uas\del$, or simply by $c\uas\del$ in $\E(C',A)$. This gives a well-defined map $\ovl{c}\uas=\E(\ovl{c},A)\co\E(C,A)\to \E(C',A)$.
\item The map $\ovl{a}\sas=\E(C,\ovl{a})\co\E(C,A)\to \E(C,A')$ is given in a dual manner.
For any $\del=\dbr{X\up}\in\E(C,A)$, there exist $Y\up\in{}_{A'}\calS_C$ and an $n$-pushout morphism $f\up\colon X\up\to Y\up$ such that $\ovl{f}^0=\ovl{a}$, and $\ovl{a}\sas\del=a\sas\del$ is defined to be  $\dbr{Y\up}$ in $\E(C,A')$. 
\end{enumerate}
\end{dfn}

Well-definedness of the maps $\E(\ovl{c},A)$ and $\E(C,\ovl{a})$ in Definition~\ref{dfn:maps_ac} follows from Proposition~\ref{prp:ts-fibration}. Furthermore, the following holds.
\begin{cor}\label{cor:bifunctor_E}
With the correspondences of objects and morphisms described in Definitions~\ref{dfn:equivalence_of_sequences} and \ref{dfn:maps_ac}, $\E$ forms a functor $\E\colon H^0(\A)\op\times H^0(\A)\to\Sets$.
\end{cor}
\begin{proof}
This is a formal consequence of Proposition~\ref{prp:ts-fibration}.
\end{proof}

Proposition~\ref{prop:pullpush_factorization} relates morphisms in $\calS$ to \emph{morphisms of extensions}, whose definition is as below. This will be used in the following sections.
\begin{dfn}
(\cite[Definition 2.3.]{HLN1})
For any pair of extensions $\del\in\E(C,A)$ and $\rho\in\E(D,B)$, a \emph{morphism of extensions} $(\abf,\cbf)\colon\del\to\rho$ is a pair of morphisms $\abf\in H^0(\A)(A,B)$ and $\cbf\in  H^0(\A)(C,D)$ that satisfy $\abf\sas\del=\cbf\uas\rho$. 
\end{dfn}

\begin{cor}\label{cor:morph_in_S_to_morph_of_extension}
Any morphism $f\up\in\calS(X\up,Y\up)$ induces a morphism $(s(f\up),t(f\up))\co\del\to\rho$ of extensions.
\end{cor}
\begin{proof}
This is immediate from  Proposition~\ref{prop:pullpush_factorization}.
\end{proof}

\subsection{Biadditive functor $H^0(\A)\op\times H^0(\A)\to\Ab$}

We proceed to show that the functor $\E\colon H^0(\A)\op\times H^0(\A)\to\Sets$ obtained in Corollary~\ref{cor:bifunctor_E} factors through the forgetful functor $\Ab\to\Sets$, to give a biadditive functor $\E\colon H^0(\A)\op\times H^0(\A)\to\Ab$ (Proposition~\ref{prp:biadditive_functor_E}). We denote it by the same symbol $\E$.

Essentially this is a consequence of the fact that $s,t$ are additive functors between additive categories, and can be shown in a way similar to the proof of \cite[Proposition~4.32]{HLN1}.
For the convenience of the reader, we include the proof here.

\begin{dfn}
By Proposition~\ref{prop:any_split_belongs_to_S}, the split sequence ${}_AN_C\up\in\Cbf^{n+2}(\A)$ in Definition~\ref{dfn:splitN} belongs to ${}_A\calS_C$.
We denote the element $\dbr{{}_AN_C\up}\in\E(C,A)$ by ${}_A0_C$, or simply by $0$ in each $\E(C,A)$.
\end{dfn}

\begin{lem}\label{lem:Equivtozero}
Let $A,C\in\ob(\A)$ be any pair of objects. For any $X\up\in {}_A\calS_C$, the following are equivalent.
\begin{enumerate}
\item $\dbr{X\up}=0$ holds in $\E(C,A)$.
\item $X\up$ is a split sequence.
\end{enumerate}
\end{lem}
\begin{proof}
This is immediate from Proposition~\ref{prop:NtoX}.

\end{proof}

\begin{dfn}
For any $A,C\in H^0(\A)$, the element ${}_A0_C\in\E(C,A)$ is called the \emph{split extension} for every $A,C\in\ob(H^0(\A))$.
\end{dfn}

By Corollary~\ref{cor:morph_in_S_to_morph_of_extension}, Lemma~\ref{lem:NtoX} tells us that $(\id_A,0)\colon {}_A0_C\to {}_A\del_C$ and $(0,\id_C)\colon {}_A\del_C\to {}_A0_C$ are morphisms of extensions for any $\del\in\E(C,A)$. Thus $0\uas\del=0$ and $0\sas\del=0$ hold in $\E(C,A)$.
Next, we define the addition in $\E(C,A)$.
\begin{dfn}\label{dfn:Baer_sum}
Let $A,C\in\ob(H^0(\A))$, and let $\del,\del'\in\E(C,A)$ be any pair of elements. Their \emph{Baer sum} $\del+\del'\in\E(C,A)$ is defined as follows:
By definition, there are ${}_AX\up_C,{}_AX\upp_C\in {}_A\calS_C$ such that $\del=\dbr{X\up}$ and $\del'=\dbr{X\upp}$.
By Proposition~\ref{prop:S_is_closed_by_oplus}, we have $X\up\oplus X\upp\in {}_{A\oplus A}\calS_{C\oplus C}$.
We define $\del+\del'$ by
\[ \del+\del'= \nabla^A\sas\Delta_C\uas(\del\oplus\del'),  \]
which is also equal to $\Delta_C\uas\nabla^A\sas(\del\oplus\del')$. Here, $\nabla^A\colon A\oplus A\to A$ is the folding morphism $\nabla^A=\begin{bmatrix}
1&1\end{bmatrix}$, while $\Delta_C\colon C\to C\oplus C$ is the diagonal morphism $\Delta_C=\begin{bmatrix}
1\\1\end{bmatrix}$.
\end{dfn}

\begin{lem}\label{lem:abelian_group_E(C,A)}
For any $A,C\in\ob(H^0(\A))$, the set $\E(C,A)$ equipped with the Baer sum forms an abelian group, whose zero element is given by the split extension.
\end{lem}
\begin{proof}
\noindent \underline{Commutativity.} Let $\del=\dbr{X\up},\del'=\dbr{X\upp}\in\E(C,A)$ be any pair of elements. As a morphism between direct sums, there is the twisting morphism
$t\up=\thh_{X\up,X\upp}
\colon X\up\oplus X\upp\to X\upp\oplus X\up$ given by
\[
t^{i,j}=\begin{cases}
\thh_{X^i,X^{\prime i}}=\begin{bmatrix}0&1\\1&0\end{bmatrix}&\text{if}\ i=j\\
0&\text{otherwise}
\end{cases}
\]
in $\calS$.
Since $t^{0,0}=\thh_{A,A}$ and $t^{n+1,n+1}=\thh_{C,C}$, it follows
\[ (\thh_{A,A})\sas\dbr{X\up\oplus X\upp}=(\thh_{C,C})\uas\dbr{X\upp\oplus X\up}  \]
by Corollary~\ref{cor:morph_in_S_to_morph_of_extension}. Thus we obtain
\begin{eqnarray*}
\del+\del' &=& \Delta_C\uas\nabla^A\sas\dbr{X\up\oplus X\upp}
\ =\ \Delta_C\uas\nabla^A\sas(\tw_{A,A})\sas\dbr{X\up\oplus X\upp}\\
&=&\nabla^A\sas\Delta_C\uas(\tw_{C,C})\uas\dbr{X\upp\oplus X\up}
\ =\ \nabla^A\sas\Delta_C\uas\dbr{X\upp\oplus X\up}
\ =\ \del'+\del.
\end{eqnarray*}

\smallskip

\noindent \underline{Associativity.} Take arbitrary $\del=\dbr{X\up},\del'=\dbr{X\upp},\del''=\dbr{X\uppp}\in\E(C,A)$.
We claim that both $(\del+\del')+\del''$ and $\del+(\del'+\del'')$ are equal to $m\sas u\uas\rho$,
where we put $\rho=\dbr{X\up\oplus X\upp\oplus X\uppp}$ and
\[
m=[\id_A\ \id_A\ \id_A]\colon A\oplus A\oplus A\to A,\quad u=\begin{bmatrix}\id_C\\\id_C\\\id_C\end{bmatrix} \colon C\to C\oplus C\oplus C
\]
in $Z^0(\A)$ for simplicity. Since a similar argument shows $\del+(\del'+\del'')=m\sas u\uas\rho$, let us only show the equation $(\del+\del')+\del''=m\sas u\uas\rho$ in the below.

By the definition of $\del+\del'$, there exist an $n$-pullback morphism $f\up\colon Y\up\to X\up\oplus X\upp$ and an $n$-pushout morphisms $g\up\colon Y\up\to Z\up$ satisfying
$\ovl{f}^{n+1}=\Delta_C$ and $\ovl{g}^0=\nabla^A$ in $H^0(\A)$, which give
\[
\dbr{Y\up}=\Delta_C\uas\dbr{X\up\oplus X\upp}\quad \text{and}\quad 
\dbr{Z\up}=\nabla^A\sas\dbr{Y\up}=\del+\del'.
\]
By Corollary~\ref{cor:morph_in_S_to_morph_of_extension}, morphism $f\up\oplus\id\colon Y\up\oplus X\uppp\to X\up\oplus X\upp\oplus X\uppp$ in $\calS$ induces a morphism of extensions $(\ovl{f}^0\oplus \id,\ovl{f}^{n+1}\oplus \id)$.
This means that we have
$\dbr{Y\up\oplus X\uppp}=\left[\begin{array}{cc}1&0\\1&0\\0&1\end{array}\right]\uas \rho$, hence
$\Delta_C\uas\dbr{Y\up\oplus X\uppp}=u\uas\rho$ follows.
Again by Corollary~\ref{cor:morph_in_S_to_morph_of_extension},
morphism $g\up\oplus \id\colon Y\up\oplus X\uppp\to Z\up\oplus X\uppp$ induces a morphism of extensions $(\ovl{g}^0\oplus \id,\ovl{g}^{n+1}\oplus \id)$.
This means that we have
$\dbr{Z\up\oplus X\uppp}=\left[\begin{array}{ccc}1&1&0\\0&0&1\end{array}\right]\sas \dbr{Y\up\oplus X\uppp}$,
hence
$\nabla^A\sas\dbr{Z\up\oplus X\uppp}=m\sas\dbr{Y\up\oplus X\uppp}$ follows.
Thus we obtain
\[
(\del+\del')+\del''
=\Delta_C\uas\nabla^A\sas\dbr{Z\up\oplus X\uppp}
=\Delta_C\uas m\sas \dbr{Y\up\oplus X\uppp}
=m\sas\Delta_C\uas\dbr{Y\up\oplus X\uppp}
=m\sas u\uas\rho
\]
as desired.

\smallskip

\noindent\underline{Existence of the zero element.}

Let $N\up={}_AN\up_C\in {}_A\calS_C$ be the split sequence in Definition~\ref{dfn:splitN}.
We claim that $0=\dbr{N\up}\in\E(C,A)$ is the zero element. Take any element $\del=\dbr{X\up}\in\E(C,A)$.
By Lemma~\ref{lem:NtoX}, there exists a morphism $f\up\colon N\up\to X\up$ such that $\ovl{f}^0=\id_A$ and $\ovl{f}^{n+1}=0$ hold in $H^0(\A)$. Then, for the morphism $g\up=[\id\ f\up]\colon X\up\oplus N\up\to X\up$ in $\calS$, we have
$\ovl{g}^0=\nabla^A\colon A\oplus A\to A$ and $\ovl{g}^{n+1}=[\id_C\ 0]\colon C\oplus C\to C$ in $H^0(\A)$.
Thus we obtain
\[
\del+0
=\Delta_C\uas\nabla^A\sas\dbr{X\up\oplus N\up}
=\Delta_C\uas [\id_C\ 0]\uas\dbr{X\up}
=\id_C\uas\del
=\del
\]
by Corollary~\ref{cor:morph_in_S_to_morph_of_extension}.

\smallskip

\noindent\underline{Existence of the additive inverse.}

Let $\del=\dbr{X\up}\in\E(C,A)$ be any element. By definition, $(-\id_C)\uas\del=\dbr{Y\up}$ is given by using an $n$-pullback morphism $f\up\colon Y\up\to X\up$ satisfying $\ovl{f}^{n+1}=-\id_C$ in $H^0(\A)$. Then a morphism $g\up=[\id\ f\up]\colon X\up\oplus Y\up\to X\up$ satisfies $\ovl{g}^0=\nabla^A$ and $\ovl{g}^{n+1}=[\id_C\ -\id_C]$ in $H^0(\A)$.
Thus we obtain
\[
\del+(-\id_C)\uas\del
=\Delta_C\uas\nabla^A\sas\dbr{X\up\oplus Y\up}
=\Delta_C\uas [\id_C\ -\id_C]\uas\dbr{X\up}
=0\uas\del
=0
\]
by Corollary~\ref{cor:morph_in_S_to_morph_of_extension}, which means that $(-\id_C)\uas\del=-\del$ holds in $\E(C,A)$.
\end{proof}

\begin{prop}\label{prp:biadditive_functor_E}
The functor $\E$ in Corollary~\ref{cor:bifunctor_E} factors through the forgetful functor $\Ab\to\Sets$, to give a biadditive functor $\E\colon H^0(\A)\op\times H^0(\A)\to\Ab$.
\end{prop}
\begin{proof}
By Lemma~\ref{lem:abelian_group_E(C,A)}, $\E(C,A)$ is an abelian group for any $A,C\in\ob(H^0(\A))$. We will show that the maps $\E(\ovl{c},\ovl{a})\colon \E(C,A)\to\E(C',A')$ are group homomorphisms, for all $\ovl{a}\in H^0(\A)(A,A')$ and $\ovl{c}\in H^0(\A)(C',C)$. 

Let us show that $\E(\ovl{c},A)\colon\E(C,A)\to\E(C',A)$ is a group homomorphism. Take $\del=\dbr{X\up},\del'=\dbr{X\upp}\in\E(C,A)$ arbitrarily. By definition, we have $c\uas\del=\dbr{Y\up}$ and $c\uas\del'=\dbr{Y\upp}$ for some $n$-pullback morphisms $f\up\colon Y\up\to X\up$ and $f\upp\colon Y\upp\to X\upp$ satisfying $\ovl{f}^{n+1}=\ovl{f'}^{n+1}=\ovl{c}$ in $H^0(\A)$.
Since $f\up\oplus f\upp\colon Y\up\oplus Y\upp\to X\up\oplus X\upp$ is a morphism in $\calS$,
by Corollary~\ref{cor:morph_in_S_to_morph_of_extension}
we also have $(\ovl{c}\oplus \ovl{c})\uas\dbr{X\up\oplus X\upp}=\dbr{Y\up\oplus Y\upp}$. Thus we obtain
\begin{eqnarray*}
c\uas(\del+\del')&=&c\uas\Delta_C\uas\nabla^A\sas\dbr{X\up\oplus X\upp}
\ =\ \nabla^A\sas\Delta_{C'}\uas(\ovl{c}\oplus \ovl{c})\uas\dbr{X\up\oplus X\upp}\\
&=& \nabla^A\sas\Delta_{C'}\uas\dbr{Y\up\oplus Y\upp}
\ =\ c\uas\del+c\uas\del',
\end{eqnarray*}
which means that $\E(\ovl{c},A)=\ovl{c}\uas$ is a group homomorphism. In a dual manner we can show that $\E(C',\ovl{a})$
is a group homomorphism, hence so is their composition $\E(\ovl{c},\ovl{a})=\E(C',\ovl{a})\circ\E(\ovl{c},A)$.

The argument so far shows that $\E\colon H^0(\A)\op\times H^0(\A)\to\Ab$ is indeed a functor. It remains to show the biadditivity. As a similar argument works for the second component, we only show that $\E$ is additive in the first component.
Let $\del=\dbr{X\up}\in\E(C,A)$ be any element, and let $\ovl{c_1},\ovl{c_2}\colon C'\to C$ be a pair of morphisms in $H^0(\A)$. 
Let $f_1\up\colon Y_1\up\to X\up$ and $f_2\up\colon Y_2\up\to X\up$ be $n$-pullback morphisms satisfying $\ovl{f_1}^{n+1}=\ovl{c_1}$ and $\ovl{f_2}^{n+1}=\ovl{c_2}$
in $H^0(\A)$, which give $c_1\uas\del=\dbr{Y_1\up}$ and $c_2\uas\del=\dbr{Y_2\up}$.
Since $f_1\up\oplus f_2\up\colon Y_1\oplus Y_2\up\to X\up\oplus X\up$ is a morphism in $\calS$, 
we have $(\ovl{c_1}\oplus \ovl{c_2})\uas\dbr{X\up\oplus X\up}=\dbr{Y_1\up\oplus Y_2\up}$ by Corollary~\ref{cor:morph_in_S_to_morph_of_extension}, hence
\[
c_1\uas\del+c_2\uas\del
=\nabla^A\sas\Delta_{C'}\uas\dbr{Y_1\up\oplus Y_2\up}
=\nabla^A\sas\Delta_{C'}\uas(\ovl{c_1}\oplus \ovl{c_2})\uas\dbr{X\up\oplus X\up}
\]
holds. Since the folding morphism $h\up=\begin{bmatrix}\id_{X\up}&\id_{X\up}\end{bmatrix}\colon X\up\oplus X\up\to X\up$ in $\calS$ satisfies $\ovl{h}^0=\nabla^A$ and $\ovl{h}^{n+1}=\nabla^C$ in $H^0(\A)$,
we obtain
\[
\Delta_{C'}\uas(\ovl{c_1}\oplus \ovl{c_2})\uas\nabla^A\sas\dbr{X\up\oplus X\up}
=\Delta_{C'}\uas(\ovl{c_1}\oplus \ovl{c_2})\uas(\nabla^C)\uas\dbr{X\up}
=(\ovl{c_1}+\ovl{c_2})\uas\del,
\]
thus $(\ovl{c_1}+\ovl{c_2})\uas\del=c_1\uas\del+c_2\uas\del$ as desired.
\normalcolor
\end{proof}

\section{$n$-exangulated structure on $H^0(\A)$} \label{section:n-exangulated}

Let $(\A,\calS)$ be an $n$-exact dg-category as in the previous section. 
We continue to assume that $H^0(\A)$ is skeletally small and $n\ge2$ (see Caution~\ref{caution_n>1}).
Our goal in this section is to show that $H^0(\A)$ admits the
structure of an $n$-exangulated category.
The following condition will be used in the proof of the main theorem, and in several lemmas and propositions leading to it.
For clarity, we state it here once, although the precise minimal
assumptions needed for each statement will be specified
separately in the corresponding lemmas and propositions.

\begin{cond}\label{condition:H^i=0}
$H^i\A(A,B)=0$ holds for all $A,B\in\ob(\A)$ and for all $-n<i<0$.
\end{cond}
Under this condition, our main theorem
(Theorem~\ref{thm:H0_of_n-ex-dg})
shows that $H^0(\A)$ admits the structure of an
$n$-exangulated category.
We remark that Condition~\ref{condition:H^i=0}
is automatically satisfied, for example,
when $\A$ has \emph{$n$-sparse cohomologies},
defined as follows.
\begin{dfn}\label{dfn:with-n-sparse}(\cite[\S4.3.2]{JKM})
Let $\A$ be a connective additive dg-category. $\A$ is said to be \emph{with $n$-sparse cohomologies} if $H^i\A(A,B)=0$ for all $A,B\in\ob(\A)$ whenever $i\notin n\bbZ$.
\end{dfn}

\subsection{Long exact sequence of cohomologies}

For any $k\in\bbZ$ and any object $X\up$ in $\Cbf(\A)$, by Proposition~\ref{prop:alphabetagamma} applied to $\sig_{\le k}X\up$, we have 
a distinguished triangle
\begin{equation}\label{tri_ups_wp}
(\sig_{\le k-1}X\up )[-1]\xrightarrow{\al_k\up}X^k[-k]\xrightarrow{\be_k\up}\sig_{\le k}X\up\xrightarrow{\gam_k\up}\sig_{\le k-1}X\up 
\end{equation}
in $\Kbf(\A)$, which we will use in the following lemma.

\begin{lem}\label{lem:cohom_vanish}
Let $X\up\in\Cbf^{n+2}(\A)$ be any left $n$-exact sequence.
For any $A\in\ob(\A)$ and $u,k\in\bbZ$, put
$S_X^{u,k}(A)=\Kbf(\A)(A[-u],\sig_{\le k}X\up)$.
Then, the following holds.
\begin{enumerate}
\item For any $A\in\ob(\A)$ and any $k\in\bbZ$, distinguished triangle $(\ref{tri_ups_wp})$ induces a long exact sequence
\begin{equation}\label{long_seq_S}
\cdots \xrightarrow{m_{u-1,k-1}} H^{u-k}\A(A,X^k)
\xrightarrow{e_{u,k}} S_X^{u,k}(A)
\xrightarrow{s_{u,k-1}} S_X^{u,k-1}(A)
\xrightarrow{m_{u,k-1}} 
\cdots,
\end{equation}
which moreover makes the diagram
\begin{equation}\label{comm_d}
\begin{tikzcd}[row sep = 14, column sep = 2]
H^v\A(A,X^k) &   & H^v\A(A,X^{k+1}) \\
            & S_X^{v+k,k}(A) &   
\Ar{1-1}{1-3}{"\overline{d}_X^k\ci\blank"}
\Ar{1-1}{2-2}{"e_{v+k,k}"'}
\Ar{2-2}{1-3}{"m_{v+k,k}"'}
\end{tikzcd}
\end{equation}
commutative in $\Ab$ for all $v,k\in\bbZ$.
\item $S_X^{u,k}(A)=0$ holds in either of the following cases.
\begin{enumerate}
\item $k<0$, and $u$ is arbitrary.
\item $u>k$.
\item $u\le n$ and $k\ge n+1$.
\end{enumerate}
\item Let $w\le 0$, and suppose that $A$ satisfies $H^i\A(A,X^j)=0$ for all $j\in\bbZ$ whenever $-n+w<i<w$. 
Then, the following holds.
\begin{enumerate}
\item $S_X^{w+k,k+1}(A)=0$ for $1\le k\le n$. 
\item There is a unique surjection $\varphi^w_A\co H^{w-n}\A(A,X^{n+1})\to S_X^{w,1}(A)$ determined by the commutativity of the diagram
\[
\begin{tikzcd}[row sep = 22]
S_X^{w,n}(A) & &  H^{w-n}\A(A,X^{n+1}) \\
S_X^{w,n-1}(A) &   \cdots & S_X^{w,1}(A) 
\Ar{1-1}{1-3}{"m_{w,n}", "\cong"'}
\Ar{1-1}{2-1}{"s_{w,n-1}"'}
\Ar{2-1}{2-2}{"s_{w,n-2}"', "\cong"}
\Ar{2-2}{2-3}{"s_{w,1}"', "\cong"}
\Ar{1-3}{2-3}{"\varphi^w_A"}
\end{tikzcd}
\]
in which $s_{w,1},\ldots, s_{w,n-2}$  and $m_{w,n}$ are isomorphisms, while $s_{w,n-1}$ is surjective.
\item $S_X^{w+k-n,k-1}(A)=0$ for $k<n$.
\end{enumerate}
\end{enumerate}
\end{lem}
\begin{proof}
{\rm (1)} 
By definition, we have
\begin{equation}\label{eq_albe}
\alpha_k^{i,j}=\begin{cases}
d_X^{i-1,k} & \text{if}\ i\le k\ \text{and}\ j=k\\
0 & \text{otherwise}
\end{cases}\ \ ,\ \ 
\be_k^{i,j}=\begin{cases}
\id_{X^k} & \text{if}\ i=j=k\\
0 & \text{otherwise}
\end{cases}.
\end{equation}
For any $u\in\bbZ$, applying the cohomological functor $\Kbf(\A)(A[-u],\blank)\co \Kbf(\A)\to \Ab$ to $(\ref{tri_ups_wp})$, we obtain a long exact sequence
\footnotesize
\begin{equation}\label{long_seq_to_be_S}
\cdots \to \Kbf(\A)(A[-u],(\sig_{\le k-1}X\up)[-1])
\xrightarrow{\al\up_k\ci\blank}
\Kbf(\A)(A[-u],X^k[-k])
\xrightarrow{\be\up_k\ci\blank}
S_X^{u,k}(A)
\xrightarrow{\gam\up_k\ci\blank}
S_X^{u,k-1}(A)
\to
\cdots.
\end{equation}
\normalsize
We have an isomorphism 
\[
h_{u,k}\colon\Kbf(\A)(A[-u],X^k[-k])\xrightarrow{\cong}H^{u-k}\A(A,X^k)\, ;\, f\up\mapsto\ovl{f^{u,k}}.
\]
Put $s_{u,k-1}=(\gam\up_k\ci\blank)\co S_X^{u,k}(A)\to  S_X^{u,k-1}(A)$, and define $e_{u,k},m_{u-1,k-1}$ to be the morphisms which make the diagram
\[
\begin{tikzpicture}[>=stealth]
\node (1) at (-4.6,0.8) {$S_X^{u-1,k-1}(A)$};
\node (2) at (0.8,0.8) {$H^{u-k}\A(A,X^k)$};
\node (3) at (4.6,0.8) {$S_X^{u,k}(A)$};
\node (4) at (-4.6,-0.8) {$\Kbf(\A)(A[-u],(\sig_{\le k-1}X\up)[-1])$};
\node (5) at (0.8,-0.8) {$\Kbf(\A)(A[-u],X^k[-k])$};
\draw[->] (1) -- node[above,font=\scriptsize] {$m_{u,k}$} (2);
\draw[->] (1) -- node[left,font=\scriptsize] {$[-1]$} node[right,font=\scriptsize] {$\cong$} (4);
\draw[->] (4) -- node[below,font=\scriptsize] {$\al\up_k\ci\blank$} (5);
\draw[->] (5) -- node[right,font=\scriptsize] {$h_{u,k}$} node[left,font=\scriptsize] {$\cong$} (2);
\draw[->] (2) -- node[above,font=\scriptsize] {$e_{u,k}$} (3);
\draw[->] (5) -- node[right,font=\scriptsize] {$\ \ \ \be\up_k\ci\blank$} (3);
\end{tikzpicture}
\]
commutative. 
Then, by the exactness of $(\ref{long_seq_to_be_S})$, we obtain the long exact sequence $(\ref{long_seq_S})$.

It remains to show the commutativity of $(\ref{comm_d})$. By the definitions of $e_{u,k}$ and $m_{u,k}$, it suffices to show that
\begin{equation}\label{comm_6}
\begin{tikzpicture}[>=stealth]
\node (1) at (-2.6,1.5) {$H^v\A(A,X^k)$};
\node (2) at (2.6,1.5) {$H^v\A(A,X^{k+1})$};
\node (3) at (-3.2,0) {$\Kbf(\A)(A[-(v+k)],X^k[-k])$};
\node (4) at (3.2,0) {$\Kbf(\A)(A[-(v+k+1)],X^{k+1}[-(k+1)])$};
\node (5) at (-2.8,-1.5) {$S_X^{v+k,k}(A)$};
\node (6) at (2.8,-1.5) {$\Kbf(\A)(A[-(v+k+1)],(\sig_{\le k}X\up)[-1])$};
\draw[->] (1) -- node[above,font=\scriptsize] {$\ovl{d}_X^k\ci\blank$} (2);
\draw[->] (3) -- node[left,font=\scriptsize] {$h_{v+k,k}$} node[right,font=\scriptsize] {$\cong$} (1);
\draw[->] (4) -- node[right,font=\scriptsize] {$h_{v+k+1,k+1}$} node[left,font=\scriptsize] {$\cong$} (2);
\draw[->] (3) -- node[left,font=\scriptsize] {$\be\up_k\ci\blank$}   (5);
\draw[->] (6) -- node[right,font=\scriptsize] {$\al\up_{k+1}\ci\blank$}   (4);
\draw[->] (5) -- node[below,font=\scriptsize] {$[-1]$} node[above,font=\scriptsize] {$\cong$} (6);
\end{tikzpicture}
\end{equation}
is commutative. Let $f\up\in Z^0(\Cbf(\A))(A[-(v+k)],X^k[-k])$ be any element. Then, $g\up=\al\up_{k+1}\ci((\be_k\up\ci f\up)[-1])$ satisfies
\[
g^{i,j}=\begin{cases}
d_X^{k,k+1}\ci f^{v+k,k} & \text{if}\ (i,j)=(v+k+1,k+1)\\
0 & \text{otherwise}
\end{cases}
\]
by $(\ref{eq_albe})$.
Thus $h_{v+k+1,k+1}(g\up)=\ovl{d_X^{k,k+1}\ci f^{v+k,k}}=\ovl{d}_X^k\ci h_{v+k,k}(f\up)$ holds, which shows the commutativity of $(\ref{comm_6})$.

{\rm (2)}  {\rm (a)} This case is obvious, since $\sig_{\le k}X\up=0$ if $k<0$. {\rm (b)} If $u>k$, since $\A$ is strictly connective, we have $\Cbf(\A)(A[-u],\sig_{\le k}X\up)^0=0$. This implies $S_X^{u,k}(A)=0$.
{\rm (c)} If $k\ge n+1$, then we have $\sig_{\le k}X\up=X\up$, hence $S_X^{u,k}(A)=\Kbf(\A)(A[-u],X\up)$. This is equal to $0$ whenever $u\le n$, since $X\up$ is left $n$-exact.

{\rm (3)} By {\rm (1)}, the sequence
\[
H^{w+k-p-1}(A,X^{p+1})\to S_X^{w+k,p+1}(A)
\xrightarrow{s_{w+k,p}} S_X^{w+k,p}(A)
\to H^{w+k-p}(A,X^{p+1})
\]
is exact. By the assumption for $H^i\A(A,\blank)$ and this exactness, it follows that $s_{w+k,p}$ is an isomorphism if $k<p<k+n-1$, and that $s_{w+k,k+n-1}$ is surjective. Thus, for each $1\le k\le n$, the composition (which we define to be the identity map if $k=n$)
\[
s_{w+k,k+1}\ci s_{w+k,k+2}\ci\cdots\ci s_{w+k,n}\co S_X^{w+k,n+1}(A)\to S_X^{w+k,k+1}(A)
\]
is surjective. Since $S_X^{w+k,n+1}(A)=0$ by {\rm (2)(c)}, we obtain $S_X^{w+k,k+1}(A)=0$. Thus {\rm (a)} is shown.
For {\rm (b)}, since we have already seen that $s_{w,1},\ldots, s_{w,n-2}$ are isomorphisms and $s_{w,n-1}$ is surjective, 
it remains to show that $m_{w,n}$ is an isomorphism. This follows from the exactness of 
\[ S_X^{w,n+1}(A)\to S_X^{w,n}(A)\xrightarrow{m_{w,n}}H^{w-n}\A(A,X^{n+1})\to S_X^{w+1,n+1}(A) \]
obtained in {\rm (1)}, since $S_X^{w,n+1}(A)=S_X^{w+1,n+1}(A)=0$ by {\rm (2)(c)}.

Let us show {\rm (c)}. For $k\le 0$, it is immediate from {\rm (2)(a)}. For $0<k<n$, it can be shown in a similar way to {\rm (a)}. Indeed, by the exactness of $(\ref{long_seq_S})$ obtained in {\rm (1)} and by the assumption for $H^i\A(A,\blank)$, it follows that $s_{w-n+k,p}$ is an isomorphism if $-n+k<p<k-1$, and that $s_{w-n+k,-n+k}$ is injective. Thus, for any $0< k< n$, the composition
\[
s_{w-n+k,-1}\ci s_{w-n+k,0}\ci\cdots\ci s_{w-n+k,k-2}\co S_X^{w-n+k,k-1}(A)\to S_X^{w-n+k,-1}(A)
\]
is injective. Since $S_X^{w-n+k,-1}(A)=0$ by {\rm (2)(a)}, it follows that $S_X^{w-n+k,k-1}(A)=0$.
\end{proof}

\begin{lem}\label{lem:long_ex_seq}
Let $A\in\ob(\A)$ be any object, and let $X\up\in\Cbf^{n+2}(\A)$ be any left $n$-exact sequence. We continue to use the notation in Lemma~\ref{lem:cohom_vanish}.
Let $w$ be an integer such that $w\le-n$, or $w=0$.
Assume that $A$ satisfies $H^i\A(A,X^j)=0$
\begin{itemize}
\item for all $j\in\bbZ$ and all $-n+w<i<w$, and
\item for all $j\in\bbZ$ and all $w<i<n+w$ (which is always satisfied if $w=0$ since $\A$ is connective).
\end{itemize}
Then, the following holds.
\begin{enumerate}
\item The sequence
\begin{equation}\label{seq_for_exact_s}
H^w\A(A,X^0)
\xrightarrow{\ovl{d}_X^0\ci\blank}H^w\A(A,X^1)
\xrightarrow{\ovl{d}_X^1\ci\blank}\cdots
\xrightarrow{\ovl{d}_X^n\ci\blank}H^w\A(A,X^{n+1})
\end{equation}
is exact. 
\item If we define $\vartheta_A^w$ to be the unique morphism that makes the diagram 
\[
\begin{tikzcd}[row sep = 22]
H^{w-n}\A(A,X^{n+1}) & H^w\A(A,X^0)  \\
S_X^{w,1}(A) &  S_X^{w,0}(A) 
\Ar{1-1}{1-2}{"\vartheta_A^w"}
\Ar{1-1}{2-1}{"\varphi_A^w"'}
\Ar{2-1}{2-2}{"s_{w,0}"'}
\Ar{1-2}{2-2}{"e_{w,0}","\cong"'}
\end{tikzcd}
\]
commutative, in which $\varphi_A^w$ is the surjection obtained in Lemma~\ref{lem:cohom_vanish} {\rm (3)(b)}, then 
\[
H^{w-n}\A(A,X^{n+1})
\xrightarrow{\vartheta_A^w}
H^w\A(A,X^0)
\xrightarrow{\ovl{d}_X^0\ci\blank}H^w\A(A,X^1)
\]
is exact.
\end{enumerate}
\end{lem}
\begin{proof}
{\rm (1)} By Lemma~\ref{lem:cohom_vanish} {\rm (1)},
\[
\begin{tikzpicture}
\matrix (m) [matrix of math nodes,
  row sep=3em,
  column sep=4em] {
  H^{w-1}\A(A,X^{k+1}) & S_X^{w+k,k+1}(A) & S_X^{w+k,k}(A) \\
  H^w\A(A,X^{k+1}) & S_X^{w+k+1,k+1}(A) & S_X^{w+k+1,k}(A) \\
};

\draw[->] (m-1-1) -- node[midway, above, font=\scriptsize] {$e_{w+k,k+1}$} (m-1-2);
\draw[->] (m-1-2) -- node[midway, above, font=\scriptsize] {$s_{w+k,k}$} (m-1-3);

\draw[->] (m-2-1) -- node[midway, above, font=\scriptsize] {$e_{w+k+1,k+1}$} (m-2-2);
\draw[->] (m-2-2) -- node[midway, above, font=\scriptsize] {$s_{w+k+1,k}$} (m-2-3);

\draw[->]
  (m-1-3)
  .. controls +(1,-1) and +(-1,1)
  .. node[midway, above, font=\scriptsize] {$m_{w+k,k}$}(m-2-1);
  \end{tikzpicture}
\]
is exact for any $k\in\bbZ$. We have $S_X^{w+k,k+1}(A)=0$ for $1\le k\le n$ by Lemma~\ref{lem:cohom_vanish} {\rm (3)(a)}. 
Also, we have $S_X^{w+k+1,k}(A)=0$ for all $k<n-1$. Indeed, this is by Lemma~\ref{lem:cohom_vanish} {\rm (2)(b)} if $w=0$, and by Lemma~\ref{lem:cohom_vanish} {\rm (3)(c)} applied to $w+n$ if $w\le-n$.

Hence
\begin{equation}\label{to_splice_1_n}
0\to S_X^{w+k,k}(A)
\xrightarrow{m_{w+k,k}}H^w\A(A,X^{k+1})
\xrightarrow{e_{w+k+1,k+1}}S_X^{w+k+1,k+1}(A)
\to0
\end{equation}
is exact for all $k=1,2,\ldots,n-2$. 

For $k=n-1$,
\[
0\to S_X^{w+n-1,n-1}(A)
\xrightarrow{m_{w+n-1,n-1}}H^w\A(A,X^n)
\xrightarrow{e_{w+n,n}}S_X^{w+n,n}(A)
\xrightarrow{s_{w+n,n-1}}S_X^{w+n,n-1}(A)
\]
is exact. Moreover, $m_{w+n,n}\co S_X^{w+n,n}(A)\to H^w\A(A,X^{n+1})$ is injective by Lemma~\ref{lem:cohom_vanish} {\rm (2)(c)} if $w=0$, while it is an isomorphism by by Lemma~\ref{lem:cohom_vanish} {\rm (3)(b)} applied to $w+n$ if $w\le-n$. In either case, this shows that 
\begin{equation}\label{to_splice_n-1}
0\to S_X^{w+n-1,n-1}(A)
\xrightarrow{m_{w+n-1,n-1}}H^w\A(A,X^n)
\xrightarrow{\ovl{d}_X^n\ci\blank}H^w\A(A,X^{n+1})
\end{equation}
is exact, by the commutativity of $(\ref{comm_d})$ shown in Lemma~\ref{lem:cohom_vanish} {\rm (1)}.

For $k=0$, the above sequence becomes 
\begin{equation}\label{to_splice_0}
0\to S_X^{w,1}(A)
\xrightarrow{s_{w,1}} S_X^{w,0}(A)
\xrightarrow{m_{w,0}}H^w\A(A,X^1)
\xrightarrow{e_{w+1,1}}S_X^{w+1,1}(A)
\to 0,
\end{equation}
in which the exactness at the leftmost term follows from $H^{w-1}\A(A,X^1)=0$.
Similarly for $k=-1$, the sequence
$S_X^{w-1,-1}(A)
\to H^w\A(A,X^0)
\xrightarrow{e_{w,0}}S_X^{w,0}(A)
\to 0$ becomes exact.
Since $S_X^{w-1,-1}(A)=0$ by Lemma~\ref{lem:cohom_vanish} {\rm (2)(a)}, this exactness shows that $e_{w,0}$ is an isomorphism.
Splicing up $(\ref{to_splice_1_n})$, $(\ref{to_splice_n-1})$ and $(\ref{to_splice_0})$, we obtain the exact sequence $(\ref{seq_for_exact_s})$.

{\rm (2)} By the definition of $\vartheta_A^w$ and Lemma~\ref{lem:cohom_vanish} {\rm (1)}, the diagram 
\[
\begin{tikzcd}[row sep = 22]
&H^{w-n}\A(A,X^{n+1}) & H^w\A(A,X^0) & \\
0&S_X^{w,1}(A) &  S_X^{w,0}(A) & H^w\A(A,X^1) 
\Ar{1-2}{1-3}{"\vartheta_A^w"}
\Ar{1-2}{2-2}{"\varphi_A^w"'}
\Ar{2-2}{2-3}{"s_{w,0}"'}
\Ar{1-3}{2-3}{"e_{w,0}","\cong"'}
\Ar{1-3}{2-4}{"\ovl{d}_X^0\ci\blank"}
\Ar{2-1}{2-2}{}
\Ar{2-3}{2-4}{"m_{w,0}"'}
\end{tikzcd}
\]
is commutative.
By the exactness of the bottom row and the surjectivity of $\varphi_A^w$, the sequence
$H^{w-n}\A(A,X^{n+1})\xrightarrow{\vartheta_A^w}H^w\A(A,X^0)\xrightarrow{\ovl{d}_X^0\ci\blank}H^w\A(A,X^1)$ becomes exact.
\end{proof}

\begin{prop}\label{prp:long_ex_seq}
Let $A\in\ob(\A)$ be any object, and let $X\up\in\Cbf^{n+2}(\A)$ be any left $n$-exact sequence. We continue to use the notation in Lemma~\ref{lem:cohom_vanish}. 
The following holds.
\begin{enumerate}
\item Assume that $A$ satisfies $H^i\A(A,X^j)=0$ for all $j\in\bbZ$ and all $-n<i<0$. This assumption is satisfied, for instance, under Condition~\ref{condition:H^i=0}.
Then, the sequence
\[
H^{-n}\A(A,X^{n+1})
\xrightarrow{\vartheta_A^0}
H^0\A(A,X^0)
\xrightarrow{\ovl{d}_X^0\ci\blank}H^0\A(A,X^1)
\xrightarrow{\ovl{d}_X^1\ci\blank}\cdots
\xrightarrow{\ovl{d}_X^n\ci\blank}H^0\A(A,X^{n+1})
\]
is exact. 
\item Assume that $A$ satisfies $H^i\A(A,X^j)=0$ for all $j\in\bbZ$ and all $i\in\bbZ\setminus n\bbZ$. Then we have the following long exact sequence.
\[
\begin{tikzpicture}
\matrix (m) [matrix of math nodes,
  row sep=3em,
  column sep=3em] {
   \qquad\qquad\cdots\ \  & H^{-2n}\A(A,X^1) & \cdots & H^{-2n}\A(A,X^{n+1}) \\
   H^{-n}\A(A,X^0) & H^{-n}\A(A,X^1) & \cdots & H^{-n}\A(A,X^{n+1}) \\
   H^0\A(A,X^0) & H^0\A(A,X^1) & \cdots & H^0\A(A,X^{n+1}) \\
};

\draw[->] (m-1-1) -- node[midway, above, font=\scriptsize] {$\ovl{d}_X^0\ci\blank$} (m-1-2);
\draw[->] (m-1-2) -- node[midway, above, font=\scriptsize] {$\ovl{d}_X^1\ci\blank$} (m-1-3);
\draw[->] (m-1-3) -- node[midway, above, font=\scriptsize] {$\ovl{d}_X^n\ci\blank$} (m-1-4);
\draw[->] (m-2-1) -- node[midway, above, font=\scriptsize] {$\ovl{d}_X^0\ci\blank$} (m-2-2);
\draw[->] (m-2-2) -- node[midway, above, font=\scriptsize] {$\ovl{d}_X^1\ci\blank$} (m-2-3);
\draw[->] (m-2-3) -- node[midway, above, font=\scriptsize] {$\ovl{d}_X^n\ci\blank$} (m-2-4);
\draw[->] (m-3-1) -- node[midway, above, font=\scriptsize] {$\ovl{d}_X^0\ci\blank$} (m-3-2);
\draw[->] (m-3-2) -- node[midway, above, font=\scriptsize] {$\ovl{d}_X^1\ci\blank$} (m-3-3);
\draw[->] (m-3-3) -- node[midway, above, font=\scriptsize] {$\ovl{d}_X^n\ci\blank$} (m-3-4);
\draw[->]
  (m-1-4)
  .. controls +(1,-1) and +(-1,1)
  .. node[midway, above, font=\scriptsize] {$\vartheta_A^{-n}$}(m-2-1);
\draw[->]
  (m-2-4)
  .. controls +(1,-1) and +(-1,1)
  .. node[midway, above, font=\scriptsize] {$\vartheta_A^0$}(m-3-1);
\end{tikzpicture}
\]
\end{enumerate}
\end{prop}
\begin{proof}
{\rm (1)} Note that $H^i\A(A,X^j)=0$ holds for all $j\in\bbZ$ and for $i>0$ by the connectivity of $\A$. Thus, this follows from Lemma~\ref{lem:long_ex_seq} for $w=0$. 

{\rm (2)} For any $q\in\bbZ_{\ge0}$, by Lemma~\ref{lem:long_ex_seq} applied to $w=-qn$, the sequence
\[
H^{-(q+1)n}\A(A,X^{n+1})
\xrightarrow{\vartheta_A^{-qn}}H^{-qn}\A(A,X^0)
\xrightarrow{\ovl{d}_X^0\ci\blank}
\cdots
\xrightarrow{\ovl{d}_X^n\ci\blank}H^{-qn}\A(A,X^{n+1})
\]
becomes exact. Thus it remains to show that 
\[
H^{-(q+1)n}\A(A,X^n)
\xrightarrow{\ovl{d}_X^n\ci\blank}H^{-(q+1)n}\A(A,X^{n+1})
\xrightarrow{\vartheta_A^{-qn}}H^{-qn}\A(A,X^0)
\]
is exact for $q\in\bbZ_{\ge0}$.
By the definitions of $\varphi_A^w,\vartheta_A^w$ for $w=-qn$, the diagram
\[
\begin{tikzcd}[row sep = 22]
H^{-(q+1)n}\A(A,X^n) & S_X^{-qn,n}(A) &S_X^{-qn,n-1}(A) \\
H^{-(q+1)n}\A(A,X^n) & H^{-(q+1)n}\A(A,X^{n+1}) & H^{-qn}\A(A,X^0) 
\Ar{1-1}{1-2}{"e_{-qn,n}"}
\Ar{1-2}{1-3}{"s_{-qn,n-1}"}
\Ar{1-1}{2-1}{equal}
\Ar{1-2}{2-2}{"m_{-qn,n}"',"\cong"}
\Ar{2-3}{1-3}{"(e_{-qn,0})^{-1}\ci s_{-qn,0}\ci \omega^{-1}"'}
\Ar{2-1}{2-2}{"\ovl{d}_X^n\ci\blank"'}
\Ar{2-2}{2-3}{"\vartheta_A^{-qn}"'}
\end{tikzcd}
\]
is commutative. Here we put $\om=s_{-qn,1}\ci s_{-qn,2}\ci\cdots\ci s_{-qn,n-2}$, which is an isomorphism.
Since $(e_{-qn,0})^{-1}\ci s_{-qn,0}\ci \omega^{-1}$ is injective and the top row is exact by Lemma~\ref{lem:cohom_vanish}, it follows that the bottom row is also exact.
\end{proof}

\subsection{Comparison with complexes in $H^0(\A)$} \label{subsection:Comparison}
For use in the next subsection, we examine properties of the functor $\ovl{(\blank)}\co\Kbf(\A)\to\Kbff(H^0(\A))$ and related functors in detail.

\begin{lem}\label{lem:comparison_H_plus}
Assume that $\A$ satisfies $H^i\A(A,B)=0$ for all $A,B\in\ob(\A)$ and $-n<i<0$.
Then, the functor $\ovl{(\blank)}\co\Kbf_{[1,n]}(\A)\to\Kbff_{[1,n]}(H^0(\A))$ is fully faithful.
\end{lem}
\begin{proof}
For simplicity, for each $X\up,Y\up\in\Cbf(\A)$, let us denote the map
\[
\ovl{(\blank)}\co\Kbf(\A)(X\up,Y\up)\to\Kbff(H^0(\A))(\ovl{X}\up,\ovl{Y}\up)\,;\,f\up\mapsto\ovl{f}\up
\]
by $F_{X\up,Y\up}\co M_{(X\up,Y\up)}\to\ovl{M}_{(X\up,Y\up)}$.
We show that $F_{X\up,Y\up}$ is an isomorphism for all $X\up,Y\up\in\Cbf_{[1,n]}(\A)$ in three steps, by successively weakening the assumptions on $X\up,Y\up$ as follows.
\begin{enumerate}
\item The case where $X\up=A[-i],Y\up=B[-j]$ for some $A,B\in\ob(\A)$ and $i,j\in\{1,\ldots,n\}$.
\item The case where $X\up\in\Cbf_{[1,n]}(\A)$ is arbitrary,  and $Y\up=B[-j]$ for some $B\in\ob(\A)$, $j\in\{1,\ldots,n\}$.
\item The case where $X\up,Y\up\in\Cbf_{[1,n]}(\A)$ are arbitrary.
\end{enumerate}
We proceed to prove each of these cases.

\smallskip

{\rm (1)} Assume that $X\up=A[-i]$ and $Y\up=B[-j]$ hold for some $A,B\in\ob(\A)$ and $i,j\in\{1,\ldots,n\}$. If $i\ne j$ we have $\ovl{M}_{(A[-i],B[-j])}=0$, and $M_{(A[-i],B[-j])}=\Kbf(\A)(A[-i],B[-j])\cong H^{i-j}(\A)(A,B)=0$ by assumption, hence $F_{A[-i],B[-j]}$ is trivially an isomorphism. If $i=j$, note that we have a commutative diagram
\[
\begin{tikzcd}[column sep = 20pt]
M_{(A[-i],B[-i])} & M_{(A,B)} \\
\ovl{M}_{(A[-i],B[-i])} & \ovl{M}_{(A,B)}
\Ar{1-1}{1-2}{"{[i]}", "\cong"'}
\Ar{1-1}{2-1}{"F_{A[-i],B[-i]}"'}
\Ar{1-2}{2-2}{"F_{A,B}"}
\Ar{2-1}{2-2}{"{[i]}"', "\cong"}
\end{tikzcd}
\]
in which the horizontal arrows are isomorphisms induced by shifts. Since $F_{A,B}$ is an isomorphism by Lemma~\ref{lem:comparison_H}, so is $F_{A[-i],B[-i]}$.

{\rm (2)} Let $X\up\in\Cbf_{[1,n]}(\A)$ be arbitrary. It suffices to show that $F_{\sig_{\le k}X\up,Y\up}$ is an isomorphism for any $Y\up=B[-j]$ with $B\in\ob(\A)$ and $j\in\{1,\ldots,n\}$, by induction on $k\in\{1,\ldots,n\}$.

For $k=1$, this is shown in {\rm (1)}.
Assume inductively that this is shown for $k-1$, where $k\ge2$.
By Proposition~\ref{prop:alphabetagamma}, we have a distinguished triangle
$(\sig_{\le k-1}X\up)[-1]\xrightarrow{\al\up}X^k[-k]\xrightarrow{\be\up}\sig_{\le k}X\up\xrightarrow{\gam\up}\sig_{\le k-1}X\up$
in $\Kbf(\A)$. This is sent to distinguished triangle
$(\sig_{\le k-1}\ovl{X}\up)[-1]\xrightarrow{\ovl{\al}\up}X^k[-k]\xrightarrow{\ovl{\be}\up}\sig_{\le k}\ovl{X}\up\xrightarrow{\ovl{\gam}\up}\sig_{\le k-1}\ovl{X}\up$
in $\Kbff(H^0(\A))$ by the functor $\ovl{(\blank)}$.
Thus, we have a commutative diagram
\small
\[
\begin{tikzcd}[column sep = 13pt, row sep = 20pt]
M_{(X^k[-k+1],Y\up)} & M_{(\sig_{\le k-1}X\up,Y\up)}& M_{(\sig_{\le k}X\up,Y\up)}& M_{(X^k[-k],Y\up)}& M_{((\sig_{\le k-1}X\up)[-1],Y\up)} \\
\ovl{M}_{(X^k[-k+1],Y\up)} & \ovl{M}_{(\sig_{\le k-1}X\up,Y\up)}& \ovl{M}_{(\sig_{\le k}X\up,Y\up)}& \ovl{M}_{(X^k[-k],Y\up)}& \ovl{M}_{((\sig_{\le k-1}X\up)[-1],Y\up)} 
\Ar{1-1}{1-2}{}
\Ar{1-2}{1-3}{"\blank\ci\gam\up"}
\Ar{1-3}{1-4}{"\blank\ci\be\up"}
\Ar{1-4}{1-5}{"\blank\ci\al\up"}
\Ar{1-1}{2-1}{"F_{X^k[-k+1],Y\up}"'}
\Ar{1-2}{2-2}{"F_{\sig_{\le k-1}X\up,Y\up}"'}
\Ar{1-3}{2-3}{"F_{\sig_{\le k}X\up,Y\up}"'}
\Ar{1-4}{2-4}{"F_{X^k[-k],Y\up}"'}
\Ar{1-5}{2-5}{"F_{(\sig_{\le k-1}X\up)[-1],Y\up}"'}
\Ar{2-1}{2-2}{}
\Ar{2-2}{2-3}{"\blank\ci\ovl{\gam}\up"'}
\Ar{2-3}{2-4}{"\blank\ci\ovl{\be}\up"'}
\Ar{2-4}{2-5}{"\blank\ci\ovl{\al}\up"'}
\end{tikzcd}
\]
\normalsize
in which rows are exact.

Among the vertical arrows, we see that $F_{X^k[-k+1],Y\up}$ and $F_{X^k[-k],Y\up}$ are isomorphisms by {\rm (1)}, since $2\le k\le n$ and $Y\up=B[-j]$ with $j\in\{1,\ldots,n\}$. Moreover, $F_{\sig_{\le k-1}X\up,Y\up}$ is an isomorphism by the induction hypothesis. The rightmost $F_{(\sig_{\le k-1}X\up)[-1],Y\up}$ is also an isomorphism. Indeed, in the commutative diagram below,
\[
\begin{tikzcd}[column sep = 20pt]
M_{((\sig_{\le k-1}X\up)[-1],Y\up)} & M_{(\sig_{\le k-1}X\up,Y\up[1])} \\
\ovl{M}_{(\sig_{\le k-1}X\up)[-1],Y\up)} & \ovl{M}_{(\sig_{\le k-1}X\up,Y\up[1])}
\Ar{1-1}{1-2}{"{[}1{]}","\cong"'}
\Ar{1-1}{2-1}{"F_{(\sig_{\le k-1}X\up)[-1],Y\up}"'}
\Ar{1-2}{2-2}{"F_{\sig_{\le k-1}X\up,Y\up[1]}","\cong"'}
\Ar{2-1}{2-2}{"{[}1{]}"',"\cong"}
\end{tikzcd}
\]
the vertical arrow on the right is an isomorphism by the induction hypothesis when $j>1$, and by Lemma~\ref{lem:comparison_H} when $j=1$, hence $F_{(\sig_{\le k-1}X\up)[-1],Y\up}$ on the left is also an isomorphism in either case.
By the five lemma, we see that $F_{\sig_{\le k}X\up,Y\up}$ is an isomorphism.

{\rm (3)} Let $X\up\in\Cbf_{[1,n]}(\A)$ be any object. It suffices to show that $F_{X\up,Y\up}$ are isomorphisms for all $Y\up\in\Cbf_{[a,b]}(\A)$ and $1\le a\le b\le n$, by induction on $l=b-a$.

For $l=0$, this is shown in {\rm (2)}.
Assume inductively that this is shown for $l-1$, where $l\ge 1$. 
Let us show that $F_{X\up,Y\up}$ becomes an isomorphism for any $Y\up\in\Cbf_{[a,b]}(\A)$ with $l=b-a$.
By Proposition~\ref{prop:alphabetagamma}, we have a distinguished triangle
\[
(\sig_{\le b-1}Y\up)[-1]\xrightarrow{\al\up}
Y^b[-b]\xrightarrow{\be\up}
Y\up\xrightarrow{\gam\up}
\sig_{\le b-1}Y\up
\]
in $\Kbf(\A)$.
This is sent to the distinguished triangle
$(\sig_{\le b-1}\ovl{Y}\up)[-1]\xrightarrow{\ovl{\al}\up}
Y^b[-b]\xrightarrow{\ovl{\be}\up}
\ovl{Y}\up\xrightarrow{\ovl{\gam}\up}
\sig_{\le b-1}\ovl{Y}\up$
in $\Kbff(H^0(\A))$ by the functor $\ovl{(\blank)}$. Thus, we have a commutative diagram
\small
\[
\begin{tikzcd}[column sep = 13pt, row sep = 20pt]
M_{(X\up,(\sig_{\le b-1}Y\up)[-1])} & M_{(X\up,Y^b[-b])} & M_{(X\up,Y\up)} & M_{(X\up,\sig_{\le b-1}Y\up)} & M_{(X\up,Y^b[-b+1])} \\
\ovl{M}_{(X\up,(\sig_{\le b-1}Y\up)[-1])} & \ovl{M}_{(X\up,Y^b[-b])} & \ovl{M}_{(X\up,Y\up)} & \ovl{M}_{(X\up,\sig_{\le b-1}Y\up)} & \ovl{M}_{(X\up,Y^b[-b+1])}
\Ar{1-1}{1-2}{"\al\up\ci\blank"}
\Ar{1-2}{1-3}{"\be\up\ci\blank"}
\Ar{1-3}{1-4}{"\gam\up\ci\blank"}
\Ar{1-4}{1-5}{}
\Ar{1-1}{2-1}{"F_{X\up,(\sig_{\le b-1}Y\up)[-1]}"}
\Ar{1-2}{2-2}{"F_{X\up,Y^b[-b]}"}
\Ar{1-3}{2-3}{"F_{X\up,Y\up}"}
\Ar{1-4}{2-4}{"F_{X\up,\sig_{\le b-1}Y\up}"}
\Ar{1-5}{2-5}{"F_{X\up,Y^b[-b+1]}"}
\Ar{2-1}{2-2}{"\ovl{\al}\up\ci\blank"'}
\Ar{2-2}{2-3}{"\ovl{\be}\up\ci\blank"'}
\Ar{2-3}{2-4}{"\ovl{\gam}\up\ci\blank"'}
\Ar{2-4}{2-5}{}
\end{tikzcd}
\]
\normalsize
in which rows are exact. Among the vertical arrows, we see that $F_{X\up,(\sig_{\le b-1}Y\up)[-1]}$ and $F_{X\up,\sig_{\le b-1}Y\up}$ are isomorphisms by the induction hypothesis. Moreover, $F_{X\up,Y^b[-b]}$ is an isomorphism by {\rm (2)}. The rightmost $F_{X\up,Y^b[-b+1]}$ is also an isomorphism. Indeed, this follows from {\rm (2)} if $b\ge2$, and from Lemma~\ref{lem:comparison_H} if $b=1$. By the five lemma, we see that $F_{X\up,Y\up}$ in the middle is an isomorphism, as desired.
\end{proof}

\begin{prop}\label{prop:comparison_H}
Assume that $\A$ satisfies $H^i\A(A,B)=0$ for all $A,B\in\ob(\A)$ and $-n<i<0$. Then, the following holds.
\begin{enumerate}
\item The functor $\ovl{(\blank)}\co\Kbf_{[1,n+1]}(\A)\to\Kbff_{[1,n+1]}(H^0(\A))$ is full. In fact, more strongly, for any $X\up,Y\up\in\Cbf_{[1,n+1]}(\A)$ and any $\mathbf{f}\upf\in Z^0(\Cbff_{[1,n+1]}(H^0(\A)))(\ovl{X}\up,\ovl{Y}\up)$, there exists $f\up\in Z^0(\Cbf_{[1,n+1]})(X\up,Y\up)$ such that $\ovl{f}\up=\mathbf{f}\upf$ in $\Kbff_{[1,n+1]}(H^0(\A))$ and $\ovl{f}^{n+1}=\mathbf{f}^{\, n+1}$ in $H^0(\A)$.

Moreover, if $\sig_{\le n}\mathbf{f}\upf$ is an isomorphism in $\Kbff_{[1,n]}(H^0(\A))$ and if $\mathbf{f}^{\, n+1}$ is an isomorphism in $H^0(\A)$, then such $f\up$ can be taken as an isomorphism in $\Kbf_{[1,n+1]}(\A)$.
\item The functor $\ovl{(\blank)}\co\Kbf^{n+2}(\A)\to\Kbff^{n+2}(H^0(\A))$ is surjective on objects. Namely, for any object $\Xbf\up\in \Kbff^{n+2}(H^0(\A))$, there exists an object $X\up\in\Kbf^{n+2}(\A)$ such that $\ovl{X}\up=\Xbf\up$.
\end{enumerate}
\end{prop}
\begin{proof}
{\rm (1)} Let $X\up,Y\up\in\Cbf(\A)$ and $\mathbf{f}\upf\in Z^0(\Cbff_{[1,n+1]}(H^0(\A)))(\ovl{X}\up,\ovl{Y}\up)$ be arbitrary.
As in Proposition~\ref{prop:alphabetagamma}, we have distinguished triangles
\[
(\sig_{\le n}X\up)[-1]\xrightarrow{\al_X\up} X^{n+1}[-(n+1)]\xrightarrow{\be_X\up}X\up\xrightarrow{\gam_X\up}\sig_{\le n}X\up \]
and
\[
(\sig_{\le n}Y\up)[-1]\xrightarrow{\al_Y\up} Y^{n+1}[-(n+1)]\xrightarrow{\be_Y\up}Y\up\xrightarrow{\gam_Y\up}\sig_{\le n}Y\up \]
in $\Kbf(\A)$. Note that we have a commutative diagram
\begin{equation}\label{diagram_to_be_lifted}
\begin{tikzcd}
\sig_{\le n}\ovl{X}\up & X^{n+1}{[}-n{]}  \\
\sig_{\le n}\ovl{Y}\up & Y^{n+1}{[}-n{]}
\Ar{1-1}{1-2}{"\ovl{\al_X}\up{[}1{]}"}
\Ar{1-1}{2-1}{"\sig_{\le n}\mathbf{f}\upf"'}
\Ar{1-2}{2-2}{"(\sig_{\ge n+1}\mathbf{f}\upf) {[}1{]}"}
\Ar{2-1}{2-2}{"\ovl{\al_Y}\up{[}1{]}"'}
\end{tikzcd}
\end{equation}
in $\Kbff(H^0(\A))$.
By Lemma~\ref{lem:comparison_H_plus}, this can be lifted to a commutative diagram
\[
\begin{tikzcd}[column sep = 20]
\sig_{\le n}X\up & X^{n+1}{[}-n{]}  \\
\sig_{\le n}Y\up & Y^{n+1}{[}-n{]}
\Ar{1-1}{1-2}{"(\al_X\up){[}1{]}"}
\Ar{1-1}{2-1}{"g\up"'}
\Ar{1-2}{2-2}{"e\up{[}1{]}"}
\Ar{2-1}{2-2}{"(\al_{Y}\up){[}1{]}"'}
\end{tikzcd}
\]
in $\Kbf(\A)$, where $g\up,e\up$ are morphisms in $Z^0(\Cbf(\A))$ such that $\ovl{g}\up=\sig_{\le n}\mathbf{f}\upf$ and $\ovl{e}\up=\sig_{\ge n+1}\mathbf{f}\upf=\mathbf{f}^{\, n+1}[-(n+1)]$ in $\Kbf(\A)$, respectively. By Corollary~\ref{cor:give_a_morph}, we obtain a morphism $f\up\in Z^0(\Cbf(\A))(X\up,Y\up)$ that satisfies $\sig_{\le n}f\up=g\up$ and $\sig_{\ge n+1}f\up=e\up$ in $\Cbf(\A)$. Then $f\up$ satisfies $\sig_{\le n}\ovl{f}\up=\ovl{g}\up=\sig_{\le n}\mathbf{f}\upf$ and $\sig_{\ge n+1}\ovl{f}\up=\ovl{e}\up=\sig_{\ge n+1}\mathbf{f}\upf$ in $\Kbff(H^0(\A))$. This shows that $f\up$ satisfies $\ovl{f}\up=\mathbf{f}\upf$ in $\Kbff(H^0(\A))$ and $\ovl{f}^{n+1}=\mathbf{f}^{\, n+1}$ in $H^0(\A)$.

Suppose that $\sig_{\le n}\mathbf{f}\upf$ is an isomorphism in $\Kbff_{[1,n]}(H^0(\A))$ and $\mathbf{f}^{\, n+1}$ is an isomorphism in $H^0(\A)$. Then the vertical arrows in $(\ref{diagram_to_be_lifted})$ are isomorphisms in $\Kbff(H^0(\A))$, hence their lifts $g\up$ and $e\up[1]$ are isomorphisms in $\Kbf(\A)$ by Lemma~\ref{lem:comparison_H_plus}. Then, Corollary~\ref{cor:give_a_morph} shows that $f\up$ is an isomorphism in $\Kbf(\A)$.

{\rm (2)}
Let $\Xbf\up\in\Cbff^{n+2}(H^0(\A))$ be any object. For $0\le k\le n+1$, we construct $X_{(k)}\up\in\Cbf_{[k,n+1]}(\A)$ satisfying $\ovl{X_{(k)}}\up=\sig_{\ge k}\Xbf\up$ and $\sig_{\ge k+1}X_{(k)}\up=X_{(k+1)}\up$ by downward induction on $k$.
When $k=n+1$, we simply take $X_{(n+1)}=\Xbf^{n+1}[-(n+1)]$.
Suppose that we have constructed $X_{(m)}\up$ for $k+1\le m\le n+1$ where $k\le n$.
By {\rm (1)}, for a morphism $\mathbf{f}\upf\in Z^0(\Cbff(H^0(\A)))(X^k[-(k+1)],\sig_{\ge k+1}\Xbf\up)$ defined by
\[
\mathbf{f}^{\, i}=\begin{cases}
-d_{\Xbf}^k&\text{if}\ i=k+1\\
0&\text{otherwise}
\end{cases},
\]
there exists a morphism
$f\up\in Z^0(\Cbf(\A))(X^k[-(k+1)],X_{(k+1)}\up)$
such that $\ovl{f}\up=\mathbf{f}\upf$ in $\Kbff(H^0(\A))$. In particular we have $\ovl{f}^{k+1}=\mathbf{f}^{\, k+1}=-d_{\Xbf}^{k}$.
If we put $X\upp=C_f\up$, then it satisfies 
$\sig_{\ge k+1}X\upp=X_{(k+1)}\up$ and $d_{X'}^{k,k+1}=f^{k+1,k+1}$
up to the identification via the canonical isomorphisms $0\oplus X^i\cong X^i\cong X^i\oplus X^i\oplus 0$. Thus we obtain $X_{(k)}\up\in\Cbf_{[k,n+1]}(\A)$ that satisfies $\sig_{\ge k+1}X_{(k)}\up=X_{(k+1)}\up$ and  $\ovl{d}_{X_{(k)}}^k=\ovl{f}^{k+1}=d_{\Xbf}^k$, hence also satisfies $\ovl{X_{(k)}}\up=\sig_{\ge k}\Xbf\up$.
\end{proof}

\begin{rem}
By Proposition~\ref{prop:comparison_H} {\rm (2)}, the functors $\ovl{(\blank)}\co\Kbf_{[1,n+1]}(\A)\to\Kbff_{[1,n+1]}(H^0(\A))$ and $\ovl{(\blank)}\co\Kbf_{[1,n]}(\A)\to\Kbff_{[1,n]}(H^0(\A))$ are also surjective on objects. In particular, the latter
is an equivalence of categories by Lemma~\ref{lem:comparison_H_plus}.
\end{rem}

\normalcolor

\subsection{Realization of $\E$}\label{section:realization}

By restricting the functor $\overline{(\blank)}\colon\Kbf^{n+2}(\A)\to\Kbff^{n+2}(H^0(\A))$ in Definition~\ref{dfn:bar_functor3} to $\Kbf^{n+2}_{A,C}(\A)$, we also obtain a functor
to $\Kbff^{n+2}_{A,C}(H^0(\A))$, which we denote by the same symbol $\overline{(\blank)}\colon\Kbf^{n+2}_{A,C}(\A)\to\Kbff^{n+2}_{A,C}(H^0(\A))$.
Henceforth, for any $\Xbf\up\in\Kbff^{n+2}_{A,C}(H^0(\A))$, we write $[\Xbf\up]$ for its isomorphism class in $\Kbff^{n+2}_{A,C}(H^0(\A))$.

\begin{lem}\label{lem:EquivtoEquiv}
Let $A,C\in\ob(\A)$ be any pair of objects. Assume that ${}_AX\up_C,{}_AY\up_C\in {}_A\calS_C$ satisfy $\dbr{X\up}=\dbr{Y\up}$ in $\E(C,A)$.
Then $\overline{X}\up\cong\overline{Y}\up$ holds in $\Kbff^{n+2}_{A,C}(H^0(\A))$. 
Namely, $[\ovl{X}\up]=[\ovl{Y}\up]$ holds in $\ob(\Kbff^{n+2}_{A,C}(H^0(\A)))/_{\cong}$.
\end{lem}
\begin{proof}
By $\dbr{X\up}=\dbr{Y\up}$, there exists $f\up\in Z^0(\Cbf^{n+2}(\A))(X\up,Y\up)$ such that $\ovl{f}^0=\id$ and $\ovl{f}^{n+1}=\id$, which is an isomorphism in ${}_A\calS_C\subset\Kbf^{1,n,1}_{A,C}(\A)$, hence also an isomorphism in $\Kbf_{A,C}^{n+2}(\A)$.
By the functoriality of $\overline{(\blank)}\colon\Kbf^{n+2}_{A,C}(\A)\to\Kbff^{n+2}_{A,C}(H^0(\A))$, the morphism $\ovl{f}\up\colon\ovl{X}\up\to\ovl{Y}\up$ becomes an isomorphism in $\Kbff^{n+2}_{A,C}(H^0(\A))$.
\end{proof}

\begin{dfn}\label{dfn:realization}
Let $A,C\in\ob(\A)$ be any pair of objects.
For each $\del\in\E(C,A)$, take any ${}_AX\up_C\in{}_A\calS_C$ satisfying $\del=\dbr{X\up}$, and put $\sfr(\del)=[X\up]\in\ob(\Kbff^{n+2}_{A,C}(H^0(\A)))/_{\cong}$.
This is well-defined by Lemma~\ref{lem:EquivtoEquiv}.
\end{dfn}

\begin{prop}\label{prp:realization}
Assume that Condition~\ref{condition:H^i=0} is satisfied.
Then $\sfr$ is an exact realization of $\E$, in the sense of \cite[Definition~2.22]{HLN1}.
\end{prop}
\begin{proof}
Let us confirm that $\sfr$ satisfies conditions {\rm (R0),(R1),(R2)} in \cite[Definition~2.22]{HLN1}.

{\rm (R0)} Let $\del=\dbr{X\up}\in\E(C,A)$ and $\rho=\dbr{Y\up}\in\E(B,D)$ be any pair of extensions.
Let $(\abf,\cbf)\colon\del\to\rho$ be any morphism of extensions.
By definition, $\abf\sas\del=\dbr{X\upp}$ and $\cbf\uas\rho=\dbr{Y\upp}$ are given by
an $n$-pushout morphism $f\up\colon X\up\to X\upp$ in $\Cbf^{n+2}(\A)$ such that $\ovl{f}^0=\abf$, and
an $n$-pullback morphism $g\up\colon Y\upp\to Y\up$ in $\Cbf^{n+2}(\A)$ such that $\ovl{g}^{n+1}=\cbf$.
By Lemma~\ref{lem:EquivtoEquiv}, the equality $\dbr{X\upp}=\dbr{Y\upp}$ implies $[X\upp]=[Y\upp]$ in $\ob(\Kbf^{n+2}_{B,C}(\A))/_{\cong}$, hence there exists a morphism $\al\up\colon \ovl{X'}\up\to\ovl{Y'}\up$ in $\Cbff^{n+2}_{B,C}(H^0(\A))$. Thus we have a morphism $\ovl{g}\up\ci\al\up\ci\ovl{f}\up$ in $\Cbff^{n+2}(H^0(\A))$, which gives a lift of $(\abf,\cbf)$.

{\rm (R1)} Let $X\up\in{}_A\calS_C$ be any object. By Proposition~\ref{prp:long_ex_seq},
\[ H^0\A(\blank,X^0)\xrightarrow{\ovl{d}_X^0\ci\blank}H^0\A(\blank,X^1)\xrightarrow{\ovl{d}_X^1\ci\blank}\cdots\to H^0\A(\blank,X^n)\xrightarrow{\ovl{d}_X^{n}\ci\blank} H^0\A(\blank,X^{n+1})
\]
is exact in $\Mod H^0\A$. 

Let us show that
\[ H^0\A(\blank,X^n)\xrightarrow{\ovl{d}_X^{n}\ci\blank} H^0\A(\blank,X^{n+1})\xrightarrow{\del\ssh}\E(\blank,X^0) \]
is exact for $\del=\dbr{X\up}$. Here $\del\ssh$ denotes the natural transformation corresponding to $\del$ by the Yoneda lemma. For any $C'\in\ob(\A)$, the map $(\del\ssh)_{C'}\co H^0\A(C',X^{n+1})\to \E(C',X^0)$ sends each $\cbf\in H^0\A(C',X^{n+1})$ to $\cbf\uas\dbr{X\up}\in\E(C',X^0)$.
By its definition, there exist $Y\up\in {}_A\calS_{C'}$ and an $n$-pullback morphism $f\up\co Y\up\to X\up$ such that $\ovl{f}^{n+1}=\cbf$, which gives $\cbf\uas\dbr{X\up}=\dbr{Y\up}$. If $\cbf\in\Ker((\del\ssh)_{C'})$, then $Y\up$ is a split sequence by Lemma~\ref{lem:Equivtozero}, hence $\ovl{d}_Y^n$ has a section $\sbf\in H^0\A(C',Y^n)$. Then $\ovl{f}^n\ci \sbf\in H^0\A(C',X^n)$ satisfies $\ovl{d}_X^n\ci(\ovl{f}^n\ci \sbf)=\cbf\ci \ovl{d}_Y^n\ci \sbf=\cbf$, which shows $\cbf\in\Ima((\ovl{d}_X^n\ci-))$.

To show $\Ima(\ovl{d}_X^n\ci-)\subset\Ker(\del\ssh)$, by the Yoneda lemma, it is enough to show $\dbr{(\ovl{d}_X^n)\uas X\up}=0$. Take ${}_AZ\up_{X^n}\in {}_A\calS_{X^n}$ and an $n$-pullback morphism $g\co Z\up\to X\up$ such that
$\ovl{g}^{n+1}=\ovl{d}_X^n$, so that $\dbr{Z\up}=\dbr{(\ovl{d}_X^n)\uas X\up}$ holds.
By Proposition~\ref{prop:pushout_inflation}, we see that $C\up=\CoCone(\sig_{\ge1}g\up[1])$ belongs to $\calS$, hence in particular it is left $n$-exact. Thus by Proposition~\ref{prp:long_ex_seq},
\[ H^0\A(-,C^{n-1})\xrightarrow{d_{\ovl{C}}^{n-1}\ci-}H^0\A(-,C^n)\xrightarrow{d_{\ovl{C}}^n\ci-}H^0\A(-,C^{n+1}) \]
becomes exact. This means that, especially,
\[
H^0\A(X^n,Z^n\oplus X^{n-1})\xrightarrow{d_{\ovl{C}}^{n-1}\ci-}H^0\A(X^n,X^n\oplus X^n)\xrightarrow{d_{\ovl{C}}^n\ci-}H^0\A(X^n,X^{n+1})
\]
is exact, where
\[ d_{\ovl{C}}^{n-1}=\left[\begin{array}{cc}
d_{\ovl{Z}}^n&0\\
\ovl{f}^n&-d_{\ovl{X}^{n-1}}
\end{array}\right]
\ ,\quad
d_{\ovl{C}}^n=\left[\begin{array}{cc}
\ovl{d}_X^n&-\ovl{d}_X^n
\end{array}\right].
\]
Since $\Delta=\begin{bmatrix}\id_X\\\id_X\end{bmatrix}\in H^0\A(X^n,X^n\oplus X^n)$ satisfies $d_{\ovl{C}}^n\ci\Delta=0$, there exists $\begin{bmatrix}\abf\\ \mathbf{b}\end{bmatrix}\in H^0\A(X^n,Z^n\oplus X^{n-1})$ such that $d_{\ovl{C}}^{n-1}\ci\begin{bmatrix}\abf\\ \mathbf{b}\end{bmatrix}=\Delta$. Thus $\abf\in H^0\A(X^n,Z^n)$ gives a section of $d_{\ovl{Z}}^n$, hence $\dbr{Z\up}=0$ holds by Lemma~\ref{lem:Equivtozero}.
Exactness of
\[ H^0\A(X^{n+1},\blank)\xrightarrow{\blank\ci \ovl{d}_X^n}H^0\A(X^n,\blank)\xrightarrow{\blank\ci \ovl{d}_X^{n-1}}\cdots
\xrightarrow{\blank\ci \ovl{d}_X^0} H^0\A(X^0,\blank)
\xrightarrow{\del\ush} \E(X^{n+1},\blank)
\]
can be shown in a dual manner, in which $\del\ush$ denotes the natural transformation corresponding to $\del\in \E(X^{n+1},X^0)$ by the Yoneda lemma.

{\rm (R2)} By Definition~\ref{dfn:splitN}, we have $\sfr({}_A0_0)=[A\xrightarrow{\id_A}A\to0\to\cdots\to0]$ and $\sfr({}_00_C)=[0\to\cdots\to0\to C\xrightarrow{\id_C}C]$ for all $A,C\in\ob(H^0(\A))$. 
\end{proof}

\begin{lem}\label{lem:nea-1}
Assume that Condition~\ref{condition:H^i=0} is satisfied.
Let $A,C\in\ob(\A)$, and let $\Ybf\up\in\Cbff^{n+2}_{A,C}(H^0(\A))$ be arbitrary. Suppose that $\ovl{X}\up\cong\Ybf\up$ holds in $\Kbff^{n+2}_{A,C}(H^0(\A))$ for some $X\up\in {}_A\calS_C$. Then, there exists $Y\upp\in\calS$ that satisfies $Y^{\prime0}=X^0$ in $\A$ and $\sig_{\ge1}\ovl{Y'}\up=\sig_{\ge1}\Ybf\up$ in $\Kbff(H^0(\A))$. In particular, it satisfies $\ovl{d}_{Y'}^n=d_{\Ybf}^n$ in $H^0(\A)$.
\end{lem}
\begin{proof}
Since $\ovl{X}\up\cong\Ybf\up$ holds in $\Kbff^{n+2}_{A,C}(H^0(\A))$, there exist
$\mathbf{f}\upf\in Z^0(\Cbff^{n+2}(H^0(\A)))(\ovl{X}\up,\Ybf\up)$ and $\mathbf{f}\uppf\in Z^0(\Cbff^{n+2}(H^0(\A)))(\Ybf\up,\ovl{X}\up)$ such that $\mathbf{f}\uppf\ci\mathbf{f}\upf=\id_{\Ybf\up}$, $\mathbf{f}\upf\ci\mathbf{f}\uppf=\id_{\ovl{X}\up}$in $\Kbff^{n+2}(H^0(\A))$ and  $\mathbf{f}^{\, 0}=\mathbf{f}^{\, \prime 0}=\id_A$, $\mathbf{f}^{\, n+1}=\mathbf{f}^{\, \prime n+1}=\id_C$.
First we show that $\sig_{[1,n]}\mathbf{f}\upf$ becomes an isomorphism in $\Kbff_{[1,n]}(H^0(\A))$. 

By Proposition~\ref{prp:long_ex_seq} and its dual,
\begin{equation}\label{ex_1}
H^0\A(\blank,X^0)
\xrightarrow{\ovl{d}_X^0\ci\blank}H^0\A(\blank,X^1)
\xrightarrow{\ovl{d}_X^1\ci\blank}H^0\A(\blank,X^2)
\end{equation}
and
\begin{equation}\label{ex_2}
H^0\A(X^{n+1},\blank)
\xrightarrow{\blank\ci\ovl{d}_X^n}H^0\A(X^n,\blank)
\xrightarrow{\blank\ci\ovl{d}_X^{n-1}}H^0\A(X^{n-1},\blank)
\end{equation}
are exact. Since $\Ybf\up$ is isomorphic to $\ovl{X}\up$ in $\Kbff(H^0(\A))$, it follows that the complex
\begin{equation}\label{ex_3}
H^0\A(\blank,\Ybf^0)
\xrightarrow{d_{\Ybf}^0\ci\blank}H^0\A(\blank,\Ybf^1)
\xrightarrow{d_{\Ybf}^1\ci\blank}H^0\A(\blank,\Ybf^2)
\end{equation}
is homotopy equivalent to $(\ref{ex_1})$, hence $(\ref{ex_3})$ becomes also exact. Similarly,
\begin{equation}\label{ex_4}
H^0\A(\Ybf^{n+1},\blank)
\xrightarrow{\blank\ci d_{\Ybf}^n}H^0\A(\Ybf^n,\blank)
\xrightarrow{\blank\ci d_{\Ybf}^{n-1}}H^0\A(\Ybf^{n-1},\blank)
\end{equation}
is exact. 
As before, we treat the case $n\ge 2$.
(When $n=1$, then $\mathbf{f}^{\, 1}$ becomes an isomorphism in $H^0(\A)$ by \cite[Lemma~4.1]{HLN1}.) 
Since $\mathbf{f}\uppf\ci\mathbf{f}\upf=\id_{\ovl{X}\up}$ in $\Kbf(\A)$, there exists a homotopy from $\mathbf{f}\uppf\ci\mathbf{f}\upf$ to  $\id_{\ovl{X}\up}$, that is, a morphism $\varphi\up\in\Cbff^{n+2}(H^0(\A))(\ovl{X}\up,\ovl{X}\up)^{-1}$ such that 
\begin{equation}\label{htpy_phi}
\id_{\ovl{X}\up}-\mathbf{f}\uppf\ci\mathbf{f}\upf=d\varphi\up.
\end{equation}
Let us show that $\varphi\up$ can be modified to satisfy $\varphi^1=0$ and $\varphi^{n+1}=0$, in a similar argument to that of Proposition~\ref{prp:htpy_modif}. By $(\ref{htpy_phi})$, especially we have $\varphi^1\ci\ovl{d}_X^0=0$ and $\ovl{d}_X^n\ci\varphi^{n+1}=0$. By the exactness of $(\ref{ex_1})$ and $(\ref{ex_2})$, there are
$\xi^1\in H^0(\A)(X^2,X^0)$ and $\xi^{n+1}\in H^0(\A)(X^{n+1},X^{n-1})$ such that $\varphi^1=\xi^1\ci\ovl{d}_X^1$, $\varphi^{n+1}=\ovl{d}_X^{n-1}\ci\xi^{n+1}$ in $H^0(\A)$.
Define $\varphi\upp\in\Cbff^{n+2}(H^0(\A))(\ovl{X}\up,\ovl{X}\up)^{-1}$ by
\[
\varphi^{\prime i}=\begin{cases}
\varphi^2+\ovl{d}_X^0\ci\xi^1&\text{if}\ i=2\\
\varphi^i&\text{if}\ 2<i<n\\
\varphi^n+\xi^{n+1}\ci\ovl{d}_X^n&\text{if}\ i=n\\
0&\text{otherwise}
\end{cases}
\]
if $n>2$, and 
\[
\varphi^{\prime i}=\begin{cases}
\varphi^2+\ovl{d}_X^0\ci\xi^1+\xi^3\ci\ovl{d}_X^2&\text{if}\ i=2\\
0&\text{otherwise}
\end{cases}
\]
if $n=2$.
Then, $\varphi\upp$ is a homotopy from $\mathbf{f}\uppf\ci\mathbf{f}\upf$ to $\id_{\ovl{X}\up}$ satisfying $\varphi^{\prime 1}=0$ and $\varphi^{\prime n+1}=0$. 
We see that $\sig_{[1,n]}\varphi\upp$ gives a homotopy from $(\sig_{[1,n]}\mathbf{f}\uppf)\ci(\sig_{[1,n]}\mathbf{f}\upf)$ to $\id_{\sig_{[1,n]}\ovl{X}\up}$.
Similarly, by using the exactness of $(\ref{ex_3})$ and $(\ref{ex_4})$, we can obtain a homotopy 
from $(\sig_{[1,n]}\mathbf{f}\upf)\ci(\sig_{[1,n]}\mathbf{f}\upp)$ to $\id_{\sig_{[1,n]}\Ybf\up}$. This means that $\sig_{[1,n]}\mathbf{f}\upf$ is an isomorphism in $\Kbff_{[1,n]}(H^0(\A))$. 

By Proposition~\ref{prop:comparison_H} {\rm (2)}, there exists $Y\up\in\Cbf^{n+2}(\A)$ such that $\ovl{Y}\up=\Ybf\up$. Then, by Proposition~\ref{prop:comparison_H} {\rm (1)}, there exists $g\up\in Z^0(\Cbf_{[1,n+1]}(\A))(\sig_{\ge 1}X\up,\sig_{\ge 1}Y\up)$ such that $\ovl{g}\up=\sig_{\ge 1}\mathbf{f}\upf$ in $\Kbff^{n,1}(H^0(\A))$ and $\ovl{g}^{n+1}=\mathbf{f}^{\, n+1}=\id_C$ in $H^0(\A)$. Moreover, since  
$\sig_{[1,n]}\mathbf{f}\upf=\sig_{\le n}(\sig_{\ge 1}\mathbf{f}\upf)$ is an isomorphism in $\Kbff_{[1,n]}(H^0(\A))$, we may assume that $g\up$ is an isomorphism in $\Kbf_{[1,n+1]}(\A)$ by the same proposition.
Let $X^0[-1]\xrightarrow{\al_{X,1}\up}\sig_{\ge 1}X\up\xrightarrow{\be_{X,1}\up}X\up$
be the sequence of morphisms in $\Cbf(\A)$ as in Proposition~\ref{prop:alphabetagamma}.
If we define $Y\upp\in\Cbf^{n+2}(\A)$ to be the cone of $h\up=g\up\ci\al_{X,1}\up\co X^0[-1]\to\sig_{\ge 1}Y\up$ in $\Cbf(\A)$, then it 
satisfies $\sig_{\ge 1}Y\upp=\sig_{\ge 1}Y\up$, $\sig_{\le 0}Y\upp=X^0[-1]$ and $h\up=\al_{Y',1}$ in $\Cbf(\A)$, up to the identification via the canonical isomorphisms $Y^{\prime i}\oplus 0\cong Y^{\prime i}\cong 0\oplus Y^{\prime i}$ in $Z^0(\A)$ for all $i\in\bbZ$. 

It remains to show that $Y\upp$ belongs to $\calS$. 
Since
\[
\begin{tikzcd}
X^0[-1] & \sig_{\ge 1}X\up\\
Y^{\prime 0}[-1] & \sig_{\ge 1}Y\upp
\Ar{1-1}{1-2}{"\al_{X,1}\up"}
\Ar{1-1}{2-1}{equal}
\Ar{1-2}{2-2}{"g\up"}
\Ar{2-1}{2-2}{"\al_{Y',1}\up"'}
\end{tikzcd}
\]
is commutative,
we obtain $f\upp\in Z^0(\Cbf(\A))(X\up,Y\upp)$ such that $\sig_{\ge 1}f\upp=g\up$ in $\Cbf(\A)$ and $f^{\prime0,0}=\id_A$ by Corollary~\ref{cor:give_a_morph}.
Since $g\up$ is an isomorphism in $\Kbf(\A)$, so is $f\upp$ by the same corollary. Then, since $X\up$ is $n$-exact and since $\ovl{f'}^0=\id_A$, $\ovl{f'}^{n+1}=\ovl{g}^{n+1}=\id_C$ are satisfied in $H^0(\A)$, it follows that $f\upp$ is an isomorphism in $\Kbf^{1,n,1}(\A)$ by Corollary~\ref{cor:htpy_modif2}. Thus $Y\upp$ belongs to $\calS$ by {\rm ($n$-Ex0)}, since $X\up\in\calS$.
\end{proof}

\normalcolor

\begin{thm}\label{thm:H0_of_n-ex-dg}
Let $(\A,\calS)$ be an $n$-exact dg-category with skeletally small $H^0(\A)$.
Under Condition~\ref{condition:H^i=0}, the triplet $(H^0(\A),\E,\sfr)$ is an $n$-exangulated category defined in \cite[Definition~2.32]{HLN1}.
Here, $\E$ is the biadditive functor obtained in Proposition~\ref{prp:biadditive_functor_E}, and $\sfr$ is its exact realization obtained in Proposition~\ref{prp:realization}.
\end{thm}
\begin{proof}
It remains to check the conditions {\rm (EA1),(EA2),(EA2$\op$)} listed in \cite[Definition~2.32]{HLN1}.

By Lemma~\ref{lem:nea-1}, a morphism $\mathbf{d}\in H^0(\A)(A,B)$ is a deflation in $(H^0(\A),\E,\sfr)$ if and only if there exists a deflation $d\in Z^0(\A)(A,B)$ in the $n$-exact dg category $(\A,\calS)$ such that $\mathbf{d}=\ovl{d}$. Similarly for inflations. Thus, {\rm (EA1)} follows immediately from {\rm ($n$-Ex1)} in Definition~\ref{dfn:n-exact-dg}.

As {\rm (EA2)} can be shown dually, let us only show {\rm (EA2$\op$)}. Take arbitrary $\del\in\E(C,A), a\in Z^0(\A)(A,B)$, and let ${}_A\dtri{\Xbf\up}{\del}_C$, ${}_B\dtri{\Ybf\up}{a\sas\del}_C$ be any pair of distinguished $n$-exangles.
By the dual of \cite[Remark~2.33 {\rm (2)}]{HLN1}, we may assume that
there exist $X\up,Y\up\in\calS$ which give $\del=\dbr{X\up}$, $a\sas\del=\dbr{Y\up}$ and $\Xbf\up=\ovl{X}\up$, $\Ybf\up=\ovl{Y}\up$  from the beginning.

In particular, by $a\sas\del=\dbr{Y\up}$ and by Remark~\ref{rem:strict_npo} {\rm (1)}, there exists an $n$-pushout morphism $f\up\colon X\up\to Y\up$ in $\Cbf(\A)$ such that $f^{0,0}=a$, $f^{n+1,n+1}=\id_C$ in $H^0(\A)$.
By the dual of Proposition~\ref{prop:pushout_inflation}, $C\up=\Cone ((\sigma_{\leq n}f\up)[-1])$ belongs to $\calS$.
If we define a morphism $g\up\co C\up\to X\up$ by
\[
g^{i,j}=\begin{cases}
\id & \text{if}\ (i,j)=(0,0)\\
[\id\ 0] & \text{if}\ 1\le i=j\le n\\
d_Y^{n,n+1} & \text{if}\ (i,j)=(n+1,n+1)\\
[-f^{i,n+1}\ d_Y^{i-1,n+1}] & \text{if}\ j=n+1\text{and}\ 1\le i\le n\\
-f^{0,n+1} & \text{if}\ (i,j)=(0,n+1)\\
0 & \text{otherwise}
\end{cases}
\]
then we can check directly that $g\up\in Z^0(\Cbf(\A))(C\up,X\up)$ holds. 
By Corollary~\ref{cor:morph_in_S_to_morph_of_extension}, this implies $(\ovl{d}_Y^n)\uas\del=\dbr{C\up}$ in $\bbE(Y^n,A)$.
Thus $\dtri{\ovl{C}\up}{(\ovl{d}_Y^n)\uas\del}$ becomes a distinguished $n$-exangle, which shows that $(H^0(\A),\E,\sfr)$ satisfies {\rm (EA2$\op$)}.
\end{proof}

\begin{dfn}\label{dfn:ext_closed_sub}
Let $(\A,\calS)$ be as before, and let $\B\subset\A$ be a strictly additive full dg-subcategory to which $0\in\ob(\A)$ belongs. We say that $\B$ is an \emph{$n$-extension-closed dg-subcategory} of $\A$ if for any $X\up\in\calS$ with $X^0,X^{n+1}\in\ob(\B)$, there exist $Y\up\in\Cbf^{n+2}(\B)$ and a morphism $f\up\in Z^0(\Cbf^{n+2}(\A))(X\up,Y\up)$ such that $f\up$ is an isomorphism in $\Kbf^{n+2}(\A)$ and $\ovl{f}^0=\id$, $\ovl{f}^{n+1}=\id$ in $H^0(\A)$.

We remark that such $f\up$ is also an isomorphism in $\Kbf^{1,n,1}(\A)$ by Corollary~\ref{cor:htpy_modif2}, and hence $Y\up\in\calS$ automatically.
\end{dfn}

\begin{lem}\label{lem:ext_closed_sub}
Assume that Condition~\ref{condition:H^i=0} is satisfied. Let $(\A,\calS)$ be as before, and assume that $\B\subset\A$ is an $n$-extension-closed dg-subcategory. Then, in the $n$-exangulated category $(H^0(\A),\E,\sfr)$ obtained in Theorem~\ref{thm:H0_of_n-ex-dg}, the additive full subcategory $H^0(\B)\subset H^0(\A)$ is \emph{extension-closed} in the sense of \cite[Definition~4.1]{HLN2} (also called \emph{$n$-extension closed} in \cite[Definition~1.18]{Kla}).
\end{lem}
\begin{proof}
Since $\B\subset\A$ is $n$-extension-closed, for any $A,C\in\ob(\B)$ and any $\del\in\E(C,A)$ there exists $X\up\in\calS_{\B}$ such that $\del=\dbr{X\up}$ in $\E(C,A)$. We have $\sfr(\del)=[X\up]$ by the definition of $\sfr$ in Definition~\ref{dfn:realization}, which means that $H^0(\B)\subset H^0(\A)$ is extension-closed.
\end{proof}

\begin{prop}\label{prop:ext_closed_sub}
Assume that Condition~\ref{condition:H^i=0} is satisfied. Let $(\A,\calS)$ be as before, and assume that $\B\subset\A$ is an $n$-extension-closed dg-subcategory. If we define $\calS_{\B}\subset\Kbf^{1,n,1}(\B)$ to be the full subcategory consisting of all the objects $X\up$ such that $X\up\in\calS$, then $(\B,\calS_{\B})$ is also an $n$-exact dg-category.
\end{prop}
\begin{proof}
We note that $\Kbf(\B)\subset\Kbf(\A)$ is a triangulated full subcategory, and that $\Kbf^{n+2}(\B)\subset\Kbf^{n+2}(\A)$, $\Kbf^{1,n,1}(\B)\subset\Kbf^{1,n,1}(\A)$, and $H^0(\B)\subset H^0(\A)$ are additive full subcategories.
Let us show that $(\B,\calS_{\B})$ satisfies conditions in Definition~\ref{dfn:n-exact-dg}.

{\rm ($n$-Ex0)} By definition we have $\calS_{\B}=\calS\cap(\Kbf^{1,n,1}(\B))$ as full subcategories of $\Kbf^{1,n,1}(\A)$. In particular, since $\calS\subset\Kbf^{1,n,1}(\A)$ is closed by isomorphisms, so is $\calS_{\B}\subset\Kbf^{1,n,1}(\B)$.
By Proposition~\ref{prop:any_split_belongs_to_S}, the trivial split sequence $N\up={}_0N\up_0\in\Cbf^{n+2}(\A)$, that is, the object $N\up$ satisfying $N^i=0$ for all $i\in\bbZ$, should belong to $\calS$. Since $0\in\ob(\B)$, it follows $N\up\in\calS_{\calB}$.

{\rm ($n$-Ex1)} The statement for $\calS_{\B}$-inflations can be treated dually, so it suffices to show that
$\calS_{\B}$-deflations are closed by composition.
Let $B\xrightarrow{b}C\xrightarrow{c}D$ be any sequence of $\calS_{\B}$-deflations. 

As in Theorem~\ref{thm:H0_of_n-ex-dg}, $(H^0(\A),\E,\sfr)$ is an $n$-exangulated category. Also, by Lemma~\ref{lem:ext_closed_sub}, $H^0(\B)\subset H^0(\A)$ is extension-closed.
Then, by \cite[Theorem 3.3]{Kla} (or \cite[Theorem~3.2]{HZ} for the Krull-Schmidt case), the restriction of $\E$ and $\sfr$ to $H^0(\B)$ induces an $n$-exangulated structure on $H^0(\B)$. In particular, as shown in \cite[Lemma~3.2]{Kla}, the morphism $\ovl{c}\ci\ovl{b}$ is a deflation with respect to this $n$-exangulated structure. This means that there exist $A\in\ob(\B)$, $\del\in\E(D,A)$ and $\Ybf\up\in\Cbff^{n+2}(H^0(\B))$ such that $\langle\Ybf\up,\del\rangle$ is a distinguished $n$-exangle which satisfies $d_{\Ybf}^n=\ovl{c}\ci\ovl{b}$.

By the definition of $\E$, there exists $X\up\in {}_A\calS_D$ such that $\del=\dbr{X\up}$. Since $\B\subset\A$ is $n$-extension-closed, we may assume that $X\up$ belongs to $\calS_{\B}$. Then we have $[\Ybf\up]=\sfr(\del)=[\ovl{X}\up]$, namely, $\ovl{X}\up\cong\Ybf\up$ holds in $\Kbff^{n+2}_{A,D}(H^0(\A))$.
By Lemma~\ref{lem:nea-1}, there exist $Y\upp\in\calS$ such that $Y^{\prime 0}=X^0$ and $\sig_{\ge1}\ovl{Y'}\up=\sig_{\ge1}\Ybf\up$.
Since $Y^{\prime i}\in\ob(\B)$ holds for all $i\in\bbZ$, we have $Y\upp\in\calS_{\B}$. In particular, $d_{Y'}^{n,n+1}$ is an $\calS_{\B}$-deflation. By Remark~\ref{rem:strict_npo} {\rm (2)}, it follows that $c\ci b$ is also an $\calS_{\B}$-deflation. 
Thus $\calS_{\B}$-deflations are closed by composition.

{\rm ($n$-Ex2)} Let $X\up\in\calS_{\B}$ and $a\in Z^0(\B)(X^0,A)$ be arbitrary. As $(\A,\calS)$ satisfies {\rm ($n$-Ex2)}, there exist $Y\up\in\calS$ and an $n$-pushout morphism $g\up\colon X\up\to Y\up$ in $\Cbf^{n+2}(\A)$ such that $Y^0=A$ and $\ovl{g}^0=\ovl{a}$ in $H^0(\A)$. Since $\B\subset\A$ is $n$-extension-closed, there exist $Y\upp\in\calS_{\B}$ and a morphism $y\up\in Z^0(\Cbf^{n+2}(\A))(Y\up,Y\upp)$ such that $y\up$ is an isomorphism in $\Kbf^{n+2}(\A)$ and $\ovl{y}^0=\id$, $\ovl{y}^{n+1}=\id$ in $H^0(\A)$. Then, since $\Kbf(\B)\subset\Kbf(\A)$ is a triangulated full subcategory, it follows that $g\upp=y\up\ci g\up$ is an $n$-pushout morphism in $\Cbf^{n+2}(\B)$. Moreover it satisfies $\ovl{g'}^0=\ovl{a}$ in $H^0(\B)$. {\rm ($n$-Ex2$\op$)} can be checked in a dual manner. 
\normalcolor
\end{proof}

\subsection{Relation to ordinary $n$-exact category}
When $\calA$ is an ordinary category regarded as a dg-category, we show that the notion of an $n$-exact dg-category in Definition~\ref{dfn:n-exact-dg} is equivalent to that of an $n$-exact category in the sense of Jasso. Note that Condition~\ref{condition:H^i=0} is trivially satisfied.

We remark that, for $\calS\subset\ob(\Cbf^{n+2}(\calA))=\ob(\Cbff^{n+2}(\calA))$, it follows from
Proposition~\ref{prop:n-exact_ordinary} that $\calS$ consists of $n$-exact sequences in our sense if and only if it consists of $n$-exact sequences in the sense of \cite[Definition~2.2]{J}.
We begin by recalling the following result from \cite{HLN1}, adapting the terminology to our setting.
\begin{fact}\label{fact:Prop4.23_HLN1}(\cite[Definitions~4.21, 4.22, and Proposition~4.23]{HLN1})
Let $\calA$ be an ordinary additive category, and let $\calS\subset\ob(\Cbff^{n+2}(\calA))$ be a class consisting of some $n$-exact sequences. The pair $(\calA,\calS)$ is an $n$-exact category in the sense of \cite[Definition~4.2]{J} if and only if it satisfies the following conditions.
\begin{itemize}
\item[{\rm (E0)}] The split sequence ${}_0N_0\in\Cbff^{n+2}(\calA)$ in Definition~\ref{dfn:splitN}, that is, the sequence
$0\to0\to\cdots\to0$
belongs to $\calS$.
\item[{\rm (E1)}] $\calS$-inflations are closed by composition in $\calA$. 
\item[{\rm (E1$\op$)}] $\calS$-deflations are closed by composition in $\calA$. 
\item[{\rm (EC$'$)}] Let $X\up,Y\up\in\Cbff^{n+2}(\calA)$ be any pair of $n$-exact sequences with $X^0=Y^0$ and $X^{n+1}=Y^{n+1}$. Suppose that there are morphisms $f\up\in Z^0(\Cbff^{n+2}(\calA))(X\up,Y\up)$ and $g\up\in Z^0(\Cbff^{n+2}(\calA))(Y\up,X\up)$ satisfying $f^0=g^0=\id$ and $f^{n+1}=g^{n+1}=\id$, which are inverse to each other in $\Kbff^{n+2}(\calA)$. Then, $X\up\in\calS$ if and only if $Y\up\in\calS$.
\item[{\rm (E2$'$)}] The dual of the following {\rm (E2${}^{\prime\mathrm{op}}$)}.
\item[{\rm (E2${}^{\prime\mathrm{op}}$)}] \begin{itemize}
\item[{\rm (i)}] For any $c\in\calA(C',C)$ and any $Y\up\in\calS$ with $Y^{n+1}=C$, there exist $X\up\in\calS$ and $f\up\in Z^0(\Cbff^{n+2}(\calA))(X\up,Y\up)$ such that $f^0=\id$ and $f^{n+1}=c$.
\item[{\rm (ii)}] For any $X\up,Y\up\in\calS$ and any $f\up\in Z^0(\Cbff^{n+2}(\calA))(X\up,Y\up)$ satisfying $f^0=\id$, we have $\CoCone((\sig_{\ge 1}f\up)[1])\in\calS$. 
\end{itemize}
\item[{\rm (EI)}] (cf.\ \cite[Corollary~4.13]{HLN1})
For any $X\up\in\calS$ and any $n$-exact sequence $Y\up\in\Cbff^{n+2}(\calA)$, the following holds.
\begin{itemize}
\item[{\rm (I1)}] If there is $f\up\in Z^0(\Cbff^{n+2}(\calA))(X\up,Y\up)$ such that $f^0$ and $f^{n+1}$ are isomorphisms in $\calA$, then $f\up$ is an isomorphism in $\Kbff^{n+2}(\calA)$.
\item[{\rm (I2)}] If there is $g\up\in Z^0(\Cbff^{n+2}(\calA))(Y\up,X\up)$ such that $g^0$ and $g^{n+1}$ are isomorphisms in $\calA$, then $g\up$ is an isomorphism in $\Kbff^{n+2}(\calA)$.
\end{itemize}
\end{itemize}
\end{fact}

\begin{prop}\label{prop:compare_ordinary_n-exact}
Let $\calA$ be an ordinary additive category, which we may regard $\calA$ as a dg-category satisfying $\calA^i=0$ for $i\ne0$. Let $\calS\subset\ob(\Cbff^{n+2}(\calA))$ consists of $n$-exact sequences. 
Then, the following are equivalent. 
\begin{enumerate}
\item $(\calA,\calS)$ is an $n$-exact category in the sense of \cite[Definition~4.2]{J}.
\item $(\calA,\calS)$ is an $n$-exact dg-category in the sense of Definition~\ref{dfn:n-exact-dg}. 
\end{enumerate}
\end{prop}
\begin{proof}
First we note that, by Corollary~\ref{cor:htpy_modif2}, condition {\rm (EC$'$)} is equivalent to the condition {\rm (2)} in Proposition~\ref{prop:closed_equiv}.

$(1)\Rightarrow(2)$  Suppose that $(\calA,\calS)$ is an $n$-exact category in the sense of \cite[Definition~4.2]{J}. 
By Proposition~\ref{prop:closed_equiv} and {\rm (EC$'$)}, it follows that $\calS\subset\Kbff^{1,n,1}(\calA)$ is closed by isomorphisms.
Together with {\rm (E0)}, we see that $(\A,\calS)$ satisfies {\rm ($n$-Ex0)}.
{\rm ($n$-Ex1)} follows from {\rm (E1)} and  {\rm (E1$\op$)}.
{\rm ($n$-Ex2)} and {\rm ($n$-Ex2$\op$)} are by \cite[Lemma~4.22 {\rm (2)}]{HLN1}.

$(2)\Rightarrow(1)$ For $n=1$, this is a consequence of Proposition~\ref{prop:exact_dg_=_1-exact_dg}. Let us consider the case $n\ge2$.
Suppose that $(\calA,\calS)$ is an $n$-exact dg-category with $n\ge2$. Let us check the conditions in Fact~\ref{fact:Prop4.23_HLN1}. {\rm (E0)} holds by Proposition~\ref{prop:any_split_belongs_to_S}. {\rm (E1)} and {\rm (E1$\op$)} are immediate from {\rm ($n$-Ex1)}. {\rm (EC$'$)} follows from {\rm ($n$-Ex0)} by Proposition~\ref{prop:closed_equiv}.
{\rm (E2${}^{\prime\mathrm{op}}$)}{\rm (i)} is immediate from {\rm ($n$-Ex2$\op$)}. {\rm (E2${}^{\prime\mathrm{op}}$)}{\rm (ii)} follows from Proposition~\ref{prop:pushout_inflation}.

It remains to show that condition {\rm (EI)} is satisfied. 
Let $X\up\in\calS$ and $n$-exact sequence $Y\up\in\Cbff^{n+2}(\calA)$ be arbitrary. Suppose that there is $g\up\in Z^0(\Cbff^{n+2}(\calA))(Y\up,X\up)$ such that $g^0$ and $g^{n+1}$ are isomorphisms in $\calA$. Then $d_X^0\ci g^0=g^1\ci d_Y^0$ holds in $\calA$, and the left hand side is an inflation by {\rm ($n$-Ex1)} and the dual of Corollary~\ref{cor:deflation_examples} {\rm (1)}. 
Since $n\ge2$, by a result of Klapproth \cite[Lemma~2.1]{Kla}, it follows that $d_Y^0$ is also a inflation. Thus $Y\up\in\calS$ by Lemma~\ref{lem:pushout_inflation}. Then Proposition~\ref{prop:pullpush_factorization} {\rm (4)} shows that $g\up$ is an isomorphism in $\Kbff^{n+2}(\calA)$. This shows {\rm (I2)}.

The other condition {\rm (I1)} can be checked in a similar way. Indeed, if there is $f\up\in Z^0(\Cbff^{n+2}(\calA))(X\up,Y\up)$ such that $f^0$ and $f^{n+1}$ are isomorphisms in $\calA$, then $d_Y^{n+1}$ becomes a deflation by the dual of \cite[Lemma~2.1]{Kla}, and hence the same proof as {\rm (I1)} can be applied.
\end{proof}

\begin{rem}\label{rem:comparison_with_n-exact}
For $n=1$, it is shown in \cite{NP1} that for a $1$-exangulated category (i.e., an extriangulated category), the condition that every inflation is a monomorphism and every deflation is an epimorphism is equivalent to the category corresponding to an exact category.

For $n \ge 2$, it is shown in \cite[Proposition~4.37]{HLN1} that, for an $n$-exangulated category satisfying the following {\rm (a)} and {\rm (b)}:
\begin{enumerate}
\item[(a)] If $b\ci a$ is an inflation, then so is $a$, 
\item[(b)] If $b\ci a$ is a deflation, then so is $b$,
\end{enumerate}
the same condition-namely that every inflation is a monomorphism and every deflation is an epimorphism-is equivalent to the category corresponding to an $n$-exact category.
However, as shown in the above-mentioned \cite[Lemma~2.1]{Kla}, these conditions are automatically satisfied for $n \ge 2$.
Therefore, this equivalence holds for any $n$-exangulated category.
\end{rem}

\normalcolor

\section{From $n$-cluster tilting subcategories of exact dg-categories}\label{section:Typical}

In this section, we give typical examples of $n$-exact dg-categories, which are obtained from $n$-cluster tilting subcategories of exact dg-categories. 

\subsection{Splicing constructions}

Before discussing $n$-cluster tilting subcategories, we describe a general method for constructing $(n+2)$-term sequences in a dg-category $\A$ by splicing.

As before, let $\A$ be a strictly connective and strictly additive dg-category.

\begin{lem}\label{lem:splising1'}
Let $L\up\in\Cbf_{[0,2]}(\A)$ and $X\up\in\Cbf_{[1,m]}(\A)$ be any pair of objects satisfying $L^2=X^1$.
Then, we obtain an object $(L\ast X)\up=L\up\ast X\up\in\Cbf_{[0,m]}(\A)$ defined by the following.
\begin{equation}\label{L*X}
 (L\ast X)^i=\begin{cases}
 L^i & \text{if}\ i\le 1\\
 X^i & \text{if}\ i\ge 2 \\
 0 & \text{otherwise}
 \end{cases}
 \ \ ,\quad
 d_{L\ast X}^{i,j}=\begin{cases}
 d_L^{i,j} & \text{if}\ j\le 1\\
 d_X^{1,j}\ci d_L^{i,2} & \text{if}\ i\le 1\ \text{and}\ j\ge 2\\
 d_X^{i,j} & \text{if}\ i\ge 2\\
  0 & \text{otherwise} 
 \end{cases}\ \ .
 \end{equation}
\end{lem}
\begin{proof}
By Proposition~\ref{prop:alphabetagamma}, we have sequences of morphisms
\begin{eqnarray*}
&X^1[-2]\xrightarrow{\al\up_{X,2}}\sig_{\ge 2}X\up\xrightarrow{\be\up_{X,2}}X\up\xrightarrow{\gam\up_{X,2}}X^1[-1],&\\
&(\sig_{\le 1}L\up)[-1]\xrightarrow{\al\up_{L,2}}L^2[-2]\xrightarrow{\be\up_{L,2}}L\up\xrightarrow{\gam\up_{L,2}}\sig_{\le 1}L\up&
\end{eqnarray*}
in $\Cbf(\A)$.
By assumption, we have $X^1[-2]=L^2[-2]$. Put $\wp\up=\al\up_{X,2}\ci\al\up_{L,2}\co(\sig_{\le 1}L\up)[-1]\to \sig_{\ge2}X\up$. Explicitly, we have
\[
\wp^{i,j}=\begin{cases}
d_X^{1,j}\ci d_L^{i-1,2} & \text{if}\ i\le 2\ \text{and}\ j\ge 2\\
0 & \text{otherwise}
\end{cases}
\]
by the definition of $\alpha\up_{X,2}$ and $\alpha\up_{L,2}$.
If we take the cone $C_{\wp}\up=\Cone \wp\up\in\Cbf(\A)$, then by Definition~\ref{dfn:cocone}, up to the identification via the canonical isomorphisms $A\oplus 0\cong A\cong 0\oplus A$ in $\A$ for any $A\in\ob(\A)$, we have
\[
 C_{\wp}^i=\begin{cases}
 L^i & \text{if}\ i\le 1\\
 X^i & \text{if}\ i\ge 2 \\
 0 & \text{otherwise}
 \end{cases}
 \ \ ,\quad
 d_{C_{\wp}}^{i,j}=\begin{cases}
 d_L^{i,j} & \text{if}\ j\le 1\\
 d_X^{1,j}\ci d_L^{i,2} & \text{if}\ i\le 1\ \text{and}\ j\ge 2\\
 d_X^{i,j} & \text{if}\ i\ge 2\\
  0 & \text{otherwise} 
 \end{cases}\ \ ,  \]  
which agrees with $(\ref{L*X})$.
\end{proof}

\normalcolor

\begin{nota}\label{notation:3t}
In this section, we frequently deal with one-sided twisted complexes of length~$3$, that is, objects of $\Cbf^3(\A)$.
As suggested by Proposition~\ref{prop:K111=H3t}, these correspond to $3$-term complexes via an equivalence of categories.
Therefore, from now on, through this identification,
we shall represent an object $X\up$ of $\Cbf^3(\A)$
by extracting its nontrivial terms, either in the form
\[
\begin{tikzcd}
X^0 \arrow[r,"{\scriptsize d_X^{0,1}}"] \arrow[rr, bend right=30, dashed, "{\scriptsize d_X^{0,2}}"'] & X^1 \arrow[r,"{\scriptsize d_X^{1,2}}"] & X^2
\end{tikzcd}
\]
as in $(\ref{diag_3t})$,
or, more simply, in the form
\[
X^0\xrightarrow{d_X^{0,1}}X^1\xrightarrow{d_X^{1,2}}X^2\quad (d_X^{0,2}\co X^0\to X^2).
\]
When written in this way, the morphism $d_X^{1,2}$ in parentheses
is understood to be of degree $-1$, and to satisfy $d_{\A}(d_X^{0,2})=-d_X^{1,2}\ci d_X^{0,1}$.

Similarly, for a morphism $f\up\in\Cbf^3(\A)(X\up,Y\up)$, via the above equivalence of categories, we represent it as a $6$-tuple 
$(f^{0,0},f^{1,1},f^{2,2},f^{0,1},-f^{1,2},f^{0,2})$ 
consisting of its nontrivial components. 
When represented diagrammatically, we write it as follows, 
\[
\begin{tikzcd}
X^0 \ar{r}{d_X^{0,1}} \ar{d}{f^{0,0}}& X^1 \ar{r}{d_X^{1,2}} \ar{d}{f^{1,1}} & X^2 \ar{d}{f^{2,2}} \\
Y^0 \ar{r}[swap]{d_Y^{0,1}} & Y^1 \ar{r}[swap]{d_Y^{1,2}} & Y^2
\Ar{1-1}{2-2}{"f^{0,1}",dashed}
\Ar{1-2}{2-3}{"-f^{1,2}",dashed}
\Ar{1-1}{1-3}{"d_X^{0,2}",bend left=30, dashed}
\Ar{2-1}{2-3}{"d_Y^{0,2}"',bend right=30, dashed}
\end{tikzcd}
\]
where the morphism $f^{0,2}$ is implicit. 
\normalcolor
\end{nota}

The following lemma provides a method for splicing $3$-term one-sided twisted complexes.
\begin{prop}\label{prop:splicing1}
Let
\[
M^{i-1}\xrightarrow{m^{i-1}} X^{i}\xrightarrow{e^{i}} M^{i}\quad(h^{i-1}\co M^{i-1}\to M^{i})
\]
be objects in $\Cbf^3(\A)$ for $1\le i\le n$.
For notational convenience, put $e^0=\id_{M^0}$ and $m^n=\id_{M^n}$.
Then, the following data define an object $X\up\in \Cbf^{n+2}(\A)$:
\begin{enumerate}
\renewcommand{\labelenumi}{(\roman{enumi})}
\item $(X\up)^i=X^i$ for $1\le i\le n$, $X^0=M^0$, and $X^{n+1}=M^n$,
\item The differentials of $X\up$ are defined by
$$d_X^{i,j}:=
\begin{cases}
m^i\circ e^i & \text{if}\ j=i+1\\
m^{j-1}\circ h^{j-2}\circ \cdots \circ h^{i}\circ e^i & \text{if}\ i+1<j\\
0 & \text{otherwise}
\end{cases}
$$
for $i,j\in\bbZ$.
\end{enumerate}
\end{prop}

\begin{proof}
This is shown by using Lemma~\ref{lem:splising1'} inductively. 
\end{proof}

\begin{lem}\label{lem:splicing_morphism}
Let $L\up, M\up\in\Cbf_{[0,2]}(\A)$ and $X\up,Y\up\in\Cbf_{[1,m]}(\A)$ be any pair of objects satisfying $L^2=X^1$ and $M^2=Y^1$. Let $\varphi\up\colon L\up\to M\up$ and $\psi\up\colon X\up\to Y\up$ be morphisms in $Z^0(\Cbf(\A))$ satisfying $\varphi^2=\psi^1$ in $\A$. Then, we obtain a morphism $(\varphi\ast\psi)\up=\varphi\up\ast\psi\up\colon L\up\ast X\up\to M\up\ast Y\up$ in $Z^0(\Cbf(\A))$ such that $\sig_{\leq 1}(\varphi\up\ast\psi\up)=\sig_{\leq 1}(\varphi\up)$ and $\sig_{\geq 2}(\varphi\up\ast\psi\up)=\sig_{\geq 2}(\psi\up)$ in $\Cbf(\A)$.
\end{lem}
\begin{proof}
As in the proof of Lemma~\ref{lem:splising1'}, we have $L\up\ast X\up=\Cone \wp\up$ and $M\up\ast Y\up=\Cone \zeta\up$ for 
$$\wp\up:=\al\up_{X,2}\ci\al\up_{L,2}\co(\sig_{\le 1}L\up)[-1]\to \sig_{\ge2}X\up \quad \text{and} \quad \zeta\up:=\al\up_{Y,2}\ci\al\up_{M,2}\co(\sig_{\le 1}M\up)[-1]\to \sig_{\ge2}Y\up.$$
On the other hand, by $\varphi^2=\psi^1$, the diagram 
$$\begin{tikzcd}
(\sig_{\le 1}L\up)[-1] \ar{r}{\wp\up} \ar{d}[swap]{(\sig_{\le 1}\varphi\up)[-1]} & \sig_{\ge2}X\up \ar{d}{\sig_{\ge2}\psi\up} \\
(\sig_{\le 1}M\up)[-1] \ar{r}{\zeta\up} & \sig_{\ge2}Y\up
\end{tikzcd}$$
is commutative in $\Cbf(\A)$.
Thus, by Proposition~\ref{prop:give_a_morph}, there exists a morphism $\varphi\up\ast\psi\up\colon L\up\ast X\up\to M\up\ast Y\up$ in $Z^0(\Cbf(\A))$ such that 
$$\sig_{\leq 1}(\varphi\up\ast\psi\up)=\sig_{\leq 1}(\varphi\up)\quad \text{ and }\quad \sig_{\geq 2}(\varphi\up\ast\psi\up)=\sig_{\geq 2}(\psi\up)$$ in $\Cbf(\A)$, as required.
\end{proof}

\begin{prop}\label{prop:splicing_morphism}
Assume that the following data are given.
\begin{enumerate}
\renewcommand{\labelenumi}{(\roman{enumi})}
\item Objects in $\Cbf^3(\A)$
\begin{eqnarray*}
&M^{i-1}\xrightarrow{m^{i-1}_X} X^i\xrightarrow{e^{i}_X} M^i
\quad (h_X^{i-1}\co M^{i-1}\to M^{i+1}),&\\ 
&N^{i-1}\xrightarrow{m^{i-1}_Y} Y^i\xrightarrow{e^{i}_Y} N^i \quad (h_Y^{i-1}\co N^{i-1}\to N^{i+1})&
\end{eqnarray*}
for $1\le i\le n$.

\item Morphisms $(a^{i-1},f^i,a^{i},c^{i-1}, b^{i},t^{i-1})$ in $Z^0(\Cbf^3(\A))$
 from
\[M^{i-1}\xrightarrow{m^{i-1}_X} X^i\xrightarrow{e^{i}_X} M^i
\quad (h_X^{i-1}\co M^{i-1}\to M^{i+1})\]
to 
\[
N^{i-1}\xrightarrow{m^{i-1}_Y} Y^i\xrightarrow{e^{i}_Y} N^i \quad (h_Y^{i-1}\co N^{i-1}\to N^{i+1}),
\]
for $1\le i\le n$.
\end{enumerate}
Let $X\up,Y\up\in \Cbf^{n+2}(\A)$ be the
objects obtained from data {\rm (i)} by splicing, using Proposition~\ref{prop:splicing1}.
Then, data {\rm (ii)} yields a morphism $f\up\colon X\up\to Y\up$ in $Z^0(\Cbf(\A))$ such that $f^{i,i}=f^i$ for $1\le i\le n$ and $f^{0,0}=a^0$, $f^{n+1,n+1}=a^n$.
\end{prop}
\begin{proof}
This is shown by using Lemma~\ref{lem:splicing_morphism} inductively. 
\end{proof}

\normalcolor

\subsection{$n$-cluster tilting subcategories of exact dg-categories}

In this subsection, we explain how to construct $n$-exangulated categories from $n$-cluster tilting subcategories of exact dg-categories. Throughout this subsection, we fix a strictly connective and strictly additive exact dg-category $\scE$, and assume that $H^0(\scE)$ is skeletally small. 
From now on, we simply call an admissible $1$-exact sequence
$X\up\in\Cbf^3(\scE)$ in the exact dg-category $\scE$ a \emph{conflation} in $\scE$, and represent it using the notation introduced in Notation~\ref{notation:3t}. Similarly for
morphisms between them.

By \cite[Theorem~4.26]{C1}, $H^0(\scE)$ has a structure of an extriangulated category $(H^0(\scE),\bbE,\sfr)$.
We note that, if
\[
X^0\xrightarrow{d_X^{0,1}}X^1\xrightarrow{d_X^{1,2}}X^2\quad (d_X^{0,2}\co X^0\to X^2).
\]
is a conflation in $\scE$, then the sequence
$
X^0\xrightarrow{\ovl{d}_X^{0}}X^1\xrightarrow{\ovl{d}_X^{1}}X^2
$ in $H^0(\scE)$ is a conflation in the extriangulated category $(H^0(\scE),\bbE,\sfr)$, in the sense of \cite[Definition~2.15]{NP1}.
Moreover, by using the coend over a skeleton $\mathrm{sk}\,H^0(\scE)$ of $H^0(\scE)$, one can define $\mathbb{E}^i$ inductively for any positive integer $i>0$ by
\[
\mathbb{E}^1=\mathbb{E}, \qquad
\mathbb{E}^i=\mathbb{E}^{i-1}\otimes_{\mathrm{sk}\,H^0(\scE)}\mathbb{E},
\]
in a manner similar to \cite{GNP}.
Then, by \cite[Theorem~3.5]{GNP}, for any conflation
$
A \xrightarrow{\mathbf{x}} B \xrightarrow{\mathbf{y}} C
$
in the extriangulated category $(H^0(\scE),\bbE,\sfr)$, one obtains long exact sequences continuing to the right
\scriptsize
\[
H^0\scE(\blank,A)\xrightarrow{\mathbf{x}\ci\blank}
H^0\scE(\blank,B)\xrightarrow{\mathbf{y}\ci\blank}
H^0\scE(\blank,C)\to
\bbE(\blank,A)\xrightarrow{\mathbf{x}\sas}
\bbE(\blank,B)\xrightarrow{\mathbf{y}\sas}
\bbE(\blank,C)\to
\bbE^2(\blank,A)\xrightarrow{\mathbf{x}\sas}
\bbE^2(\blank,B)\xrightarrow{\mathbf{y}\sas}\cdots
\]
\normalsize
and
\scriptsize
\[
H^0\scE(C,\blank)\xrightarrow{\blank\ci\mathbf{y}}
H^0\scE(B,\blank)\xrightarrow{\blank\ci\mathbf{y}}
H^0\scE(A,\blank)\to
\bbE(C,\blank)\xrightarrow{\mathbf{y}\uas}
\bbE(B,\blank)\xrightarrow{\mathbf{x}\uas}
\bbE(A,\blank)\to
\bbE^2(C,\blank)\xrightarrow{\mathbf{y}\uas}
\bbE^2(B,\blank)\xrightarrow{\mathbf{x}\uas}\cdots.
\]
\normalsize
These exact sequences can also be obtained from
\cite[Proposition~3.20]{HLN2} when $(H^0(\scE),\bbE,\sfr)$
has enough projectives and enough injectives.

\begin{rem}
If an extriangulated category is \emph{algebraic} (see \cite[p.3]{C1}), that is, if it admits an enhancement by an exact dg-category, then there always exists a strictly connective and strictly additive exact dg-category $\scE$ such that $H^0(\scE)$ is equivalent to the given extriangulated category.
\end{rem}

\begin{dfn}
We say that $\scE$ satisfies condition (WIC) if the following hold:
\begin{enumerate}
\item For any sequence of morphisms $A\xrightarrow{f}Y\xrightarrow{g}Z$ in $Z^0(\scE)$, if $g\circ f$ is a deflation, then $g$ is also a deflation.
\item  For any sequence of morphisms $A\xrightarrow{f}Y\xrightarrow{g}Z$ in $Z^0(\scE)$, if $g\circ f$ is an inflation, then $f$ is also an inflation.
\end{enumerate}
\end{dfn}

A morphism $f$ in $Z^0(\scE)$ is an
inflation in the exact dg-category $\scE$ if and only if its image
$\overline{f}$ is an inflation in the extriangulated category
$H^0(\scE)$. Similarly for deflations.
Thus, the above condition is equivalent to saying that the extriangulated category $H^0(\scE)$ satisfies condition (WIC) in the sense of \cite[Condition~5.8]{NP1}. By \cite[Proposition~2.7]{Kla}, this is also equivalent to that $H^0(\scE)$ is weakly idempotent complete.

\begin{dfn}\label{dfn:adm_applox}
Let $\scT\subset \scE$ be a full dg-subcategory. A morphism $\mathbf{f}\colon T\to Z$ in $H^0(\scE)$ is called a \emph{right $H^0(\scT)$-approximation} if $T\in\ob(H^0(\scT))$ and the induced morphism
$$H^0(\scE)(T',T)\xrightarrow{-\circ \mathbf{f}} H^0(\scE)(T',Z)$$
is surjective for every $T'\in\ob(H^0(\scT))$. 

A right $H^0(\scT)$-approximation $\mathbf{f}\colon T\to Z$ is called \emph{admissible} if, in addition, it is a deflation in $H^0(\scE)$.
Left $H^0(\scT)$-approximations and admissible left $H^0(\scT)$-approximations are defined dually.
\end{dfn}

\begin{dfn}\label{dfn:n-cluster-tilting-subcategory}
A full dg-subcategory $\scT\subset\scE$ is called an \emph{$n$-cluster tilting subcategory} if it satisfies the following conditions.
\begin{enumerate}
\item $H^0(\scT)$ is a strongly functorially finite subcategory of $H^0(\scE)$, that is, every object in $H^0(\scE)$ has both an admissible right $H^0(\scT)$-approximation and an admissible left $H^0(\scT)$-approximation. 
\item The following equalities hold:
\begin{eqnarray*}
\scT&=&\{M\in \scE\mid \bbE^i(M,H^0(\scT))=0 \text{ for } 0\le i\le n-1 \}\\
&=&\{M\in \scE\mid \bbE^i(H^0(\scT),M)=0 \text{ for } 0\le i\le n-1 \}.
\end{eqnarray*}
\end{enumerate}

If furthermore an $n$-cluster tilting subcategory $\scT$ of $\scE$ satisfies $H^{i}\scE(\scT,\scT)=0$ whenever $i\notin n\Z$, then we call $\scT$ an \emph{$n\Z$-cluster tilting subcategory}.
\end{dfn}

\begin{rem}
If $\scT$ is an $n$-cluster tilting subcategory of $\scE$, then $\scT$ automatically becomes strictly connective and strictly additive.
\end{rem}

\begin{rem}
If $H^0(\scE)$ has enough projectives and injectives, then condition~(1) in Definition~\ref{dfn:n-cluster-tilting-subcategory} is equivalent to saying that $H^0(\scT)$ is simply a functorially finite subcategory of $H^0(\scE)$. Thus, in this case, $\scT\subset\scE$ is an $n$-cluster tilting subcategory in the sense of Definition~\ref{dfn:n-cluster-tilting-subcategory} if and only if $H^0(\scT)\subset H^0(\scE)$ is an $n$-cluster tilting subcategory in the sense of
\cite[Definition~3.21]{HLN2}. See \cite[Remark~3.22]{HLN2} for details.
\end{rem}

\begin{rem}\label{rem:nZ-cluster}
If $\scE$ is a stable dg-category, then a full dg-subcategory $\scT$ is an $n$-cluster tilting subcategory in the above sense if and only if it is an $n\mathbb{Z}$-cluster tilting subcategory in the sense of \cite[Definition~1.3.1]{JKM}. This follows from \cite[Remark~1.3.5]{JKM}.
\end{rem}

For the rest of this subsection, we assume that $\scE$ satisfies  condition~(WIC), and 
fix an $n$-cluster tilting subcategory $\scT\subset\scE$.

\begin{lem}\label{lem:defl_approx}
Let $\mathbf{f}\colon T\to Z$ be a right $H^0(\scT)$-approximation in $H^0(\scE)$. Then $\mathbf{f}$ is a deflation.
\end{lem}
\begin{proof}
Let $\mathbf{g}\colon T'\to Z$ be an admissible right $H^0(\scT)$-approximation of $Z$. Then $\mathbf{g}$ factors through $\mathbf{f}$, since $\mathbf{f}$ is a right $H^0(\scT)$-approximation. Since $\mathbf{g}$ is a deflation, condition~(WIC) implies that $\mathbf{f}$ is also a deflation.
\end{proof}

\begin{prop}\label{prop:n-exactness_from_splice}
Let
\[
M^{i-1}\xrightarrow{m^{i-1}} X^{i}\xrightarrow{e^{i}} M^{i}\quad(h^i\co M^{i-1}\to M^i)
\]
be conflations in $\scE$ for $1\le i\le n$, and put $X^0=M^0$
, $X^{n+1}=M^n$.
Assume $X^i\in \scT$ for $0\le i\le n+1$.
Then, the following holds.
\begin{enumerate}
  \item
  $\ovl{e^i}$ is a right $H^0(\scT)$-approximation, for any $1\le i\le n-1$.

  \item
  $X\up\in\Cbf^{n+2}(\scT)$ obtained by Proposition~\ref{prop:splicing1} is left $n$-exact as a one-sided twisted complex in $\scT$.
\end{enumerate}
The dual statements of $(1)$ and $(2)$ also hold. In particular, $X\up$ is $n$-exact in $\scT$.
\end{prop}
\begin{proof}
First, we prove~(1). It suffices to show that $\E(T,M^{n-i})=0$ for every $2\le i \le n$ and every $T\in\scT$. Let $T\in\scT$ be any object and $2\le i\le n$. Then we have the following isomorphisms:
$$\E(T,M^{n-i})\cong \E^2(T,M^{n-i-1})\cong \cdots \cong \E^{n-i}(T,M^1)\cong \E^{n-i+1}(T,M^0)$$
Since $M^0\in\scT$ and $\E^{n-i+1}(H^0(\scT),H^0(\scT))=0$, it follows that $\E(T,M^{n-i})=0$. 

Next, we prove~(2). For $1\le i\le n$, let $Y_{(i)}\up\in \Cbf_{[n-i,n+1]}(\scE)$ be the object obtained by splicing
\[
M^{p-1}\xrightarrow{m^{p-1}} X^{p}\xrightarrow{e^{p}} M^{p}\quad(h^p\co M^{p-1}\to M^p)
\]
for $n-i+1\le p\le n$ by using Proposition~\ref{prop:splicing1}.
Note that the terms of $Y_{(i)}\up$ are given by
$$(Y_{(i)}\up)^j=\begin{cases}
0 & \text{if}\ j< n-i\\
M^{n-i} & \text{if}\ j=n-i\\
X^{j} & \text{if}\ n-i<j\le n+1\\
0 & \text{if}\ j>n+1
\end{cases}$$
It is easy to see that $\sig_{\ge n-i+1}Y_{(i)}\up=\sig_{\ge n-i+1}X\up$ and $Y_{(n)}\up=X\up$. We define a morphism
$\eta_{(i+1)}\up\colon Y\up_{(i+1)}\to Y\up_{(i)}$ in $\Cbf(\scE)$ as follows:
$$\eta^{j,j}_{(i+1)}:=
\begin{cases}
    e^{n-i} & \text{ if } j=n-i \\
    \id_{X^{j}} & \text{ if } n-i+1 \le j\le n\\
    \id_{M^n} & \text{ if } j=n+1\\
    0 & \text{otherwise}
\end{cases}
\quad, \quad
\eta^{j,j+1}_{(i+1)}:=
\begin{cases}
    h^{n-i-1} & \text{ if } j=n-i-1 \\
    0 & \text{otherwise}
\end{cases}
$$
$$\text{and } \eta^{j,k}_{(i+1)}=0 \text{ for all } k\ne j,j+1.$$
It is straightforward to check that $\eta_{(i+1)}\up$ is a closed morphism of degree $0$.
We claim that $\eta_{(i+1)}\up\colon Y\up_{(i+1)}\to Y\up_{(i)}$ induces an isomorphism
$$\Kbf(\scE)(T[-j],Y\up_{(i+1)})\xrightarrow{\cong}\Kbf(\scE)(T[-j],Y\up_{(i)})$$
for every $T\in\scT$ and every $j\in\mathbb{Z}$.

By construction, the object $\Cone(\sigma_{\le n-i}(\eta_{(i+1)}))$ corresponds to the following one-sided twisted complex:
$$\begin{tikzcd}
    \cdots & 0 & M^{n-i-1} & X^{n-i} & M^{n-i} & 0 & \cdots
    \Ar{1-1}{1-2}{}
    \Ar{1-2}{1-3}{}
    \Ar{1-3}{1-4}{"m^{n-i-1}"'}
    \Ar{1-4}{1-5}{"e^{n-i}"'}
    \Ar{1-5}{1-6}{}
    \Ar{1-6}{1-7}{}
    \Ar{1-3}{1-5}{bend left =20, dashed, "h^{n-i-1}"}
\end{tikzcd}$$
where the dashed arrow denotes a degree $-1$ differential and $M^{n-i}$ is placed in degree $n-i$ in the above diagram. Denote this one-sided twisted complex by $C\up_{(i+1)}$. Since
$$M^{n-i-1}\to X^{n-i}\to M^{n-i}\quad (h^{n-i-1}\colon M^{n-i-1}\to M^{n-i})$$
is a conflation, we have
$$\Kbf(\scE)(T[-j],C\up_{(i+1)})=0$$
for every $T\in\scT$ and every $j\neq n-i$. Moreover, by~(1), we also have $\Kbf(\scE)(T[i-n],C\up_{(i+1)})=0$. Hence
$$\Kbf(\scE)(T[-j],C\up_{(i+1)})=0$$
for every $T\in\scT$ and every $j\in\mathbb{Z}$.

By the octahedral axiom, we obtain the following diagram:

$$
\begin{tikzcd}
\sig_{\ge n-i+1}X\up & \sig_{\ge n-i+1}X\up & {} \\
Y\up_{(i+1)}& Y\up_{(i)} & C\up_{(i+1)} \\
\sig_{\le n-i} Y\up_{(i+1)} & \sig_{\le n-i} Y\up_{(i)} & C\up_{(i+1)}
\Ar{1-1}{1-2}{equal}
\Ar{2-1}{2-2}{"\eta_{(i+1)}\up"}
\Ar{2-2}{2-3}{}
\Ar{3-1}{3-2}{"\sig_{\le n-i}\eta_{(i+1)}\up"}
\Ar{3-2}{3-3}{}
\Ar{1-1}{2-1}{}
\Ar{2-1}{3-1}{}
\Ar{1-2}{2-2}{}
\Ar{2-2}{3-2}{}
\Ar{2-3}{3-3}{equal}
\end{tikzcd}
$$
In particular, we obtain a triangle
$$Y\up_{(i+1)}\xrightarrow{\eta_{(i+1)}\up} Y\up_{(i)} \to C\up_{(i+1)}\to Y\up_{(i+1)}[1]$$
Since $\Kbf(\scE)(T[-j],C\up_{(i+1)})=0$ for every $j\in\mathbb{Z}$, the morphism $\eta_{(i+1)}\up$ induces an isomorphism
$$\Kbf(\scE)(T[-j],Y\up_{(i+1)})\xrightarrow{\cong}\Kbf(\scE)(T[-j],Y\up_{(i)})$$
for every $T\in\scT$ and every $j\in \mathbb{Z}$. In particular, iterating these isomorphisms, we obtain
$$\Kbf(\scE)(T[-j],Y\up_{(n)})\cong \cdots \cong \Kbf(\scE)(T[-j],Y\up_{(1)})=0$$
for every $T\in\scT$ and every $j\leq n$, by the definition of $Y\up_{(1)}$. Since $X\up=Y\up_{(n)}$, it follows that $X\up$ is left $n$-exact.
\end{proof}

\begin{cor}\label{cor:defl_are_defl}
Let $g\in Z^0(\scE)(Y,Z)$ be a deflation in $\scE$. Then there exists an $n$-exact sequence $X\up\in\Kbf^{n+2}(\scT)$ such that $X^n=Y$, $X^{n+1}=Z$ and $d_X^{n,n+1}=g$. The dual statement also holds. 
\end{cor}
\begin{proof}
Since $g$ is a deflation, we can construct conflations as in the assumption of Proposition~\ref{prop:n-exactness_from_splice} with $X^n=Y$, $M^n=Z$ and $e^n=g$. Proposition~\ref{prop:n-exactness_from_splice} then shows that the resulting one-sided twisted complex $X\up$ is an $n$-exact sequence in $\scT$ such that 
$d_X^{n,n+1}=g$.
\end{proof}

\begin{prop}\label{prop:chara_conf}
The following conditions are equivalent for $X\up\in\Kbf^{1,n,1}(\scT)$:
\begin{enumerate}
    \item $X\up$ is isomorphic to a one-sided twisted complex obtained by Proposition~\ref{prop:n-exactness_from_splice} in $\Kbf^{1,n,1}(\scT)$. 
  \item $X\up$ is a left $n$-exact sequence in $\scT$, and $d^{n,n+1}_X$ is a deflation in $\scE$.
  \item $X\up$ is a right $n$-exact sequence in $\scT$, and $d^{0,1}_X$ is an inflation in $\scE$. 
  \item $X\up$ is an $n$-exact sequence in $\scT$, and $d^{0,1}_X$ is an inflation and $d^{n,n+1}_X$ is a deflation in $\scE$.
\end{enumerate}
\end{prop}
\begin{proof}
(1) $\Rightarrow$ (2) follows by Proposition~\ref{prop:n-exactness_from_splice}, Remark~\ref{rem:isom_n-ex} and condition (WIC). 

Let us show that (2) $\Rightarrow$ (1). Assume that $X\up$ is left $n$-exact and that $d_X^{n,n+1}\colon X^n\to X^{n+1}$ is a deflation. Then we can take the following conflation:
$$M^{n-1}\xrightarrow{m^{n-1}} X^{n}\xrightarrow{e^n} M^n\quad (h^{n-1}\colon M^{n-1}\to M^{n})$$
 where $M^n=X^{n+1}$ and $e^n=d_X^{n,n+1}$. As in the proof of Proposition~\ref{prop:n-exactness_from_splice}, let $Y\up_{(1)}\in \Kbf_{[n-1,n+1]}(\scE)$ be the one-sided twisted complex defined by the above conflation. 

Since $\Kbf(\scE)(E[-i],Y\up_{(1)})=0$ holds for every $i\leq n$ and every $E\in\scE$, there exists a morphism $\varphi\up_{(1)}\colon \sig_{< n} X\up \to \sig_{< n} Y\up_{(1)}$ which makes the following diagram in $\Kbf(\scE)$ commutative:
$$
\begin{tikzcd}
\sigma_{< n} X\up[-1] & \sigma_{\ge n}X\up & X\up & \sigma_{< n}X\up\\
\sigma_{< n} Y\up_{(1)}[-1] & \sigma_{\ge n}Y\up_{(1)} & Y\up_{(1)} & \sigma_{<n}Y\up_{(n-1)}
\Ar{1-1}{1-2}{}
\Ar{1-2}{1-3}{}
\Ar{1-3}{1-4}{}
\Ar{2-1}{2-2}{}
\Ar{2-2}{2-3}{}
\Ar{2-3}{2-4}{}
\Ar{1-1}{2-1}{"\varphi\up_{(1)}{[}-1{]}"}
\Ar{1-2}{2-2}{equal}
\Ar{1-4}{2-4}{"\varphi\up_{(1)}"}
\end{tikzcd}
$$
Thus, by Proposition~\ref{prop:give_a_morph}, we have a morphism $\psi\up_{(1)}$ so that $\ovl{\psi_{(1)}}^{i}$ is the identity morphism for $i=n,n+1$ and the following diagram in $\Kbf(\scE)$ commutes:
$$
\begin{tikzcd}
\sigma_{< n} X\up[-1] & \sigma_{\ge n}X\up & X\up & \sigma_{< n}X\up\\
\sigma_{< n} Y\up_{(1)}[-1] & \sigma_{\ge n}Y\up_{(1)} & Y\up_{(1)} & \sigma_{<n}Y\up_{(n-1)}
\Ar{1-1}{1-2}{}
\Ar{1-2}{1-3}{}
\Ar{1-3}{1-4}{}
\Ar{2-1}{2-2}{}
\Ar{2-2}{2-3}{}
\Ar{2-3}{2-4}{}
\Ar{1-1}{2-1}{"\varphi\up_{(1)}{[}-1{]}"}
\Ar{1-2}{2-2}{equal}
\Ar{1-3}{2-3}{"\psi\up_{(1)}"}
\Ar{1-4}{2-4}{"\varphi\up_{(1)}"}
\end{tikzcd}
$$

We will show that this argument can be repeated to construct an isomorphism $\psi\up_{(n)}\colon X\up\to Y\up_{(n)}$ in $\Kbf^{1,n,1}(\scT)$. To prove this, we establish the following claim by induction on $i$.

\begin{claim}
For each $1\le i\le n-2$, we can construct the following data:
\begin{itemize}
\item[(i)] Conflations in $\scE$
$$M^{k-1}\xrightarrow{m^{k}} X^{k}\xrightarrow{e^{k}} M^k \quad (h^{k-1}\colon M^{k-1}\to M^{k})$$
 for each $n-i+1\le k\le n-1$ such that:
\begin{itemize}
\item $\ovl{e}^{k}\colon X^{k}\to M^{k}$ is a right $H^0(\scT)$-approximation for every $n-i+1\le k\le n-1$.
\item $X^{n-i}, \ldots, X^n, M^n=X^{n+1}$ belong to $\scT$ and $d^{n,n+1}_{X}=e^n$. 
\end{itemize}
\item[(ii)] A morphism of one-sided twisted complexes $\psi\up_{(i)}\colon X\up\to Y\up_{(i)}$ such that $\ovl{\psi_{(i)}}^{j}=\id_{X^j}$ in $H^0(\scE)$ for every $n-i+1\le j\le n+1$, where $Y\up_{(i)}$ is the one-sided twisted complex obtained from the conflations in \textup{(i)}, as in the proof of Proposition~\ref{prop:n-exactness_from_splice}.
\end{itemize}
\end{claim}
\begin{proof}[Proof of Claim]
We have already established the case $i=1$. Assume that the claim holds for $i-1$. We first construct the conflations in (i). By the induction hypothesis, we have the following diagram:
$$
\begin{tikzcd}
\sigma_{\leq n-i+1} X\up[-1] & \sigma_{> n-i+1}X\up & X\up & \sigma_{\leq n-i+1}X\up\\
\sigma_{\leq n-i+1} Y\up_{(i-1)}[-1] & \sigma_{> n-i+1}Y\up_{(i-1)} & Y\up_{(i-1)} & \sigma_{\leq n-i+1}Y\up_{(i-1)}
\Ar{1-1}{1-2}{}
\Ar{1-2}{1-3}{}
\Ar{1-3}{1-4}{}
\Ar{2-1}{2-2}{}
\Ar{2-2}{2-3}{}
\Ar{2-3}{2-4}{}
\Ar{1-1}{2-1}{"\sig_{\leq n-i+1}\psi\up_{(i-1)}{[}-1{]}"}
\Ar{1-2}{2-2}{"\sig_{>n-i+1}\psi\up_{(i-1)}"}
\Ar{1-3}{2-3}{"\psi\up_{(i-1)}"}
\Ar{1-4}{2-4}{"\sig_{\leq n-i+1}\psi\up_{(i-1)}"}
\end{tikzcd}
$$
Then $\sig_{> n-i+1}\psi\up_{(i-1)}$ is an isomorphism in $\Kbf(\scE)$ by using Proposition~\ref{prop:give_a_morph} inductively, since $\ovl{\psi_{(i-1)}}^{j}=\id_{X^j}$ in $H^0(\scE)$ for every $n-i+2\le j\le n+1$. Therefore, for every $T\in\scT$, we obtain the following commutative diagram:
$$\begin{tikzcd}
\Kbf(\scE)(T[-n+i-1],X^{n-i+1}[-n+i-1]) & \Kbf(\scE)(T[-n+i-1],\sig_{\le n-i+1}X\up) \\
\Kbf(\scE)(T[-n+i-1],M^{n-i+1}[-n+i-1]) & \Kbf(\scE)(T[-n+i-1],\sig_{\le n-i+1}Y\up_{(i-1)})
\Ar{1-1}{1-2}{}
\Ar{2-1}{2-2}{}
\Ar{1-1}{2-1}{}
\Ar{1-2}{2-2}{}
\end{tikzcd}$$
where the horizontal morphisms are induced by the morphisms 
$$\be_{\sig_{\le n-i+1}X\up,n-i+1}\up \quad\text{and}\quad\be_{\sig_{\le n-i+1}Y\up_{(i-1)},n-i+1}\up$$
associated with the hard truncations. 

The bottom horizontal morphism is also an isomorphism, since the components of $\sig_{\le n-i+1}Y\up_{(i-1)}$ are concentrated in degrees less than or equal to $n-i+1$ by Lemma~\ref{lem:comparison_H}. Similarly, we can conclude that the top horizontal map is surjective. We will see that the right vertical map is an isomorphism. 
First, we show that $\psi\up_{(i-1)}$ induces an isomorphism
\begin{equation}\label{eq:eqeq}
\Kbf(\scE)(T[-j],X\up)\xrightarrow{\cong}\Kbf(\scE)(T[-j],Y\up_{(i-1)})
\end{equation}
for every $T\in\scT$ and every $j\in\mathbb{Z}$. 
We already know that the following holds:
$$\Kbf(\scE)(T[-j],Y\up_{(i-1)})=0,\quad \Kbf(\scE)(T[-j],X\up)=0$$
for every $j\neq n+1$ and every $T\in\scT$ by the proof of Proposition~\ref{prop:n-exactness_from_splice}.
It suffices to show that the following morphism is an isomorphism:
\begin{equation}\label{eq:eqeqeq}
\Kbf(\scE)(T[-n-1],X\up)\xrightarrow{\cong}\Kbf(\scE)(T[-n-1],Y\up_{(n-1)})
\end{equation}
This immediately follows from the fact that $\ovl{\psi_{(i-1)}}^{n+1}$ and $\ovl{\psi_{(i-1)}}^{n}$ are identity morphisms in $H^0(\scE)$. Thus, $\psi\up_{(i-1)}$ induces an isomorphism \eqref{eq:eqeq} for every $T\in\scT$ and every $j\in\mathbb{Z}$. Since $\sig_{> n-i+1}\psi\up_{(i-1)}$ is an isomorphism in $\Kbf(\scE)$, we have 
$$\Kbf(\scE)(T[-n+i-1],\sig_{\le n-i+1}X\up)\xrightarrow{\cong} \Kbf(\scE)(T[-n+i-1],\sig_{\le n-i+1}Y\up_{(i-1)}).$$
By these arguments, the following morphism is an isomorphism:
$$\Kbf(\scE)(T[-n+i-1],X^{n-i+1}[-n+i-1])\xrightarrow{}\Kbf(\scE)(T[-n+i-1],M^{n-i+1}[-n+i-1])$$
and thus, $\ovl{\psi_{(i-1)}}^{n-i+1}\colon X^{n-i+1}\to M^{n-i+1}$ is a right $H^0(\scT)$-approximation.
Let us denote by $e^{n-i+1}:=\psi_{(i-1)}^{n-i+1,n-i+1}$. 
By Lemma~\ref{lem:defl_approx}, $e^{n-i+1}$ is a deflation. Therefore, we obtain a conflation
$$M^{n-i}\xrightarrow{m^{n-i}} X^{n-i+1}\xrightarrow{e^{n-i+1}} M^{n-i+1}\quad(h^{n-i}\colon M^{n-i}\to M^{n-i+1})$$
 in $\scE$. 

Next, we construct a morphism $\psi\up_{(i)}\colon X\up\to Y\up_{(i)}$. We have a morphism $\xi\up_{(i)}\colon \sig_{\ge n-i+1}X\up\to \sig_{\ge n-i+1}Y\up_{(i)}$, defined as follows:
$$
\xi_{(i)}^{i,j}:=
\begin{cases}
\id_{X^{n-i+1}} & \text{if } i=j=n-i+1\\
\psi_{(i-1)}^{i,j} & \text{if } n-i+1< i\le j\le n+1\\
0 & \text{otherwise}
\end{cases}
$$
By the construction of $Y\up_{(i)}$, one checks that $\xi\up_{(i)}\colon \sig_{\ge n-i+1}X\up\to \sig_{\ge n-i+1}Y\up_{(i)}$ is in $Z^0(\Cbf(\scE))$. By assumption, $\ovl{\xi_{(i)}}^{j}=\id_{X^j}$ for every $n-i+1\le j\le n+1$. Since $Y\up_{(i)}$ satisfies
$$\Kbf(\scE)(T[-j],Y\up_{(i)})=0$$
for every $T\in\scT$ and every $j\le n-i+1$, there exists a morphism
$$\varphi\up_{(i)}\colon \sig_{< n-i+1}X\up\to \sig_{< n-i+1}Y\up_{(i)}$$
that makes the following diagram commute in $\Kbf(\scE)$:
$$
\begin{tikzcd}
\sigma_{< n-i+1} X\up[-1] & \sigma_{\ge n-i+1}X\up & X\up & \sigma_{< n-i+1}X\up\\
\sigma_{< n-i+1} Y\up_{(i)}[-1] & \sigma_{\ge n-i+1}Y\up_{(i)} & Y\up_{(i)} & \sigma_{< n-i+1}Y\up_{(i)}
\Ar{1-1}{1-2}{}
\Ar{1-2}{1-3}{}
\Ar{1-3}{1-4}{}
\Ar{2-1}{2-2}{}
\Ar{2-2}{2-3}{}
\Ar{2-3}{2-4}{}
\Ar{1-1}{2-1}{"\varphi\up_{(i)}{[}-1{]}"}
\Ar{1-2}{2-2}{"\xi\up_{(i)}"}
\Ar{1-4}{2-4}{"\varphi\up_{(i)}"}
\end{tikzcd}
$$
Hence, by Corollary~\ref{cor:give_a_morph}, we can extend $\varphi\up_{(i)}[-1]$ and $\xi\up_{(i)}$ to a morphism $\psi\up_{(i)}\colon X\up\to Y\up_{(i)}$ that makes the following diagram commute.
$$
\begin{tikzcd}
\sigma_{< n-i+1} X\up[-1] & \sigma_{\ge n-i+1}X\up & X\up & \sigma_{< n-i+1}X\up\\
\sigma_{< n-i+1} Y\up_{(i)}[-1] & \sigma_{\ge n-i+1}Y\up_{(i)} & Y\up_{(i)} & \sigma_{< n-i+1}Y\up_{(i)}
\Ar{1-1}{1-2}{}
\Ar{1-2}{1-3}{}
\Ar{1-3}{1-4}{}
\Ar{2-1}{2-2}{}
\Ar{2-2}{2-3}{}
\Ar{2-3}{2-4}{}
\Ar{1-1}{2-1}{"\varphi\up_{(i)}{[}-1{]}"}
\Ar{1-2}{2-2}{"\xi\up_{(i)}"}
\Ar{1-3}{2-3}{"\psi\up_{(i)}"} 
\Ar{1-4}{2-4}{"\varphi\up_{(i)}"}
\end{tikzcd}
$$
By the construction of $\psi\up_{(i)}$, we have $\sig_{\ge n-i+1}\psi\up_{(i)}=\xi\up_{(i)}$; thus, $\ovl{\psi_{(i)}}^j=\id_{X^j}$ for every $n-i+1\le j\le n+1$.
\end{proof}

By this claim, we have a morphism of one-sided twisted complexes $\psi\up_{(n-1)}\colon X\up\to Y\up_{(n-1)}$ such that $\ovl{\psi_{(n-1)}}^j=\id$ for any $2\le j\le n+1$. Thus, we have the following commutative diagram in $\Kbf(\scE)$:
$$
\begin{tikzcd}
\sigma_{< 2} X\up[-1] & \sigma_{\ge 2}X\up & X\up & \sigma_{< 2}X\up\\
\sigma_{< 2} Y\up_{(n-1)}[-1] & \sigma_{\ge 2}Y\up_{(n-1)} & Y\up_{(n-1)} & \sigma_{< 2}Y\up_{(n-1)}
\Ar{1-1}{1-2}{}
\Ar{1-2}{1-3}{}
\Ar{1-3}{1-4}{}
\Ar{2-1}{2-2}{}
\Ar{2-2}{2-3}{}
\Ar{2-3}{2-4}{}
\Ar{1-1}{2-1}{"\sig_{< 2}\psi\up_{(n-1)}{[}-1{]}"}
\Ar{1-2}{2-2}{"\sig_{\ge 2}\psi\up_{(n-1)}"}
\Ar{1-3}{2-3}{"\psi\up_{(n-1)}"} 
\Ar{1-4}{2-4}{"\sig_{< 2}\psi\up_{(n-1)}"}
\end{tikzcd}
$$

Since $\sig_{\ge 2}\psi\up_{(n-1)}$ is an isomorphism in $\Kbf(\scE)$, we have 
$$\CoCone \psi\up_{(n-1)}\cong \CoCone\sig_{<2}\psi\up_{(n-1)}.$$
Note that $\CoCone\sig_{<2}\psi\up_{(n-1)}\in\ob(\Cbf_{[0,2]}(\scE))$ is illustrated as follows:
\begin{equation}\label{eq:q}
\begin{tikzcd}
X^0 & X^1 & M^1
\Ar{1-1}{1-2}{"d_X^{0,1}"'}
\Ar{1-2}{1-3}{"\psi_{(n-1)}^{1,1}"'}
\Ar{1-1}{1-3}{bend left=20, dashed, "\psi_{(n-1)}^{0,1}"}
\end{tikzcd}
\end{equation}
We show that $\CoCone \varphi\up_{(n-1)}$ becomes a conflation in $\scE$. By the same argument as in the proof of the previous claim, we can show that the morphism $e^1:=\psi_{(n-1)}^{1,1}$ induce a right $H^0(\scT)$-approximation $\ovl{e^1}$ and thus $e^1$ is a deflation in $\scE$. Thus, we have a conflation $Y_0\up\in\Kbf^3(\scE)$ in $\scE$ as follows:
\begin{equation}\label{eq:e}
M^0\xrightarrow{m^0} X^1\xrightarrow{e^1} M^1, \quad (h^0\colon M^0\to M^1)
\end{equation}
By construction of $Y\up_{(n-1)}$, 
we already have conflations 
\begin{equation}
\label{eq:Y_0}
M^{i-1}\xrightarrow{m^{i-1}} X^i \xrightarrow{e^i} M^{i}\quad (h^{i-1}\colon M^{i-1}\to M^{i})
\end{equation}
and thus, $\E^{i}(M^0,H^0(\scT))=0$ for any $1\leq i\leq n-1$ and  $M^0\in\scT$. 

Then, there exists a morphism $(f,\id,\id,s_0,0,t)$ from \eqref{eq:q} to \eqref{eq:e} in $\Kbf^{1,1,1}(\scE)$. If we can prove that $\ovl{f}\colon X^0\to M^0$ is an isomorphism in $H^0(\scE)$, then by Proposition~\ref{prop:give_a_morph}, the morphism $(f,\id,\id,s_0,0,t)$ is an isomorphism in $\Kbf(\scE)$. This implies that $\CoCone \varphi\up_{(n-1)}$ is left exact and, moreover, a conflation in $\scE$. 

We can see that $\psi_{(n-1)}\up$ induces the following isomorphim in $\Kbf(\scE)$ by a similar argument to proof \eqref{eq:eqeqeq}:
\begin{equation}\label{eq:eq}
\Kbf(\scE)(T[-j],X\up)\xrightarrow{\cong}\Kbf(\scE)(T[-j],Y\up_{(n-1)})
\end{equation}
for every $T\in\scT$ and every $j\in\mathbb{Z}$. 
Since $\CoCone \psi\up_{(n-1)}\cong \CoCone \sig_{<2}\psi\up_{(n-1)}$ and $\psi\up_{(n-1)}$ induces an isomorphism \eqref{eq:eq}, we have
$$\Kbf(\scE)(T[-j],\CoCone \sig_{<2}\psi\up_{(n-1)})=0$$
for every $T\in\scT$ and every $j\in\mathbb{Z}$. 
We also have 
$$\Kbf(\scE)(T[-j],Y_0\up)=0$$
for every $T\in\scT$ and every $j\in\mathbb{Z}$, since $Y_0\up$ is a conflation in $\scE$ and $\ovl{e^1}$ is a right $H^0(\scT)$-approximation. The morphism $(f,\id,\id,s_0,0,t)$ induces the following commutative diagram in $\Kbf(\scE)$:
$$\begin{tikzcd}
\Kbf(\scE)(T[-1],X^0[-1]) & \Kbf(\scE)(T[-1],\sig_{\ge 1}\CoCone \sig_{<2}\psi\up_{(n-1)}) \\
\Kbf(\scE)(T[-1],M^0[-1]) & \Kbf(\scE)(T[-1],\sig_{\ge 1}Y_0\up)
\Ar{1-1}{1-2}{"\cong"}
\Ar{2-1}{2-2}{"\cong"}
\Ar{1-1}{2-1}{"f{[}-1{]}\circ \blank"}
\Ar{1-2}{2-2}{equal}
\end{tikzcd}$$
Thus, the morphism $\ovl{f}\colon X^0\to M^0$ is an isomorphism in $H^0(\scT)$ and hence also in $H^0(\scE)$. This implies that $\CoCone \sig_{<2}\psi$ is a conflation in $\scE$. 

Hence, by Lemma~\ref{lem:splising1'}, we obtain a one-sided twisted complex $$Y\up_{(n)}:=(\CoCone \sig_{<2}\psi\up_{(n-1)})\ast Y\up_{(n-1)}.$$ 
We construct a morphism $\psi\up_{(n)}\colon X\up\to Y\up_{(n)}$ as follows:
$$\psi_{(n)}^{0,j}:=
\begin{cases}
\id_{X^0} & \text{if } j=0\\
0 & \text{if } j=1\\
\psi_{(n-1)}^{0,j} & \text{if } j\leq 2\\
0 & \text{otherwise}
\end{cases}, 
\quad
\psi_{(n)}^{i,j}:=
\begin{cases}
\id_{X^1} & \text{if } i=j=1\\
\psi_{(n-1)}^{i,j} & \text{if } i< j\le n+1\\
0 & \text{otherwise}
\end{cases}
$$
for $i\ge 1$. It is straightforward to check that $\psi\up_{(n)}$ is a morphism in $Z^0(\Cbf(\scE))$ and that $\ovl{\psi_{(n)}}^j=\id$ for every $0\le j\le n+1$. Thus, $\psi\up_{(n)}$ is an isomorphism in $\Kbf^{1,n,1}(\scT)$ by Corollary~\ref{cor:htpy_modif2}. This shows that (2) implies (1).

The proof of the equivalence between (1) and (3) is dual to the argument above. The equivalence between (1) and (4) follows immediately from the equivalence between (1) and (2) and that between (1) and (3).
\end{proof}

\begin{prop}\label{prop:const_n_pb}

For any conflation 
\begin{equation}\label{conf_N}
N'\xrightarrow{m_Y}Y\xrightarrow{e_Y}N\quad(h_Y\co N'\to N)
\end{equation}
in $\scE$ with $Y\in\scT$ and any $a\in Z^0(\scE)(M,N)$, the following holds.
\begin{enumerate}
\item We can construct the conflation
\begin{equation}\label{conf_M}
M'\xrightarrow{m_X} X\xrightarrow{e_X} M \quad (h_X\colon M'\to M)
\end{equation}
where $X\in\scT$
and a morphism of conflations 
$(a',f,a,c,b,t)$ from $(\ref{conf_M})$ to $(\ref{conf_N})$.
\item If $e_Y$ is a right $\scT$-approximation, then $e_X$ can be also taken as a right $\scT$-approximation.
\end{enumerate}
\normalcolor
\end{prop}
\begin{proof}

{\rm (1)} Since $\scE$ is an exact dg-category, we can take a conflation
\[
N'\xrightarrow{m_L} L\xrightarrow{e_L} M \quad (h_L\colon N'\to M)
\]
and a $1$-pullback morphism $(\id, f_1, a, c_1, b_1, t_1)$ in $\Cbf^3(\scE)$, which can be depicted as follows,
\[
\begin{tikzcd}
N' \ar{r}{m_L} \ar[d, equals]& L \ar{r}{e_L} \ar{d}{f_1} & M \ar{d}{a} \\
N' \ar{r}[swap]{m_Y} & Y \ar{r}[swap]{e_Y} & N
\Ar{1-1}{2-2}{"c_1",dashed}
\Ar{1-2}{2-3}{"b_1",dashed}
\Ar{1-1}{1-3}{"h_L",bend left=30, dashed}
\Ar{2-1}{2-3}{"h_Y"',bend right=30, dashed}
\end{tikzcd}
\]
where $t_1$ is implicit.
Take $f_0\colon X\to L$ in $Z^0(\scE)$ such that $\ovl{f_0}$ is an admissible right $H^0(\scT)$-approximation.
By Lemma~\ref{lem:defl_approx}, it is a deflation. Put $e_X=e_L\ci f_0$.
Since $e_X$ is a deflation, we obtain a conflation $(\ref{conf_M})$
and a morphism of conflations $(a', f_0, \id, c_0, b_0, t_0)$ in $\Cbf^3(\scE)$, which can be depicted as below,
\[
\begin{tikzcd}
M'\ar{r}{m_X} \ar{d}[swap]{a'}& X \ar{r}{e_X} \ar{d}{f_0} & M \ar[d, equals] \\
N'\ar{r}[swap]{m_L} & L \ar{r}[swap]{e_L} & M
\Ar{1-1}{2-2}{"c_0",dashed}
\Ar{1-2}{2-3}{"b_0",dashed}
\Ar{1-1}{1-3}{"h_X",bend left=30, dashed}
\Ar{2-1}{2-3}{"h_L"',bend right=30, dashed}
\end{tikzcd}
\]
where $t_0$ is implicit.
Then,
\[
M'\xrightarrow{\begin{bmatrix}a'\\ m_X\end{bmatrix}}N'\oplus X\xrightarrow{\begin{bmatrix}-m_L&f_0\end{bmatrix}}L
\quad(c_0\co M'\to L)
\]
becomes a conflation by the dual of Proposition~\ref{prop:pushout_inflation}. Thus,
\scriptsize
\[
H^0(\scE)(T,N')\oplus H^0(\scE)(T,X)\xrightarrow{\begin{bmatrix} (-m_L)\sas & (f_0)\sas \end{bmatrix}}
H^0(\scE)(T,L)\to\E(T,M')\xrightarrow{\begin{bmatrix} a'\sas \\ (m_X)\sas \end{bmatrix}}\E(T,N')\oplus\E(T,X)
\]
\normalsize
is exact for any $T\in\scT$. Since $\ovl{f_0}$ is a right $H^0(\scT)$-approximation, this shows that the sequence $0\to \E(T,M')\xrightarrow{a'\sas}\E(T,N')$ is exact.
 
Thus the composite of $(a', f_0, \id,  c_0, b_0, t_0)$ and $(\id, f_1, a, c_1, b_1, t_1)$ gives a morphism of conflations $(a',f,a,c,b,t)$ with the required property.

\normalcolor

{\rm (2)} We prove that $\ovl{e_X}$ becomes a right $H^0(\scT)$-approximation, for the morphism $e_X$ obtained in the above construction. Let $T\in\ob(\scT)$ be any object.
From the morphism of conflations $(\id, f_1, a, b_1, c_1, t_1)$, we obtain a commutative diagram
\[
\begin{tikzcd}
H^0(\scE)(T,N')&H^0(\scE)(T,L) &H^0(\scE)(T,M)&\bbE(T,N') \\
H^0(\scE)(T,N')&H^0(\scE)(T,Y) &H^0(\scE)(T,N)&\bbE(T,N')
\Ar{1-1}{1-2}{"\ovl{m_L}\ci\blank"}
\Ar{1-2}{1-3}{"\ovl{e_L}\ci\blank"}
\Ar{1-3}{1-4}{}
\Ar{2-1}{2-2}{"\ovl{m_Y}\ci\blank"}
\Ar{2-2}{2-3}{"\ovl{e_Y}\ci\blank"}
\Ar{2-3}{2-4}{}
\Ar{1-1}{2-1}{equal}
\Ar{1-2}{2-2}{"\ovl{f_1}\ci\blank"}
\Ar{1-3}{2-3}{"\ovl{a}\ci\blank"}
\Ar{1-4}{2-4}{equal}
\end{tikzcd}
\]
whose rows are exact.
Since $\ovl{e_Y}\colon Y\to N$ is a right $H^0(\scT)$-approximation, the morphism $H^0(\scE)(T,Y)\xrightarrow{\ovl{e_Y}\ci\blank}H^0(\scE)(T,Y)$ is surjective. By the exactness, this implies that 
\[H^0(\scE)(T,L)\xrightarrow{\ovl{e_L}\ci\blank}H^0(\scE)(T,M)\]
is also surjective, which means that $\ovl{e_L}$ is a right $H^0(\scT)$-approximation. Since $\ovl{f_0}$ is a right $H^0(\scT)$-approximation, so is their composite $\ovl{e_X}=\ovl{e_L}\ci\ovl{f_0}$.  
\end{proof}

\begin{prop}\label{splicing_n_pb}
Let $(a^{i-1},f^i,a^i,c^{i-1},b^i,t^i)$ depicted as
\[
\begin{tikzcd}
M^{i-1} \ar{r}{m_X^{i-1}} \ar{d}[swap]{a^{i-1}} & X^{i} \ar{r}{e_X^{i}} \ar{d}{f^{i}} & M^{i} \ar{d}{a^{i}} \\
N^{i-1} \ar{r}[swap]{m_Y^i} & Y^{i} \ar{r}[swap]{e_Y^{i}} & N^{i}
\Ar{1-1}{2-2}{"c^{i-1}",dashed}
\Ar{1-2}{2-3}{"b^{i}",dashed}
\Ar{1-1}{1-3}{"h_X^{i-1}",bend left=30, dashed}
\Ar{2-1}{2-3}{"h_Y^{i-1}"',bend right=30, dashed}
\end{tikzcd}
\]
be morphisms of conflations in $\scE$ for $1\le i\le n$. Put $X^0=M^0$, $Y^0=N^0$, $X^{n+1}=M^{n+1}$, $Y^{n+1}=N^{n+1}$, and assume:
\begin{itemize}
  \item the morphisms $\ovl{e_X^i}$ and $\ovl{e_Y^i}$ are right $H^0(\scT)$-approximations for each $0\le i\le n-2$,
  \item $X^{i}, Y^{i}\in\scT$ for all $0\le i\le n+1$,
  \item $M^0=N^0$ and $\ovl{a^0}=\id$ in $H^0(\scT)$.
\end{itemize}

Let $\phi\up\in Z^0(\Cbf^{n+2}(\A))(X\up,Y\up)$ be the morphism obtained by Proposition~\ref{prop:splicing_morphism}, where $X\up$ and $Y\up$ are obtained from the rows by Proposition~\ref{prop:splicing1}. Then $\phi\up$ is an $n$-pullback morphism. In particular, we have $\ovl{\phi}^i=f^{i}$ for each $1\le i\le n$, $\ovl{\phi}^{n+1}=a^{n+1}$, and $\ovl{\phi}^0=a^0=\id$ in $H^0(\scE)$.
\end{prop}
\begin{proof}
By Proposition~\ref{prop:chara_conf}, we already know that $X\up$ and $Y\up$ are $n$-exact. By Lemma~\ref{lem:equivalent_n-pbm1}, it suffices to show that 
$\CoCone(\sig_{\ge 1}\phi\up[1])\cong\Cone(\sig_{\ge 1}\phi\up)\in\Kbf(\scT)$
is left $n$-exact in $\scT$. Moreover, by Proposition~\ref{prop:chara_conf}, it is enough to show that $\Cone(\sig_{\ge 1}\phi\up)$ is right $n$-exact in $\scT$ and that the leftmost differential of $\Cone(\sig_{\ge 1}\phi\up)$ is an inflation in $\scE$.

We first show that $\Cone(\sig_{\ge 1}\phi\up)\in \Cbf^{n+2}(\scT)$ is right $n$-exact in $\scT$. Since $\Cone(\sig_{\ge 1}\phi\up)$ is isomorphic to $\Cone(\phi\up)$ in $\Kbf(\scT)$ by Corollary~\ref{cor:devide_f}, it is enough to show 
\[
\Kbf(\scT)(\Cone(\phi\up),T[i-n])=0
\]
for every $T\in\scT$ and every $i\le n$. However, this follows from the following the existence of the distinguished triangle $X\up\xrightarrow{\phi\up} Y\up \to \Cone (\phi\up)\to X\up[1]$
in $\Kbf(\scT)$ and the fact that $X\up,Y\up\in\Cbf{\scT}$ are $n$-exact.

It remains to show that the leftmost differential of $\Cone(\sig_{\ge 1}\phi\up)$ is an inflation in $\scE$. More precisely, we show that the morphism
$$\begin{bmatrix}
f^1\\
-m_X^1\circ e_X^1
\end{bmatrix}\colon X^1\to Y^1\oplus X^2$$
is an inflation in $\scE$.
Note that we have
$$
\begin{bmatrix}
f^1 \\
-m_X^1\circ e_X^1
\end{bmatrix}=
\begin{bmatrix}
\id_{Y^1} & 0 \\
0 & m_X^1
\end{bmatrix}
\circ
\begin{bmatrix}
f^1 \\
-e_X^1
\end{bmatrix}.
$$
On the right-hand side, the morphism
$$
\begin{bmatrix}
\id_{Y^1} & 0 \\
0 & m_X^1
\end{bmatrix}\colon Y^1\oplus M^1\to Y^1\oplus X^2
$$
is an inflation in $\scE$ because $m^1_X$ is an inflation. Thus, it suffices to show that the morphism
$$
\begin{bmatrix}
f^1 \\
-e_X^1
\end{bmatrix}
\colon X^1\to Y^1\oplus M^1
$$
is an inflation in $\scE$. This follows from Proposition~\ref{prop:pushout_inflation}, since
\[
\begin{tikzcd}
X^0 \ar{r}{m_X^0} \ar[d,equals] & X^1 \ar{r}{e_X^1} \ar{d}{f^1} & M^1 \ar{d}{a^1} \\
Y^0 \ar{r}[swap]{m_Y^0} & Y^1 \ar{r}[swap]{e_Y^1} & M^1
\Ar{1-1}{2-2}{"c^0",dashed}
\Ar{1-2}{2-3}{"b^1",dashed}
\Ar{1-1}{1-3}{"h_X^1",bend left=30, dashed}
\Ar{2-1}{2-3}{"h_Y^1"',bend right=30, dashed}
\end{tikzcd}
\]
is a morphism of conflations.

Thus, we have checked that $\Cone(\sig_{\ge 1}\phi\up)$ is right $n$-exact in $\scT$ and that the leftmost differential of $\Cone(\sig_{\ge 1}\phi\up)$ is an inflation in $\scE$. This finishes the proof that $\phi\up$ is an $n$-pullback morphism.
\end{proof}

\begin{cor}\label{cor:nex2op}
Let $Y\up$ be an $n$-exact sequence in $\scT$ satisfying the equivalent conditions in Proposition~\ref{prop:chara_conf}. Then, for any morphism $c\colon M^n\to Y^{n+1}$ in $Z^0(\scT)$, there exists an $n$-pullback morphism $f\up\colon X\up\to Y\up$ such that $\ovl{f}^n=\ovl{c}$ in $H^0(\scT)$ and $X\up$ is an $n$-exact sequence in $\scT$ satisfying the equivalent conditions in Proposition~\ref{prop:chara_conf}.
\end{cor}
\begin{proof}
We may assume that $Y\up$ is obtained by splicing the conflations 
$$N^{i}\to Y^{i+1}\to N^{i+1} \quad (h_Y^i\colon N^i\to N^{i+1})$$
in $\scE$ for $0\le i\le n-1$ by Proposition~\ref{prop:splicing1}. By Proposition~\ref{prop:const_n_pb}, we can construct inductively morphisms of conflations $(a^i, f^{i+1}, a^{i+1}, c^i, b^{i+1}, t^i)$ depicted as
\begin{equation}\label{morph_conf_i}
\begin{tikzcd}
M^i \ar{r}{m_X^i} \ar{d}[swap]{a^i} & X^{i+1} \ar{r}{e_X^{i+1}} \ar{d}{f^{i+1}} & M^{i+1} \ar{d}{a^{i+1}} \\
N^i \ar{r}[swap]{m_Y^i} & Y^{i+1} \ar{r}[swap]{e_Y^{i+1}} & N^{i+1}
\Ar{1-1}{2-2}{"c^i",dashed}
\Ar{1-2}{2-3}{"b^{i+1}",dashed}
\Ar{1-1}{1-3}{"h_X^i",bend left=30, dashed}
\Ar{2-1}{2-3}{"h_Y^i"',bend right=30, dashed}
\end{tikzcd}
\end{equation}
for $1\le i\le n-1$.
Note that the morphisms $\ovl{e^i_X}$ and $\ovl{e^i_Y}$ are right $H^0(\scT)$-approximations for any $1\le i\le n-2$ by the same proposition.  

Consider the case $i=0$. Since we already have the following diagram:
$$\begin{tikzcd}
 {} & {} & M^1 \\
 N^0 & Y^1 & N^1
\Ar{2-1}{2-2}{"m^0_Y"}
\Ar{2-2}{2-3}{"e^1_Y"}
\Ar{1-3}{2-3}{"a^1"}
\end{tikzcd}
$$
we can take a $1$-pullback morphism that can be depicted as follows:
\[
\begin{tikzcd}
N^0 \ar{r}{m_X^0} \ar[d, equals] & X^1 \ar{r}{e_X^1} \ar{d}{f^1} & M^1 \ar{d}{a^1} \\
N^0 \ar{r}[swap]{m_Y^0} & Y^{1} \ar{r}[swap]{e_Y^{1}} & N^{1}
\Ar{1-1}{2-2}{"c^0",dashed}
\Ar{1-2}{2-3}{"b^1",dashed}
\Ar{1-1}{1-3}{"h_X^0",bend left=30, dashed}
\Ar{2-1}{2-3}{"h_Y^0"',bend right=30, dashed}
\end{tikzcd}
\]
To obtain the data required to apply Proposition~\ref{splicing_n_pb}, it remains to prove the following claim.
\begin{claim}\label{claim:LM}
The morphism $\ovl{e^1_X}\co X^1\to M^1$ is a right $H^0(\scT)$-approximation.
\end{claim}
\begin{proof}
First we show that $X^1$ belongs to $\scT$. We use an argument similar to that of Proposition~\ref{prop:n-exactness_from_splice}.
Take any object $T\in\scT$. It suffices to show $\E^i(T,X^1)=0$ for all $1\le i\le n-1$.
For any $1\le i\le n$, since $M^i\xrightarrow{m_X^i}X^{i+1}\xrightarrow{e_X^{i+1}}M^{i+1}$ is a conflation, the sequence
\begin{equation}\label{ex_seq_XMM}
\E^{j-1}(T,X^{i+1})\xrightarrow{(e_X^{i+1})\sas}\E^{j-1}(T,M^{i+1})\to\E^j(T,M^i)\to\E(T,X^{i+1})
\end{equation}
is exact for any $j\ge 1$. By $X^{i+1}\in\scT$, we obtain an isomorphism $\E^{j-1}(T,M^{i+1})\xrightarrow{\cong}\E^j(T,M^i)$ for $2\le j\le n-1$.
As for $j=1$, since $e_X^{i+1}$ is a $\scT$-approximation for $i\le n-2$, the exactness of $(\ref{ex_seq_XMM})$ gives $\E(T,M^i)=0$ for $1\le i\le n-2$.
Thus we obtain
\[
\E^i(T,M^1)\cong\E^{i-1}(T,M^2)\cong\cdots\cong\E(T,M^i)=0
\]
for $1\le i\le n-2$. Since $N^0\to X^1\to M^1$ is a conflation and $N^0\in\scT$, this shows $\E^i(T,X^1)=0$ for $1\le i\le n-2$.

It remains to show $\E^1(T,X^1)=0$. By the existence of morphisms of conflations $(\ref{morph_conf_i})$, we obtain a commutative diagram
\[
\begin{tikzcd}
\E(T,M^{n-1}) & \E^2(T,M^{n-2}) & \cdots &\E^{n-1}(T,M^1)\\
\E(T,N^{n-1}) & \E^2(T,N^{n-2}) & \cdots &\E^{n-1}(T,N^1)
\Ar{1-1}{1-2}{"\cong"}
\Ar{1-2}{1-3}{"\cong"}
\Ar{1-3}{1-4}{"\cong"}
\Ar{2-1}{2-2}{"\cong"}
\Ar{2-2}{2-3}{"\cong"}
\Ar{2-3}{2-4}{"\cong"}
\Ar{1-1}{2-1}{"a^{n-1}\sas"'}
\Ar{1-2}{2-2}{"a^{n-2}\sas"'}
\Ar{1-4}{2-4}{"a^{1}\sas"}
\end{tikzcd}
\]
in $\Ab$, in which the horizontal arrows are isomorphisms.
The leftmost vertical map $a^{n-1}\sas$ is injective by Proposition~\ref{prop:const_n_pb}, hence so is the rightmost $a^1\sas$.
Since 
\[
X^1\xrightarrow{\begin{bmatrix}-e_X^1\\ f^1\end{bmatrix}}M^1\oplus Y^1
\xrightarrow{\begin{bmatrix}a^1& e_Y^1\end{bmatrix}}N^1
\]
is a conflation and $Y^1\in\scT$, we obtain an exact sequence
\[
\E^{n-2}(T,M^1)\oplus\E^{n-2}(T,Y^1)
\xrightarrow{\begin{bmatrix}a^1\sas&(e_Y^1)\sas\end{bmatrix}}\E^{n-2}(T,N^1)\to\E^{n-1}(T,X^1)\to0
\]
by the injectivity of $a^1\sas\co\E^{n-1}(T,M^1)\to\E^{n-1}(T,N^1)$. Moreover, since $N^0\xrightarrow{m_Y}Y^1\xrightarrow{e_Y^1}N^1$ is a conflation and $N^0\in\scT$,
the map $(e_Y^1)\sas\co\E^{n-2}(T,Y^1)\to\E^{n-2}(T,N^1)$ is surjective. (In fact, we also have $\E^{n-2}(T,Y^1)=0$ if $n\ge3$.)
Thus the exactness of the above sequence shows $\E^{n-1}(T,X^1)=0$. 

We have shown $X^1\in\scT$.
It remains to show that the morphism $\ovl{e_X^1}\colon X^1\to M^1$ is a right $H^0(\scT)$-approximation. Consider the following conflation in $H^0(\scE)$:
$$N^0\to X^1\to M^1.$$
Since $N^0\in\scT$, the morphism $H^0(\scE)(-,X^1)\to H^0(\scE)(-,M^1)$ is surjective when restricted to $H^0(\scT)$. Thus, the morphism $e_X^1\colon X^1\to M^1$ is a right $H^0(\scT)$-approximation.
\end{proof}
\normalcolor

Thus Claim~\ref{claim:LM} is shown.
We have constructed the data in Proposition~\ref{splicing_n_pb}. Thus we have an $n$-pullback morphism $\phi\up\colon X\up\to Y\up$ which satisfies $\ovl{\phi}^{n+1}=\ovl{c}$ and $\ovl{\phi}^{0}=\id_{N^0}$ in $H^0(\scE)$ by Proposition~\ref{splicing_n_pb}.
\end{proof}

\begin{thm}\label{thm:n-CT_is_n-exact}
Let $\scE$ be an exact dg-category which is strictly connective and strictly additive, and let $\scT\subset\scE$ be an $n$-cluster tilting subcategory.
Let $\calS$ be the class of $n$-exact sequences in $\scT$ whose left and right end differentials are, respectively, an inflation and a deflation in $\scE$.
Then $(\scT,\calS)$ becomes an $n$-exact dg-category.
\end{thm}
\begin{proof}
First, we note that $\calS$ is closed under isomorphisms in $\Kbf^{1,n,1}(\scT)$. This follows immediately from condition~(WIC) on $\scE$. The other conditions in ($n$-Ex0) are also satisfied by the definition of $\calS$. Condition~($n$-Ex1) also follows from the definition of $\calS$, since inflations in $\scE$ correspond precisely to inflations in $\scT$. Condition~($n$-Ex2$\op$) follows immediately from Corollary~\ref{cor:nex2op}. Condition~($n$-Ex2) follows by a dual argument.
\end{proof}

\begin{cor}\label{cor:n_CT_of_stable}
If $\scE$ is a stable dg-category and $\scT\subset\scE$ is an $n$-cluster tilting subcategory, then $\scT$ becomes an $n$-stable dg-category.
\end{cor}
\begin{proof}
This follows immediately from Theorem~\ref{thm:n-CT_is_n-exact} and Proposition~\ref{prop:characterization_n-stable}.
\end{proof}

\begin{cor}\label{cor:nZ_CT}
Let $\scE$ be an exact dg-category, and let $\scT\subset\scE$ be an $n$-cluster tilting subcategory that satisfies $H^{i}(\scE)(T,T')=0$ for $-n<i<0$ for any $T,T'\in \scT$ (for example, if $n\mathbb{Z}$-cluster tilting subcategory). Then $(H^0(\scT),\E_{(\scT)},\sfr_{(\scT)})$ becomes an $n$-exangulated category, where $\E_{(\scT)}$ and $\sfr_{(\scT)}$ are as in Theorem~\ref{thm:H0_of_n-ex-dg}.
\end{cor}
\begin{proof}
If $\scT$ is an $n\mathbb{Z}$-cluster tilting subcategory of $\scE$, then the connective dg-category $\scT$ satisfies Condition~\ref{condition:H^i=0}. Thus, the statement follows from Theorem~\ref{thm:H0_of_n-ex-dg}.
\end{proof}

\subsection{More on $n\mathbb{Z}$-cluster tilting subcategories}

To explain the relationship between $n\mathbb{Z}$-cluster tilting subcategories and $(n+2)$-angulated categories, we recall the following from \cite{LZ}.

\begin{dfn}(\cite[Definition~3.2]{LZ})
Let $(\scF,\F,\sfr_{\F})$ be an $n$-exangulated category. We call $(\scF,\F,\sfr_{\F})$ a \emph{Frobenius $n$-exangulated category} if $\scF$ has enough projectives and enough injectives, and projective objects coincide with injective objects.
\end{dfn}

\begin{fact}\label{fact:frob_and_triang}(\cite[Theorem~3.13]{LZ})
Let $(\scF,\F,\sfr_{\F})$ be a Frobenius $n$-exangulated category. Then, the ideal quotient of $\scF$ by projective-injectives has a natural structure of an $(n+2)$-angulated category.
\end{fact}

As a special case, if projective objects are only zero objects in a Frobenius $n$-exangulated category $\scF$, then $\scF$ itself has a structure of $(n+2)$-angulated categories. The following corollary is an immediate consequence of this; we record it here for the reader's convenience.
\begin{cor}\label{chara_tri}
Let $(\scF,\F,\sfr_{\F})$ be an $n$-exangulated category such that $\scF$ is weakly idempotent complete. Then the following are equivalent.
\begin{enumerate}
\item $(\scF,\F,\sfr_{\F})$ is a Frobenius $n$-exangulated category and projective objects are only zero objects.
\item Any morphism $f\colon X\to Y$ in $\scF$ is inflation and deflation.
\end{enumerate}
\end{cor}
\begin{proof}
Assume {\rm (2)}. Since any morphism is an inflation and a deflation, for any $C\in \scF$, we can take a deflation $0\to C$. In this case, we can take an admissible $n$-exangle
$$A\to 0\to 0\to \cdots \to 0\to C\dashrightarrow$$
for some $A\in\scF$. 
This shows that $\scF$ has enough projectives. Dually, we can also show that $\scF$ has enough injectives. Next, we show that any projectives are zero. Let $P$ be a projective object in $\scF$. Since the morphism $P\to 0$ is a deflation, this morphism splits. Thus, $P$ is a direct summand of the zero object, and hence $P$ is a zero object. Dually, any injective object is also a zero object. Thus, {\rm (1)} holds.

Conversely, we assume {\rm (1)}. Let $f\colon X\to Y$ be any morphism in $\scF$. Since $X\to 0$ is an inflation, we get the inflation $\begin{bmatrix}f\\0\end{bmatrix}\colon X\to Y\oplus 0$. This morphism is isomorphic to $f\colon X\to Y$ and hence $f$ is an inflation. Dually, $f$ is also a deflation. Thus, {\rm (2)} holds.
\end{proof}

We can characterize $n$-stable dg-categories in terms of Frobenius $n$-exangulated categories as follows.

\begin{prop}\label{prop:chara_stable_2}
Let $\scT$ be an $n$-exact dg-category and assume that $\scT$ has $n$-sparse cohomologies. Let $(H^0(\scT),\E_{(\scT)},\sfr_{(\scT)})$ be the $n$-exangulated category obtained from Theorem~\ref{thm:H0_of_n-ex-dg}. Assume that $H^0(\scT)$ is weakly idempotent complete.
Then the following are equivalent.
\begin{enumerate}
\item $(H^0(\scT),\E_{(\scT)},\sfr_{(\scT)})$ is a Frobenius $n$-exangulated category and projective objects are only zero objects,
\item $\scT$ is an $n$-stable dg-category.
\end{enumerate}
\end{prop}
\begin{proof}
It immediately follows from Proposition~\ref{prop:characterization_n-stable},  Lemma~\ref{lem:nea-1} and Corollary~\ref{chara_tri}.
\end{proof}

\begin{cor}
Let $\scE$ be a stable dg-category, and let $\scT\subset\scE$ be an $n\mathbb{Z}$-cluster tilting subcategory. Then $\scT$ becomes an $n$-stable dg-category and its homotopy category $H^0(\scT)$ has a structure of an $(n+2)$-angulated category.
\end{cor}
\begin{proof}
Since $\scT$ is an $n\mathbb{Z}$-cluster tilting subcategory of stable dg-category $\scE$, it becomes an $n$-stable dg-category by Corollary~\ref{cor:n_CT_of_stable} and $n$-sparse. Thus, the second statement follows from Proposition~\ref{prop:chara_stable_2} and Fact~\ref{fact:frob_and_triang}.
\end{proof}

Finally, we will explain the relationship between pre-$(n+2)$-angulated categories defined in \cite{JKM} and $n$-stable categories. The following notion is defined in \cite{JKM}.
\begin{dfn}\label{dfn:pretriangulated-dg}(\cite[Definition~3.2.1.]{JKM})
A dg-category $\A$ is \emph{pre-$(n+2)$-angulated} if the canonical inclusion $\A\to\perdg\A$ defines an $n\Z$-cluster tilting subcategory $H^0(\A)\subset \per\A$ where $\perdg\A$ is a natural dg-enhancement of perfect derived category of $\per\A$.
\end{dfn}

\begin{rem}
Let $\A$ be a pre-$(n+2)$-angulated dg-category. Then, by definition, $\tau_{\leq 0}\A$ becomes a strictly connective and strictly additive dg-category. It is because $\perdg\A$ is a strictly additive.
\end{rem}

\begin{prop}\label{prop:pre-n-angulated_truncates_to_n-stable}
Let $\A$ be a pre-$(n+2)$-angulated dg-category. Then, its truncation $\tau_{\le 0}\A$ is an $n$-stable dg category.
\end{prop}
\begin{proof}
Since $\tau_{\leq 0}(\perdg\A)$ is a stable dg-category by \cite[Example~6.2.]{C2} and $H^0(\tau_{\leq 0}\A)=H^0(\A)$ is an $n\Z$-cluster tilting subcategory of $\per\A$, it follows from Corollary~\ref{cor:n_CT_of_stable}.
\end{proof}

\end{document}